\documentclass{amsart}

\usepackage{url}

\input{macros.sty}

\author{Rebecca Bellovin}
\title{$p$-adic Hodge theory in rigid analytic families}
\address{
Department of Mathematics,
Evans Hall,
University of California,
Berkeley, CA 94720
}\email{rmb@math.berkeley.edu}

\begin{document}

\begin{abstract}
We study the functors $\D_{\B_\ast}(V)$, where $\B_\ast$ is one of Fontaine's period rings and $V$ is a family of Galois representations with coefficients in an affinoid algebra $A$.  We first relate them to $(\varphi,\Gamma)$-modules, showing that $\D_{\HT}(V)=\oplus_{i\in\Z}\left(\D_{\Sen}(V)\cdot t^i\right)^{\Gamma_K}$, $\D_{\dR}(V)=\D_{\dif}(V)^{\Gamma_K}$, and $\D_{\cris}(V)=\D_{\rig}(V)[1/t]^{\Gamma_K}$; this generalizes results of Sen, Fontaine, and Berger.  We then deduce that the modules $\D_{\HT}(V)$ and $\D_{\dR}(V)$ are coherent sheaves on $\Sp(A)$, and $\Sp(A)$ is stratified by the ranks of submodules $\D_{\HT}^{[a,b]}(V)$ and $\D_{\dR}^{[a,b]}(V)$ of ``periods with Hodge--Tate weights in the interval $[a,b]$''.  Finally, we construct functorial $\B_\ast$-admissible loci in $\Sp(A)$, generalizing a result of Berger-Colmez to the case where $A$ is not necessarily reduced.
\end{abstract}

\maketitle

\tableofcontents

\setlength{\parskip}{1ex} 
\setlength{\parindent}{0ex}

\section{Introduction}

\subsection{Background}

In this article, we study rigid analytic families of representations of $\Gal_K$, where $K$ is a finite extension of $\Q_p$ and $\Gal_K:=\Gal(\overline{K}/K)$ is its absolute Galois group.  More precisely, we consider vector bundles $\mathscr{V}$ over a rigid analytic space $X$ over $\Q_p$ equipped with a continuous $\mathscr{O}_X$-linear action of $\Gal_K$.  Thus, if we specialize $\mathscr{V}$ at any closed point of $X$, we get a representation of $\Gal_K$ on a finite-dimensional $\Q_p$-vector space.  Families of Galois representations arise, for example, on the generic fibers of Galois deformation rings, as in~\cite{kisin}.  Such families of Galois representations also arise from families of $p$-adic modular forms.

The study of $p$-adic representations of $p$-adic Galois groups is quite technical, so we put off precise definitions to the body of this paper and give an overview here.  Given a finite-dimensional $\Q_p$-vector space $V$ equipped with a continuous $\Q_p$-linear action of $\Gal_K$, one can capture the information of $V$ in terms of a semilinear Frobenius $\varphi$ and a semilinear action of a one-dimensional $p$-adic Lie group, at the expense of making the coefficients more complicated.  More precisely, work of Fontaine and many others defines equivalences of categories between the category $\Rep_{E}(\Gal_K)$ of finite-dimensional $E$-linear representations of $\Gal_K$, where $E$ is some finite-dimensional $\Q_p$-algebra, and various kinds of \'etale $(\varphi,\Gamma)$-modules (see e.g. \cite{wintenberger}, \cite{cc}).

The same theory lets us sort $p$-adic Galois representations based on how ``nice'' or arithmetically significant they are.
One accomplishes this by defining certain ``period rings'' $\mathbf{B}_\ast$, such as $\mathbf{B}_{\HT}$, $\mathbf{B}_{\dR}$, $\mathbf{B}_{\st}$, and $\mathbf{B}_{\cris}$, which are equipped with Galois actions and ``linear algebra structures'', and defining $\D_{\mathbf{B}_\ast}(V):=(\mathbf{B}_\ast\otimes_{\Q_p}V)^{\Gal_K}$.  We say that $V$ is $\mathbf{B}_\ast$-admissible (or, for the specific examples of $\mathbf{B}_\ast$ listed above, ``Hodge--Tate'', ``de Rham'', ``semi-stable'', or ``crystalline'') if the $\Q_p$-dimension of $V$ is the same as the $\mathbf{B}_\ast^{\Gal_K}$-dimension of $\D_{\B_\ast}(V)$ (as part of the definition of a period ring, $\mathbf{B}_\ast^{\Gal_K}$ is required to be a field). 

In the paper~\cite{bc}, Berger and Colmez associate to a rank-$d$ Galois representation $V$ with coefficients in a Banach algebra $A$ a family of $(\varphi,\Gamma)$-modules $\D^\dagger(V)$, under the supplementary hypothesis that $V$ admits a Galois-stable integral lattice.  As an application, they show that if $A$ is an affinoid algebra, then the locus of closed points $x\in\Sp(A)$ where the specialization $V_x$ is $\B_\ast$-admissible with Hodge--Tate weights in a fixed interval $[a,b]$ is a closed analytic set, and if $A$ is reduced and $V_x$ is $\B_\ast$-admissible for every $x\in\Sp(A)$, then $\D_{\B_\ast}(V):=\left((A\widehat\otimes \B_\ast)\otimes_AV\right)^{\Gal_K}$ is a locally free $A\otimes_{\Q_p}\B_\ast^{\Gal_K}$-module of rank $d$.

In this paper, we make a closer study of the functors $\D_{\B_\ast}(V)$ for $\ast\in\{\HT,\dR,\st,\cris\}$, where $V$ is a finite projective $A$-module of rank $d$ equipped with a continuous $A$-linear action of $\Gal_K$, for some affinoid algebra $A$.  We actually treat vector bundles over rigid analytic spaces in Sections~\ref{db-functors} and \ref{b-adm-loci}, but we state our results with affinoid coefficients here.

The first theorem we prove relates $\D_{\B_\ast}(V)$ to families of $(\varphi,\Gamma)$-modules.  The modules $\D_{\Sen}(V)$, $\D_{\dif}(V)$, $\D_{\rig}^\dagger(V)$, and $\D_{\log}^\dagger(V)$ are defined in section~\ref{tate-sen}.
\begin{thm}\label{db-functors-thm}
Let $A$ and $V$ be as above.  Then 
\begin{enumerate}
\item	$\D_{\HT}^K(V)=\oplus_{i\in\Z}\left(\D_{\Sen}^K(V)\cdot t^i\right)^{\Gamma_K}$
as submodules of $(A\widehat\otimes\B_{\HT})\otimes_{A}V$,
\item	$\D_{\dR}^K(V)=\left(\D_{\dif}^K(V)\right)^{\Gamma_K}$, as submodules of $(A\widehat\otimes\B_{\dR})\otimes_{A}V$, and
\item	$\D_{\cris}^K(V)=\left(\D_{\rig,K}^\dagger(V)[1/t]\right)^{\Gamma_K}$, and $\D_{\st}^K(V)=\left(\D_{\log,K}^\dagger(V)[1/t]\right)^{\Gamma_K}$.  The first equality is as submodules of $(A\widehat\otimes\widetilde{\B}_{\rig}^\dagger)\otimes_{A}V$, and the second is as submodules of $(A\widehat\otimes\widetilde{\B}_{\log}^\dagger)\otimes_{A}V$
\end{enumerate}
\end{thm}

\begin{remark}
The first two parts of Theorem~\ref{db-functors-thm} are used in the proofs of \cite[Th\'eor\`eme 5.1.4]{bc} and \cite[Th\'eor\`eme 5.3.2]{bc}, respectively.  We are not aware of proofs of these facts in the literature when $A$ is not $\Q_p$-finite, so for the convenience of the reader we provide proofs for general $\Q_p$-affinoid algebras.
\end{remark}

We can then deduce that $\D_{\B_\ast}(V)$ is a finite $A\otimes_{\Q_p}\B_\ast^{\Gal_K}$-module, and that the formation of $\D_{\HT}(V)$ and $\D_{\dR}(V)$ commutes with flat base change on $A$.  In particular, $\D_{\HT}(V)$ and $\D_{\dR}(V)$ are coherent sheaves on the rigid analytic space $\Sp(A)$.  Together with Theorem~\ref{db-functors-thm}, this is used in~\cite{diao-liu} to prove properness of the eigencurve over weight space.  We further conjecture that the formation of $\D_{\st}(V)$ and $\D_{\cris}(V)$ also commutes with flat base change.

The key to our base change theorems is that we can express $\D_{\HT}(V)$ and $\D_{\dR}(V)$ as cohomology groups of a complex which has finite cohomology.  We do not know how to do the same for $\D_{\st}(V)$ and $\D_{\cris}(V)$.  However, the cohomological finiteness theorem of~\cite{kpx} implies that if $K/\Q_p$ is finite, then for any $\alpha\in A^\times$, the formation of $\D_{\cris}(V)^{\varphi=\alpha}$ commutes with flat base change on $A$ (\cite[Theorem 4.4.3(2)]{kpx}).  Cohomological finiteness similarly underlies the results of~\cite{liu-fsf} on interpolation of semi-stable periods.

We then have a pair of theorems about the $\B_\ast$-admissible loci in $\Sp(A)$, generalizing the results of~\cite{bc} to a base that is not necessarily reduced:

\begin{thm}\label{adm-locus-thm}
Let $A$ and $V$ be as above, and let $\ast\in\{\HT,\dR,\st,\cris\}$.  Then there is a quotient $A\twoheadrightarrow A_{\B_\ast}^{[a,b]}$ such that for any $\Q_p$-finite algebra $B$, a map $A\rightarrow B$ factors through $A_{\B_\ast}^{[a,b]}$ if and only if the induced $\Q_p$-finite $B$-linear Galois representation $V_B:=V\otimes_AB$ is $\B_\ast$-admissible with Hodge--Tate weights in the interval $[a,b]$.
\end{thm}

\begin{thm}\label{pointwise-adm-thm}
Let $A$ and $V$ be as above, and let $\ast\in\{\HT,\dR,\st,\cris\}$.  Suppose that $V_B$ is $\B_\ast$-admissible with Hodge--Tate weights in the interval $[a,b]$ for every homomorphism $A\rightarrow B$, where $B$ is an $\Q_p$-finite algebra.  Then
\begin{enumerate}
\item\label{adm-proj}	the module $\D_{\B_\ast}^K(V)$ is a projective $A\otimes_{\Q_p}\B_\ast^{\Gal_K}$-module of rank $d$, 
\item\label{adm-comp-isom}	the natural homomorphism $(A\otimes_{\Q_p}\B_\ast)\otimes_{A\otimes \B_\ast^{\Gal_K}}\D_{\B_\ast}^K(V)\rightarrow (A\otimes_{\Q_p}\B_\ast)\otimes_AV$ is an isomorphism, and
\item\label{adm-base-change}	the formation of $\D_{\B_\ast}^K(V)$ commutes with arbitrary $\Q_p$-affinoid base change on $A$.
\end{enumerate}
\end{thm}
In fact, assuming part \ref{adm-proj}, parts \ref{adm-comp-isom} and \ref{adm-base-change} are equivalent.  We do not know whether part \ref{adm-proj} implies \ref{adm-comp-isom} and \ref{adm-base-change}.

For $\ast\in\{\HT,\dR\}$, we actually prove a more general pair of theorems.  Namely, we let $\D_{\B_\ast}^{[a,b]}(V)$ be the module of ``$\B_\ast$-admissible periods with Hodge--Tate weights in the interval $[a,b]$'' (this is precisely defined in Section~\ref{b-adm-loci}), and we show that $\Sp(A)$ is stratified by the rank of the fibral modules $\D_{\B_\ast}^{[a,b]}(V_x)$.
\begin{thm}\label{dht-ddr-periods-thm}
Let $A$ and $V$ be as above, let $X=\Sp(A)$, and let $\ast\in\{\HT,\dR\}$.  There is a Zariski-open subspace $X_{\B_\ast,\leq d'}^{[a,b]}\subset X$ and a Zariski-closed subspace $X_{\B_\ast,\geq d'}^{[a,b]}\hookrightarrow X$ such that $x:\Sp(B)\rightarrow X$ factors through $X_{\B_\ast,d'}^{[a,b]}:=X_{\B_\ast,\leq d'}^{[a,b]}\cap X_{\B_\ast,\geq d'}^{[a,b]}$ if and only if $\D_{\B_\ast}^{[a,b]}(V_x)$ is a free $B\otimes_{\Q_p}K$-module of rank $d'$, where $B$ is an $\Q_p$-finite artin local ring and $V_B:=V\otimes_AB$.
\end{thm}
The $X_{\B_\ast,d'}^{[a,b]}$ give a stratification of $X$, in the sense that $X_{\B_\ast,\leq d'-1}^{[a,b]}=X_{\B_\ast,\leq d'}^{[a,b]}\smallsetminus X_{\B_\ast,d'}^{[a,b]}$ and $X=X_{\B_\ast,\leq d}^{[a,b]}$.

\begin{thm}\label{pointwise-dht-ddr-periods-thm}
Let $X$ and $V$ be as above, and let $\ast\in\{\HT,\dR\}$. Suppose that for every $\Q_p$-finite artinian point $x:A\rightarrow B$, the $B\otimes_{\Q_p}K$-module $\D_{\B_\ast}^{[a,b]}(V_x)$ is free of rank $d'$, where $0\leq d'\leq d$.  Then
\begin{enumerate}
\item	$\D_{\B_\ast}^{[a,b]}(V)$ is a locally free $A\otimes_{\Q_p}K$-module of rank $d'$, and each $\left(t^k\cdot\D_{\Sen}^{K_n}(V)\right)^{\Gamma_K}$ is a locally free $A\otimes_{\Q_p} K$-module, 
\item	the formation of $\D_{\B_\ast}^{[a,b]}(V)$ commutes with arbitrary $\Q_p$-affinoid base change $A\rightarrow A'$
\end{enumerate}
If $d'=d$, then
\begin{enumerate}
\item[(3)]	$\D_{\B_\ast}(V)=\D_{\B_\ast}^{[a,b]}(V)$,
\item[(4)]	the natural morphism 
$$\alpha_{V}:(A\widehat\otimes\B_\ast)\otimes_{A\otimes_{\Q_p}K}\D_{\B_\ast}^K(V)\rightarrow (A\widehat\otimes\B_\ast)\otimes_{A}V$$
is an isomorphism.
\end{enumerate}
\end{thm}

To prove Theorems~\ref{dht-ddr-periods-thm} and \ref{pointwise-dht-ddr-periods-thm}, we use Theorem~\ref{db-functors-thm} and Pottharst's theory of Galois cohomology with affinoid coefficients~\cite{pottharst}.  This approach makes transparent the role of boundedness of Hodge--Tate weights in the results of~\cite{bc}: boundedness of Hodge--Tate weights is equivalent to finiteness of a certain Galois cohomology group, and finiteness of cohomology groups is the essential ingredient in cohomology and base change results.

\begin{remark}
Shah~\cite{shrenik} has obtained similar results on the behavior of $\D_{\HT}(V)$ and $\D_{\dR}(V)$ when $A$ is reduced by studying the Galois cohomology of $\B_{\dR}^+$ directly.
\end{remark}

The proofs of Theorems~\ref{adm-locus-thm} and \ref{pointwise-adm-thm} when $\ast\in\{\st,\cris\}$ are quite different.  In the latter case, we largely follow the strategy of~\cite{bc}; the new ingredient that permits us to handle non-reduced coefficients is Lemma~\ref{embed}, which gives a generalization of the Shilov boundary points of a reduced Berkovich space.  This permits us to prove that ``de Rham implies uniformly potentially semi-stable'' when the coefficient ring is non-reduced; when the coefficients are reduced, this is~\cite[Th\'eor\`eme 6.3.2]{bc}.

Our results about the behavior of $\D_{\cris}(V)$ and $\D_{\st}(V)$ under base-change are quite limited except in the case where $V$ is crystalline or semi-stable, when the formation of $\D_{\cris}(V)$ and $\D_{\st}(V)$ commutes with arbitrary base change.  This is because the continuous $\Gamma$-cohomology of a $(\varphi,\Gamma)$-module is not finite, nor is there an apparent subquotient which does have finite Galois cohomology.  In addition, individual de Rham periods are not necessarily potentially semi-stable, so we are unable to extend our present technique.

We remark further that all of our results are limited to the case where $K/\Q_p$ is a finite extension.  This is primarily because overconvergence of families of Galois representations is only known in this case.  However, our use of cohomology and base change in the study of $\D_{\HT}(V)$ and $\D_{\dR}(V)$ means we would be restricted to the case where $K$ is discretely valued in any case.  We similarly have no access to information about the behavior of $\D_{\B_\ast}(V)$ under general analytic field extension on the coefficients.

\subsection{Structure of this paper}

Throughout this paper, we consider representations of $\Gal_K$, where $K/\Q_p$ is finite, on vector bundles over $\Q_p$-rigid analytic spaces, where $E/\Q_p$ is finite.

In section~\ref{background-rigid}, we review some of rigid analytic geometry that we will need.  The rigid analytic geometry is primarily standard.  However, we prove that an affinoid algebra over a discretely valued field can be embedded in a finite product of artin rings which are module-finite over a complete discretely valued field with perfect residue field.  This generalizes the fact that a reduced affinoid algebra can be embedded in the product of the residue fields at the points of its Shilov boundary, and we expect it to be of independent interest.  We then briefly recall the theory of families of $(\varphi,\Gamma)$-modules attached to families of Galois representations, and the subsequent construction of $\D_{\Sen}(V)$ and $\D_{\dif}(V)$.  We give a criterion for a coherent sheaf over a quasi-Stein space to have finitely generated global sections, and we deduce that families of $(\varphi,\Gamma)$-modules over the Robba ring have finitely generated global sections.

In section~\ref{cohomology}, we review Pottharst's results on Galois cohomology with affinoid coefficients.  We generalize some of his results to modules $M$ which are finite flat over $A[\![t]\!]$ and equipped with a continuous $A[\![t]\!]$-semilinear action of a profinite group $G$, where the action of $G$ on $A[\![t]\!]$ is trivial on $A$ and preserves the $t$-adic filtration.  We then show that the inverse system $\{\H^0(G,M/t^k)\}_{k\geq 0}$ is eventually constant, so satisfies Mittag-Leffler.

In section~\ref{db-functors}, we prove Theorem~\ref{db-functors-thm}.  This generalizes results of Sen, Fontaine, and Berger.  We deduce that each $\D_{\B_\ast}(V)$ is $A$-finite, and that the formation of $\D_{\HT}(V)$ and $\D_{\dR}(V)$ commutes with flat base change on $\Sp(A)$.  We conjecture that the formation of $\D_{\st}(V)$ and $\D_{\cris}(V)$ also commutes with flat base change on $\Sp(A)$.

In section~\ref{b-adm-loci}, we prove Theorems~\ref{adm-locus-thm} and \ref{pointwise-adm-thm}.  We first study the behavior of $\D_{\HT}(V)$ and $\D_{\dR}(V)$ under non-flat base changes.  We use Pottharst's~\cite{pottharst} theory of cohomology and base change for Galois cohomology with affinoid coefficients to prove Theorems~\ref{dht-ddr-periods-thm} and \ref{pointwise-dht-ddr-periods-thm}, giving us functorial Hodge--Tate and de Rham loci.

We can then use the existence of a functorial de Rham locus to construct functorial semi-stable and crystalline loci, following the argument of~\cite{bc}.  More precisely, we prove that ``de Rham implies uniformly potentially semi-stable'', using Lemma~\ref{embed} to generalize a similar result of~\cite{bc} when the base is reduced.  The locus of points where the Galois representation is semi-stable with Hodge--Tate weights in the interval $[a,b]$ is a union of connected components of the locus ``de Rham with Hodge--Tate weights in the interval $[a,b]$''.  We can then cut out the crystalline locus as the subspace where the monodromy operator $N$ vanishes.

The appendix contains results on sheaves of period rings.  We need to work with sheaves of various rings of $p$-adic Hodge theory, which requires us to be careful about the topologies on these rings.  In Appendix~\ref{rings}, we describe some of these rings and indicate how to sheafify them.

\subsection*{Notation and conventions}

Throughout this paper, $K$ is a finite extensions of $\Q_p$.  The rings of $p$-adic Hodge theory are as defined in~\cite{berger}.  We let $\chi$ denote the $p$-adic cyclotomic character.  Our Hodge--Tate weights are normalized so that $\chi$ has Hodge--Tate weight $-1$.  All of our $p$-adic Hodge theoretic rings are as in~\cite{berger}; in particular, we use $\B_{\max}$ instead of $\B_{\cris}$ to define the functor $\D_{\cris}$, and we let $\B_{\st}^+$ and $\B_{\st}$ denote $\B_{\max}^+[\log([\overline\pi])]$ and $\B_{\max}[\log([\overline\pi])]$, respectively.  

\subsection*{Acknowledgments}

The results in this paper form a portion of my Ph.D. thesis, and I would first like to thank my advisor, Brian Conrad, for suggesting the direction of this research and for numerous conversations, as well as for helpful comments on earlier drafts of this paper.  I am indebted to Laurent Berger and Pierre Colmez for their mathematical contributions, and I am grateful to Kiran Kedlaya and Jay Pottharst for helpful conversations, as well as Toby Gee for helpful comments on earlier drafts.  This work was partially supported by an NSF Graduate Research Fellowship.

\section{Preliminaries}\label{prelim}

We will review some of the rigid geometry we will need, before recalling the theory of families of $(\varphi,\Gamma)$-modules attached to families of Galois representations.

\subsection{Rigid geometry}\label{background-rigid}

We will use the language of classical rigid spaces.  The standard reference for such spaces, and the rings of restricted power series which underlie them, is~\cite{bgr}.  The goal of the theory is to provide a robust theory of analytic spaces over non-archimedean fields, mirroring the theory of complex analytic spaces over $\C$.  We will also assume that the reader is familiar with Raynaud's theory of formal models of quasi-compact quasi-separated rigid analytic spaces, as treated in~\cite{bosch-lut}. 

\subsubsection{Affinoid algebras}

Let $k$ be a field complete with respect to a non-archimedean valuation $|\cdot|$, let $R$ denote its valuation ring $\{x\in k | |x|\leq 1\}$, and let $\mathfrak{m}$ denote the maximal ideal of $R$, which consists of elements with absolute value strictly less than $1$.  Let $T_n(k)$ (or $T_n$, if the ground field is clear) denote the $n$-variable Tate algebra over $k$.

If $k$ is discretely valued, the value group of the norm on a $k$-affinoid algebra is discrete, and it is often possible to reduce questions about affinoid algebras to questions about discretely valued fields.  For example, if $A$ is reduced, we have the following result due to Berkovich:
\begin{prop}[{\cite[Cor. 2.4.5]{berk}},{\cite[Cor. 2.1.4]{bc}}]
If $A$ is a reduced $k$-affinoid algebra with $k$ discretely valued, there exist complete discretely valued fields $B_1,\ldots,B_m$ such that there is a closed embedding $A\hookrightarrow \prod_{i=1}^mB_i$.
\end{prop}

We will need the following strengthening of this result, in which we drop the reducedness hypothesis:
\begin{lemma}\label{embed}
Let $A$ be a $k$-affinoid algebra, where $k$ is discretely valued.  Then there is a closed embedding $A\rightarrow \prod_i R_i$ into a product of a finite collection of artinian $k$-Banach algebras $R_i$, each module-finite over a complete discretely valued field $B_i$ over $k$ (as valued fields) with perfect residue field (and each artin ring $R_i$ is topologized as a finite-dimensional vector space over $B_i$).
\end{lemma}
\begin{remark}
Kedlaya and Liu claim a similar result in~\cite[Lemma 6.4]{kl}, but we do not understand their argument.
\end{remark}
\begin{proof}
First of all, we note that this is true if we take $A$ to be a Tate algebra $T_n$.  For then we may embed $T_n$ into $Q(T_n)^\wedge$, the completion of its quotient field for the multiplicative Gauss norm.  This field will not have perfect residue field (unless $n=0$), but by \cite[Theorem 29.1]{matsumura} applied to $R\langle X_1,\ldots,X_n\rangle_{(\pi)}^\wedge$, where $\pi$ is a uniformizer of $R$, we can find a complete discretely valued extension $B/Q(T_n)^\wedge$ with perfect residue field and $e(B|Q(T_n)^\wedge)=1$.

Now let $A$ be a general $k$-affinoid algebra.  Since $A$ is noetherian, we can find a minimal primary decomposition $(0)=\mathfrak{q}_1\cap\cdots \cap \mathfrak{q}_r$.  This yields a module-finite injective (hence closed) map $A\rightarrow \prod_j A/\mathfrak{q}_j$, so if we can embed each $A/\mathfrak{q}_i$, we can embed $A$.

Since each $\mathfrak{q}_j$ is a primary ideal, every zero-divisor in $A/\mathfrak{q}_j$ is nilpotent.  Thus, we may replace $A$ with $A/\mathfrak{q}_j$ and assume that all zero-divisors are nilpotent.

By Noether normalization, we can find an integral (and finite) monomorphism $T_n\rightarrow A$ for some $n$.  Since $A\otimes_{T_n}Q(T_n)^\wedge$ is a module-finite algebra over $Q(T_n)^\wedge$, it decomposes as the product of finitely many artin local rings finite over $Q(T_n)^\wedge$.  

We claim that the natural map $A\rightarrow A\otimes_{T_n}Q(T_n)^\wedge$ is an injection.  Since the natural map
$$A\otimes_{T_n}Q(T_n)\rightarrow A\otimes_{T_n}Q(T_n)^\wedge$$ 
is injective, it is enough to show that $A\rightarrow A\otimes_{T_n}Q(T_n)$ is injective.  If a non-zero $r\in A$ dies in $A\otimes_{T_n}Q(T_n)$, there is some non-zero $\omega\in T_n$ such that $r\omega=0$ in $A$.  Then since all zero-divisors are nilpotent, $\omega$ lands in the nilradical of $A$, so some power of $\omega$ is zero in $A$.  But $T_n\rightarrow A$ is injective and $T_n$ is reduced, so we have a contradiction.

Moreover, we claim that the natural Banach topology on $A$ agrees with its subspace topology from the finite dimensional $Q(T_n)^\wedge$-vector space $A\otimes Q(T_n)^\wedge$.  To see this, we first note that the natural topology on $A$ as an affinoid algebra is the same as the topology on $A$ as a $T_n$-module.  When $A=T_n[x]/(f(x))$ for $f$ a monic polynomial over $T_n$, $A$ is free over $T_n$, and so clearly $A\rightarrow A\otimes_{T_n}Q(T_n)^\wedge$ is a closed embedding.

Now consider the general case.  For any $r\in A$, when viewing $r$ in $A\otimes_{T_n}Q(T_n)$, its minimal polynomial over $Q(T_n)$ is a monic polynomial $f(x)$ with coefficients in $T_n$, because $T_n$ is integrally closed in $Q(T_n)$. Moreover, $T_n[r]\cong T_n[x]/\ann_{T_n[x]}(r)$, where $\ann_{T_n[x]}(r)$ is the annhilator of $r$.  Because $A$ is torsion-free as a $T_n$-module (as we saw above), 
$$\ann_{T_n[x]}(r)=T_n[x]\cap f\cdot Q(T_n)[x]=f\cdot T_n[x]$$ since $f$ is monic.  Thus, $T_n[r]=T_n[x]/(f(x))$.  Therefore, the subring $T_n[r]\subset A$ is finite free as a $T_n$-module, so it is a closed $T_n$-submodule of 
$$(Q(T_n)^\wedge)[r]\subset A\otimes_{T_n}Q(T_n)^\wedge$$

We will show that $A\rightarrow A\otimes_{T_n}Q(T_n)^\wedge$ is a closed embedding by considering the collection of $T_n$-submodules $A'\subset A$ which are closed in $A\otimes_{T_n}Q(T_n)^\wedge$.  If $A'\subset A$ is a $T_n$-submodule, then $A'\otimes_{T_n}Q(T_n)^\wedge\rightarrow A\otimes_{T_n}Q(T_n)^\wedge$ is an injection, because $T_n\rightarrow Q(T_n)^\wedge$ is flat.  We begin with $A'=T_n$.  If $A'=A$, we are done.  If not, choose $s\in A-A'$, so $T_n[s]\subset A$ is a finite free $T_n$-submodule.  We claim that $A'+T_n[s]$ is closed in $A\otimes_{T_n}Q(T_n)^\wedge$.  Then we can replace $A'$ with $A'+T_n[s]$ and repeat the process.  Since $A$ is finite over $T_n$, this process terminates eventually at $A$, so we will be done.

Thus, it suffices to show that if we have a finite-dimensional $Q(T_n)^\wedge$-vector space $V$ (equipped with its natural topology) and two closed finite $T_n$-submodules $F$ and $F'$, with $F'$ free and $(F+F')\otimes_{T_n}Q(T_n)^\wedge\rightarrow V$ an injection, then $F+F'$ is also closed in $V$.  We may assume by induction on the rank of $F'$ that $F'$ is free of rank one.  We may also replace $V$ with $(F+F')\otimes_{T_n}Q(T_n)^\wedge$, since all subspaces of a finite-dimensional vector space over a non-archimedean field (such as $Q(T_n)^\wedge$) are closed.  If $F\cap F'=\{0\}$, it is clear that $F+F'\xleftarrow{\sim} F\oplus F'$ is closed, because $V\xleftarrow{\sim} (F\otimes_{T_n}Q(T_n)^\wedge)\oplus (F'\otimes_{T_n}Q(T_n)^\wedge)$ as $Q(T_n)^\wedge$-vector spaces, and the topology on the right side is the product topology (and $F$ is closed in $F\otimes_{T_n}Q(T_n)^\wedge$ due to the closedness of $F$ in $V$).  If the intersection is nonzero, there is some $\omega\in T_n-\{0\}$ such that $F'\subset \frac{1}{\omega}F$, so $F+F'\subset \frac{1}{\omega}F$.  Since $\frac{1}{\omega}F$ is closed in $V$, it only remains to show that $F+F'$ is closed in $\frac{1}{\omega}F$.  But submodules of finite modules over affinoid algebras (equipped with their natural topologies) are always closed, by~\cite[Prop. 3.7.3/1]{bgr}, so we are done.

The upshot of this is that we have a closed embedding $A\rightarrow \prod_i R_i$, where the $R_i$ are a finite collection of $Q(T_{n_i})^\wedge$-finite artin rings (equipped with the topologies of finite-dimensional $Q(T_{n_i})^\wedge$-vector spaces).  Finally, we replace $R_i$ with $R_i\otimes_{Q(T_n)^\wedge}B_i$.
\end{proof}

\subsubsection{Quasi-Stein spaces}

\begin{definition}
A rigid analytic space $Y$ over $k$ is said to be \emph{quasi-Stein} if it admits an admissible covering by a rising union of affinoid subdomains $Y_0\subset Y_1\subset\cdots$ such that the transition maps $\Gamma(Y_{n+1},\mathscr{O}_{Y_{n+1}})\rightarrow \Gamma(Y_n,\mathscr{O}_{Y_n})$ are flat with dense image.
\end{definition}
In particular, $A_\infty:=\Gamma(Y,\mathscr{O}_Y)=\varprojlim_n\Gamma(Y_n,\mathscr{O}_{Y_n})$ is a Fr\'echet-Stein algebra.  

By Kiehl's results on coherent sheaves on rigid analytic spaces, a coherent sheaf $\mathscr{F}$ on $Y$ is simply a compatible system of coherent sheaves $\{\mathscr{F}_n\}$ on $\{Y_n\}$.  Then $F_\infty:=\Gamma(Y,\mathscr{F})=\varprojlim_n\Gamma(Y_n,\mathscr{F}_n)$ is a coadmissible module over $A_\infty$, in the sense of~\cite{sch-teit}.

\begin{example}
Fix $s>0$, and let $Y$ be the coordinate on the closed unit disk.  Then the half-open annulus $0<v_p(X)\leq 1/s$ is a quasi-Stein space, as it is the rising union of the closed annuli $1/s'\leq v_p(X)\leq 1/s$ as $s'\rightarrow\infty$.
\end{example}

Quasi-Stein spaces behave much as affine schemes do in algebraic geometry.  In particular, Kiehl proved the following theorem on the cohomology of coherent sheaves on quasi-Stein spaces (which also follows from \cite[Theorem 3]{sch-teit}):
\begin{thm}[{\cite[Satz 2.4]{kiehl}}]\label{kiehl}
Let $Y$ be a quasi-Stein space, and let $\mathscr{F}$ be a coherent sheaf on $Y$.  Then 
\begin{enumerate}
\item	$\H^i(Y,\mathscr{F})=0$ for $i>0$,
\item	the image of $\mathscr{F}(Y)$ in $\mathscr{F}(Y_n)$ is dense for all $n$
\end{enumerate}
\end{thm}

There is no {\it a priori} reason for $F_\infty$ to be a finite module over $A_\infty$.  For example, let $Y_n=\Sp(\prod_{i=0}^n\Q_p(\zeta_{p^i}))$, and let $\mathscr{F}_n$ be the sheaf on $Y_n$ associated to $\prod_{i=0}^n\Q_p(\zeta_{p^i})^{\oplus i}$.  Then $F_\infty$ is not $A_\infty$-finite, because the fiber of $F_\infty$ at $\Sp(\Q_p(\zeta_{p^n}))$ is a $\Q_p(\zeta_{p^n})$-vector space of dimension $n$.  Happily, this is the only thing that can go wrong.

\begin{lemma}\label{fibral-ranks}
Let $\mathscr{F}$ be a coherent sheaf over a finite-dimensional quasi-Stein space $Y$ over $\Q_p$.  Then $\H^0(Y,\mathscr{F})$ is finitely generated as an $\H^0(Y,\mathscr{O}_Y)$-module if and only if there is some integer $d$ such that $\dim_{\kappa(y)}\mathscr{F}(y)\leq d$ for all $y\in Y$.
\end{lemma}
\begin{proof}
Necessity is clear.  To prove sufficiency, we proceed by induction on the dimension of $Y$.  If $Y$ is a zero-dimensional Stein space, the result is clear.

Now suppose we have the desired result when $\dim Y<n$, which is to say when every irreducible component of $Y$ has dimension at most $n-1$ for some $n\geq 1$, and suppose $\dim Y=n$.  Choose $i:Y'\hookrightarrow Y$ consisting of a point on every irreducible component of $Y$.  By the settled zero-dimensional case, $i^\ast \mathscr{F}$ is finitely generated over $\mathscr{O}_{Y'}$, so by Theorem~\ref{kiehl} there is some coherent $\mathscr{O}_Y$-submodule $\mathscr{F}'\subset \mathscr{F}$ on $Y$ such that $i^\ast \mathscr{F}'\twoheadrightarrow i^\ast \mathscr{F}$.  Thus, we have an exact sequence of coherent $\mathscr{O}_Y$-modules
$$0\rightarrow \mathscr{F}'\rightarrow \mathscr{F}\rightarrow \mathscr{G}\rightarrow 0$$
such that $i^\ast \mathscr{G}=0$.  Then $\mathscr{G}$ vanishes on a Zariski-open subspace of $Y$ containing $Y'$; its complement is a quasi-Stein space $Z$ with structure sheaf $\mathscr{O}_Y/\ann_{\mathscr{O}_X}\mathscr{G}$, all of whose irreducible components have dimension at most $n-1$, since $Y'$ intersects each irreducible component of $Y$.  Then we may apply our inductive hypothesis to $\mathscr{G}$ ($\mathscr{F}(y)\twoheadrightarrow \mathscr{G}(y)$, so the fibral ranks of $\mathscr{G}$ are bounded).  Therefore, $\mathscr{G}(Y)$ is $\mathscr{O}_Y(Y)$-finite.  Since $\mathscr{F}'$ is $\mathscr{O}_Y$-finite by construction, $\mathscr{F}$ is as well, since $\H^1(Y,\mathscr{F}')=0$.
\end{proof}

\begin{cor}\label{flat-fibral-ranks}
Suppose that $\mathscr{F}$ is a flat coherent sheaf over $\mathscr{O}_Y$, where $Y$ is a finite-dimensional quasi-Stein space.  Then $\H^0(Y,\mathscr{F})$ is projective of rank $d$ over $\H^0(Y,\mathscr{O}_Y)$ if and only if $\dim_{\kappa(y)}\mathscr{F}(y)= d$ for all $y\in Y$.
\end{cor}
\begin{proof}
As flat finitely presented modules are finite locally free, and hence projective, it is enough to prove that $\mathscr{F}$ is finitely presented over $\mathscr{O}_Y$.  Lemma~\ref{fibral-ranks} implies that there is a surjection $\mathscr{O}_Y^{\oplus m}\twoheadrightarrow \mathscr{F}$, and we will apply Lemma~\ref{fibral-ranks} to the kernel $\mathscr{G}$.  To do this, we need to know that $\dim_{\kappa(y)}\mathscr{G}(y)$ is bounded over all $y\in Y$.  But if we specialize the short exact sequence
$$0\rightarrow \mathscr{G}\rightarrow \mathscr{O}_X^{\oplus m}\rightarrow \mathscr{F}\rightarrow 0$$
at $y\in Y$, we get a short exact sequence
$$0\rightarrow \mathscr{G}(y)\rightarrow \kappa(y)^{\oplus m}\rightarrow \mathscr{F}(y)\rightarrow 0$$
because $\mathscr{F}$ was assumed flat.  Therefore, $\dim_{\kappa(y)}\mathscr{G}(y)\leq m$ and we are done.
\end{proof}

\subsection{Families of $(\varphi,\Gamma)$-modules}\label{tate-sen}

We briefly recall the construction of families of $(\varphi,\Gamma)$-modules associated to families of Galois representations.  Let $A$ be an $E$-affinoid algebra, and let $A\widehat\otimes\B_{(\rig),K}^{\dagger,(s)}$ denote one of the rings $A\widehat\otimes \B_K^\dagger$, $A\widehat\otimes \B_K^{\dagger,s}$, $A\widehat\otimes \B_{\rig,K}^{\dagger,s}$, or $A\widehat\otimes \B_{\rig, K}^{\dagger}$.  Similarly, let $A\widehat\otimes\widetilde{\B}^{\dagger,(s)}$ denote one of the rings $A\widehat\otimes\widetilde{\B}^{\dagger}$ or $A\widehat\otimes\widetilde{\B}^{\dagger,s}$.

Throughout this subsection, let $s_0=(p-1)/p$ and let $s_n=p^ns_0=p^{n-1}(p-1)$.  For $L$ a $p$-adic field, let $L_n=L(\zeta_{p^n})$.

\begin{definition}
A $\varphi$-module over $A\widehat\otimes\B_{(\rig),K}^{\dagger,(s)}$ is a finitely presented projective module $\D^{(s)}$  over $A\widehat\otimes\B_{(\rig),K}^{\dagger,(s)}$ together with a map $\varphi:\D^{(s)}\rightarrow \D^{(ps)}$ which is semilinear over $\varphi:\B_{(\rig),K}^{\dagger,(s)}\rightarrow\B_{(\rig),K}^{\dagger,(ps)}$, such that the linearization $\varphi':\B_{(\rig),K}^{\dagger,(ps)} {}_{\varphi}{\otimes}_{\B_{(\rig),K}^{\dagger,(s)}}\D^{(s)}\rightarrow \D^{(ps)}$ is an isomorphism.  A $(\varphi,\Gamma)$-module over $A$ is a $\varphi$-module over $A$ together with a continuous $A$-linear action of $\Gamma_K$ which is semilinear over the action of $\Gamma_K$ on $\B_{(\rig),K}^{\dagger,(s)}$ and commutes with $\varphi$.
\end{definition}


\begin{remark}
A $\varphi$-module $\D$ over $A\widehat\otimes\B_{(\rig),K}^{\dagger,s}$ is in particular a finite $A\widehat\otimes\B_{(\rig),K}^{\dagger,s}$-module.  It is therefore a finite module over either a Banach algebra or a Fr\'echet-Stein algebra.  It follows that $\D$ has a unique structure as a Fr\'echet $A\widehat\otimes\B_{(\rig),K}^{\dagger,s}$-module.  Thus, we may speak unambiguously of the continuity of any action of $\Gamma_K$.
\end{remark}

\begin{remark}
In~\cite{kl}, the authors define a family of $(\varphi,\Gamma)$-modules over $A\widehat\otimes \B_{\rig,K}^{\dagger,s}$, for $s\gg0$, to be a coherent locally free sheaf over the product of the half-open annulus $0<v_p(X)\leq 1/s$ with $\Sp(A)$ in the category of rigid analytic spaces.  By Lemma~\ref{fibral-ranks} and Corollary~\ref{flat-fibral-ranks}, this is equivalent to the definition we have given.  This equivalence is also proven in~\cite[Proposition 2.2.7]{kpx}, where the authors use the $\varphi$-module structure on a family of $(\varphi,\Gamma)$-modules to prove finite generation of its global sections.
\end{remark}

The main source of $(\varphi,\Gamma)$-modules is Galois representations; to any family of $p$-adic Galois representations, we can functorially associate a family of $(\varphi,\Gamma)$-modules, and this functor is fully faithful.

\begin{definition}
Let $X$ be a rigid analytic space over $E$.  A family of Galois representations over $X$ is a locally free $\mathscr{O}_X$-module $\mathscr{V}$ of rank $d$ together with an $\mathscr{O}_X$-linear action of $\Gal_K$ which acts continuously on $\Gamma(U,\mathscr{V})$ for every admissible affinoid open $U\subset X$.
\end{definition}

\begin{remark}
It is enough to check continuity on a single admissible affinoid cover $\{U_i\}$ of $X$.  For if $U_i=\Sp(A_i)$ is affinoid and $\Gal_K$ acts continuously on $\mathscr{V}(U_i)$, then $\Gal_K$ certainly acts continuously on $\mathscr{V}(W) = \mathscr{V}(U_i)\otimes_{A_i}\mathscr{O}_X(W)$ for any affinoid subdomain $W\subset U_i$.

On the other hand, suppose that $\{U_i=\Sp(A_i)\}$ is an admissible affinoid covering of $U=\Sp(A)$, and suppose that $\Gal_K$ acts continuously on $\mathscr{V}(U_i)$.  Since
$$0\rightarrow A\rightarrow \prod_i A_i\rightarrow \prod_{i,j}A_i\widehat\otimes_AA_j$$
is exact, $\mathscr{V}(U)$ inherits its topology from its embedding in $\prod_i\mathscr{V}(U_i)$, and $\GL(\mathscr{V}(U))$ inherits its topology from its embedding in $\prod_i\GL(\mathscr{V}(U_i))$.  Therefore, $\Gal_K$ acts continuously on $\mathscr{V}(U)$.
\end{remark}

Then we have the following theorem:
\begin{thm}[{\cite{bc}}]\label{phi-gamma-bc}
Let $\mathscr{A}$ be a formal $\mathscr{O}_E$-model for $A$, and let $V$ be a free $A$-module of rank $d$ equipped with a continuous $A$-linear action of $\Gal_K$.  Suppose that $V$ contains a free $\Gal_K$-stable $\mathscr{A}$-submodule $V_0$ of rank $d$.  Then for $s\gg0$, there is a $\varphi$- and $\Gal_K$-stable $A\widehat\otimes\B_K^{\dagger,s}$-submodule (compatible with change in $s$)
$$\D_K^{\dagger,s}(V)\subset \left((A\widehat\otimes\widetilde\B^{\dagger,s})\otimes_{\Q_p}V\right)^{H_K}$$
which is a locally free $A\widehat\otimes\B_K^{\dagger,s}$-module of constant rank $d$ such that the natural map
$$(A\widehat\otimes\widetilde\B^{\dagger,s})\otimes_{A\widehat\otimes\B_K^{\dagger,s}}\D_K^{\dagger,s}(V)\rightarrow (A\widehat\otimes\widetilde\B^{\dagger,s})\otimes_AV$$
is an isomorphism (by ``$\varphi$-stable'' we mean that $\varphi(\D_K^{\dagger,s}(V))\subset \D_K^{\dagger,ps}(V)$).  If $\Gal_K$ acts trivially on $V_0/12pV_0$, then $\D_K^{\dagger,s}(V)$ is $A\widehat\otimes\B_K^{\dagger,s}$-free of rank $d$.

The formation of $\D_K^{\dagger,s}(V)$ is compatible with base change in $\mathscr{A}$.
\end{thm}
The base change property is not stated in \cite{bc}, but follows easily from the construction.

\begin{remark}
The construction of families of $(\varphi,\Gamma)$-modules given in \cite[Proposition 4.2.8]{bc} and \cite[Th\'eor\`eme 4.2.9]{bc} in fact only requires the coefficients to be a Banach algebra, not an affinoid algebra.
\end{remark}

If $V$ admits a $\Gal_K$-stable locally free $\mathscr{A}$-submodule $V_0$ of rank $d$, we may construct $\D_K^{\dagger,s}(V)$ by working on a cover $\{\Spf\mathscr{A}_i\}$ of $\Spf\mathscr{A}$ trivializing $V_0$.  Since we know that the formation of $\D_K^{\dagger,s}(V)$ is functorial in maps $\mathscr{A}_i\rightarrow\mathscr{A}_i\widehat\otimes_{\mathscr{O}_E}\mathscr{A}_j$, we can glue the $\D_K^{\dagger,s}(V|_{\mathscr{A}_i[1/p]})$ to get a sheaf of $A\widehat\otimes\B_K^{\dagger,s}$-modules on $\Sp(A)$.  By~\cite[Proposition 3.10]{kl}, there is a finite locally free $A\widehat\otimes\B_K^{\dagger,s}$-module $\D_K^{\dagger,s}(V)$ which induces this sheaf.

By~\cite[Lemme 3.18]{chenevier}, for any family of Galois representations $\mathscr{V}$ over a quasi-compact quasi-separated rigid analytic space $X$, there is a formal model $\mathscr{X}$ of $X$ such that $\mathscr{V}$ admits a Galois-stable $\mathscr{O}_{\mathscr{X}}$-lattice.  In fact, $\D_K^{\dagger,s}(V)$ is independent of the formal model $\mathscr{A}$:
\begin{prop}
Let $A$ and $V$ be as above.  Then $\D_K^{\dagger,s}(V)$ is independent of $\mathscr{A}$.
\end{prop}
\begin{proof}
It suffices to check independence of the integral model for an admissible formal blowing up $\mathscr{X}'\rightarrow\Spf\mathscr{A}$ with center $\mathscr{I}=(f_0,\ldots f_m)$.  More precisely, if $V$ admits both a Galois-stable $\mathscr{A}$-lattice and a Galois-stable $\mathscr{A}'$-lattice, then $\Spf\mathscr{A}$ and $\Spf\mathscr{A}'$ have a common admissible blow-up $\mathscr{X}$, so it suffices to check that $\D_K^{\dagger,s}(V)$ yields the same result on the generic fibers of $\mathscr{X}$ and $\Spf\mathscr{A}$.  

Temporarily let $\D_{K,\mathscr{X}}^{\dagger,s}(V)$ denote the construction using the integral structure $\mathscr{X}$ and $\D_{K,\mathscr{A}}^{\dagger,s}(V)$ denote the construction using the integral structure $\mathscr{A}$.  Now $\mathscr{X}$ admits a covering by the formal schemes
$$\mathscr{X}_i:=\Spf\mathscr{A}\left\langle \frac{f_0}{f_i},\ldots,\frac{f_m}{f_i}\right\rangle$$
and the morphism $\mathscr{X}_i\rightarrow \Spf\mathscr{A}$ is induced by $\mathscr{A}\rightarrow \mathscr{A}\langle \frac{f_0}{f_i},\ldots,\frac{f_m}{f_i}\rangle$.  In other words,
$$\D_{K,\mathscr{X}}^{\dagger,s}(V)|_{\Sp(A\langle \frac{f_0}{f_i},\ldots,\frac{f_m}{f_i}\rangle)} = A\left\langle \frac{f_0}{f_i},\ldots,\frac{f_m}{f_i}\right\rangle\widehat\otimes_A\D_{K,\mathscr{A}}^{\dagger,s}(V)$$
It follows that $\D_{K,\mathscr{X}}^{\dagger,s}(V)=\D_{K,\mathscr{A}}^{\dagger,s}(V)$.
\end{proof}

\begin{cor}
The formation of $\D_{K}^{\dagger,s}(V)$ commutes with arbitrary base change on $A$.
\end{cor}
\begin{proof}
Let $A\rightarrow A'$ be a homomorphism of $E$-affinoid algebras, and let $X=\Sp(A)$ and $X'=\Sp(A')$.  We may choose an admissible formal $\mathscr{O}_E$-model $\mathscr{X}_1$ of $X$ such that the family of Galois representations on $X$ extends to a family of Galois representions $V_0$ over $\mathscr{X}_1$.  By Theorem~\cite[Theorem 4.1]{bosch-lut}, we can find a formal model $\mathscr{X}_2$ of $X'$ together with an admissible formal blow-up $\psi:\mathscr{X}_2\rightarrow\mathscr{X}'$ and a morphism $\psi:\mathscr{X}_2\rightarrow \mathscr{X}_1$ which induces $f$ on the generic fiber.  Thus, functoriality of $\D_{K}^{\dagger,s}(V)$ follows from functoriality in the integral model.
\end{proof}

Furthermore, it is straightforward to check the following functorial properties of the assignment $V\rightsquigarrow\D_K^{\dagger,s}(V)$:
\begin{prop}
Let $A$ be an $E$-affinoid algebra, and let $V$ and $V'$ be families of $\Gal_K$-representations over $A$ as above.  Then for $s\gg0$,
\begin{enumerate}
\item	$\D_{K}^{\dagger,s}(V\oplus V') =\D_{K}^{\dagger,s}(V)\oplus \D_{X,K}^{\dagger,s}(V')$
\item	$\D_{K}^{\dagger,s}(V\otimes_{A}V') =\D_{K}^{\dagger,s}(V)\otimes_{A\widehat\otimes\B_{K}^{\dagger,s}} \D_{K}^{\dagger,s}(V')$
\item\label{d-dag-hom}	$\D_{K}^{\dagger,s}(\Hom_{A}(V,V')) = \Hom_{A\widehat\otimes\B_{K}^{\dagger,s}}(\D_{K}^{\dagger,s}(V),\D_{K}^{\dagger,s}(V))$
\end{enumerate}
\end{prop}
In particular, the third part implies that the assignment $V\rightsquigarrow \D_K^{\dagger,s}(V)$ is a fully faithful functor.  We omit the details; they are written out in~\cite[\textsection 4.3]{bellovin}.

Combined with various refinements of \cite{kl} and \cite{liu}, we have the following corollary:
\begin{cor}
Let $X$ be a quasi-compact quasi-separated rigid analytic space over $E$, and let $\mathscr{V}$ be a rank-$d$ family of $\Gal_K$-representations over $X$.  
  Then for $s\gg0$, there is family of $(\varphi,\Gamma)$-modules $\mathscr{D}_{K,(\rig)}^{\dagger,(s)}(\mathscr{V})$ which has rank $d$ over $\mathscr{O}_X\widehat\otimes\B_{K,(\rig)}^{\dagger,(s)}$ such that the natural map
$$(\mathscr{O}_X\widehat\otimes\widetilde\B_{(\rig)}^{\dagger,(s)})\otimes_{\mathscr{O}_X\widehat\otimes\B_{K,(\rig)}^{\dagger,(s)}}\mathscr{D}_{K,(\rig)}^{\dagger,(s)}(\mathscr{V})\rightarrow (\mathscr{O}_X\widehat\otimes\widetilde\B_{(\rig)}^{\dagger,(s)})\otimes_A\mathscr{V}$$
is an isomorphism.

The formation of $\mathscr{D}_{K,(\rig)}^{\dagger,(s)}(\mathscr{V})$ is compatible with base change in $X$, and the assignment $\mathscr{V}\rightsquigarrow\mathscr{D}_{K,(\rig)}^{\dagger,(s)}(\mathscr{V})$ is a fully faithful functor compatible with direct sums, duals, and tensor products.
\end{cor}

\begin{remark}
We do not know whether there is an intrinsic characterization of $\D_K^{\dagger,s}(V)$ as a submodule of $(A\widehat\otimes\widetilde{\B}^{\dagger,s})\otimes_AV$, or an intrinsic characterization of $\D_{K,\rig}^{\dagger,s}(V)$ as a subsheaf of $\widetilde{\B}_{\rig}^{\dagger,s}\otimes_{A}V$.
\end{remark}

We pause to briefly discuss the objects we have constructed.  For simplicity, we temporarily assume that $A=\Q_p$.  Given a Galois representation $V$ of dimension $d$, we have constructed a module over $\B_{\rig,K}^{\dagger}$ of rank $d$, equipped with a semilinear Frobenius and a semilinear action of $\Gamma_K$.  There is some $s$ so that these structures descend to $\B_{\rig,K}^{\dagger,s}$, which is (non-canonically) the ring of analytic functions on the half-open annulus $0<v_p(X)\leq 1/e_Ks$; we think of $p^{-1/e_Ks(V)}$ as the minimal inner radius of an annulus to which everything descends.

Consider the analytic function $\log(1+X)\in\B_{\rig,K}^{\dagger,s}$.  It has infinitely many zeroes, at the points $X=1-\zeta_{p^n}$, which accumulate towards the boundary of the unit disk.  For a given $s$, we think of $n(s)$ as the minimal $n$ so that $X=1-\zeta_{p^n}$ lies in the annulus $0<v_p(X)\leq 1/e_Ks$.

Returning to our general setup, we use $(\varphi,\Gamma)$-modules to construct modules $\D_{\Sen}(V)$ and $\D_{\dif}(V)$, which we will use to study Hodge--Tate and de Rham representations.

Recall that there is a family of injections $i_n:\B_K^{\dagger,s}\rightarrow K_n[\![t]\!]$ for every $n\geq n(s)$, which extend to injections $i_n:\B_{\rig,K}^{\dagger,s}\rightarrow K_n[\![t]\!]$.  It is defined as the composition
$$\B_K^{\dagger,s_n}\subset \widetilde{\B}^{\dagger,s_n}\xrightarrow{\varphi^{-n}}\widetilde{\B}^{\dagger,s_0}\rightarrow\B_{\dR}^+$$
where the last map sends $\sum p^k[x_k]$ (viewed as an element of $\widetilde{\B}^+$) to its image in $\B_{\dR}^+$, and it factors through $K_n[\![t]\!]$.

\begin{definition}
Let $X$ be a quasi-compact quasi-separated rigid analytic space and let $\mathscr{V}$ be a locally free $\mathscr{O}_X$-module of rank $d$ equipped with a continuous $\mathscr{O}_X$-linear action of $\Gal_K$.  Then by the preceding discussion, there is a finite extension $L/K$ such that $\mathscr{D}_{\rig,L}^{\dagger,s}(\mathscr{V})$ is $X$-locally free.
\begin{enumerate}
\item	For any $n\geq n(s)$, we put $\mathscr{D}_{\Sen}^{L_n}(\mathscr{V}):=\mathscr{D}_{L}^{\dagger,s}(\mathscr{V})\otimes_{\mathscr{B}_{L}^{\dagger,s}}^{i_n}(\mathscr{O}_X\otimes_{\Q_p}L_n)$.  Then $\mathscr{D}_{\Sen}^{L_n}(\mathscr{V})$ is an $X$-locally free $\mathscr{O}_X\otimes L_n$-module of rank $d$ with a linear action of $\Gamma_{L_n}$.
\item	For any $n\geq n(s)$, we put $\mathscr{D}_{\dif}^{L_n,+}(\mathscr{V}):=\mathscr{D}_{L}^{\dagger,s}(\mathscr{V})\otimes_{\mathscr{B}_{L}^{\dagger,s}}^{i_n}(\mathscr{O}_X\widehat\otimes_{\Q_p}L_n[\![t]\!])$, and we define $\mathscr{D}_{\dif}^{L_n}(\mathscr{V}):=\mathscr{D}_{\dif}^{L_n,+}(\mathscr{V})[1/t]$.  Then $\mathscr{D}_{\dif}^{L_n,+}(\mathscr{V})$ is an $X$-locally free $\mathscr{O}_X\widehat\otimes_{\Q_p} L_n[\![t]\!]$-module of rank $d$ with a continuous semi-linear action of $\Gamma_{L_n}$, where $L_n[\![t]\!]$ is equipped with its natural Fr\'echet topology (i.e. as the inverse limit $\varprojlim_kL_n[t]/t^k$ of finite-dimensional $\Q_p$-vector spaces).  Here $\Gamma_{L_n}$ acts trivially on $L_n$, but acts on $t$ via $\gamma\cdot t=\chi(\gamma)t$.
\end{enumerate}
\end{definition}

\begin{remark}
Both $\mathscr{D}_{\Sen}^{L_n}(\mathscr{V})$ and $\mathscr{D}_{\dif}^{L_n,+}(\mathscr{V})$ actually have semi-linear actions of all of $\Gal_K$, ultimately by $\Gal_K$-stability of $\D_{L,n}^{\dagger,s_0}(V)$ inside $(A\widehat\otimes\widetilde{\B}^{\dagger,s_0})\otimes_AV$.  We define $\mathscr{D}_{\Sen}^{K_n}(\mathscr{V}):=\mathscr{D}_{\Sen}^{L_n}(\mathscr{V})^{H_K}$ and $\mathscr{D}_{\dif}^{K_n,+}(\mathscr{V}):=\mathscr{D}_{\dif}^{L_n,+}(\mathscr{V})^{H_K}$.
\end{remark}

\begin{remark}
If $A$ is a general $\Q_p$-Banach algebra with valuation ring $\mathscr{A}$, $V_0$ is a free $\mathscr{A}$-module of rank $d$ equipped with a continuous $\mathscr{A}$-linear action of $\Gal_K$, and $V:=V_0[1/p]$, then we may similarly define $\D_{\Sen}^{L_n}(V):=\D_{L}^{\dagger,s}(V)\otimes_{A\widehat\otimes\B_{L}^{\dagger,s}}^{i_n}(A\otimes_{\Q_p}L_n)$ and $\D_{\dif}^{L_n}(V):=\D_{L}^{\dagger,s}(V)\otimes_{A\widehat\otimes\B_{L}^{\dagger,s}}^{i_n}(A\widehat\otimes_{\Q_p}L_n[\![t]\!])$.
\end{remark}

\begin{remark}
It is also possible to construct $\D_{\Sen}^{L_n}(V)$ directly by means of Tate-Sen theory applied to semi-linear representations of $\Gal_K$ on finite $X$-locally free $\mathscr{O}_X\widehat\otimes\C_K$-modules.  In particular, there is a constant $c_3$ (fixed at the outset such that $1/(p-1)<c_3<\frac{1}{2}\ord_p(12p)$) such that $\D_{\Sen}^{L_n}(V)$ admits a \emph{$c_3$-fixed basis}, i.e., there is a basis $\{\ve{e}_1,\ldots,\ve{e}_d\}$ such that if $U_\gamma$ is the matrix of the action of a topological generator $\gamma$ of $\Gamma_n$, then every entry of $U_\gamma-\rm{Id}$ has $p$-adic valuation greater than $c_3$.
We exploit this point of view in the proof of Theorem~\ref{dsen-dht}.
\end{remark}

\begin{prop}
\begin{enumerate}
\item	$\mathscr{D}_{\Sen}^{L_n}(\mathscr{V})$ is an $X$-locally free $\mathscr{O}_X\otimes L_n$-module of rank $d$, and we have a Galois-equivariant isomorphism
$$\C_K\widehat\otimes_{L_n}\mathscr{D}_{\Sen}^{L_n}(\mathscr{V})\rightarrow \C_K\widehat\otimes_{\Q_p}V$$
\item	$\mathscr{D}_{\dif}^{L_n,+}(\mathscr{V})$ is an $X$-locally free $\mathscr{O}_X\widehat\otimes L_n[\![t]\!]$-module of rank $d$, and we have a Galois-equivariant isomorphism
$$(\mathscr{O}_X\widehat\otimes\B_{\dR}^+)\otimes_{\mathscr{O}_X\widehat\otimes L_n[\![t]\!]}\mathscr{D}_{\dif}^{L_n,+}(\mathscr{V})\rightarrow (\mathscr{O}_X\widehat\otimes\B_{\dR}^+)\otimes_{\mathscr{O}_X}\mathscr{V}$$
which respects the filtrations on each side.
\end{enumerate}
\end{prop}
\begin{proof}
For both of these, the starting point is the isomorphism
$$\widetilde{\mathscr{B}}^{\dagger,s}\otimes_{\mathscr{B}_{L}^{\dagger,s}}\mathscr{D}_{L}^{\dagger,s}(\mathscr{V})\rightarrow \widetilde{\mathscr{B}}^{\dagger,s}\otimes_{\mathscr{O}_X}\mathscr{V}$$
The composition 
$$\B_L^{\dagger,s}\xrightarrow{i_n}L_n[\![t]\!]\rightarrow \B_{\dR}^+$$ is the same as the composition
$$\B_L^{\dagger,s}\subset \widetilde{\B}^{\dagger,s}\xrightarrow{\varphi^{-n}}\widetilde{\B}^{\dagger,p^{-n}s}\rightarrow\B_{\dR}^+$$
by definition, so extending scalars on each side from $\widetilde{\mathscr{B}}^{\dagger,s}$ to $\mathscr{B}_{X,\dR}^+$ or $\mathscr{O}_X\widehat\otimes_{\Q_p}\C_K$ gives the desired result.
\end{proof}

\section{Cohomology of procyclic groups}\label{cohomology}

\subsection{Overview}

Let $G$ be a profinite group with finite $p$-cohomological dimension $e$, such that $\H^i(G,T)$ has finite cardinality for all finite $p$-torsion discrete $G$-modules $T$.  Let $M$ be a topological abelian group equipped with a continuous action of $G$.  We consider the continuous cochain complex $C^\bullet(G,M)$ and its cohomology groups $\H^\bullet(G,M)$.  Specifically, we define the $n$-cochains $C^n(G,M)$ to be the set of continuous functions 
\[	f:G^n\rightarrow M	\]
and we define the differential $d^n:C^n(G,M)\rightarrow C^{n+1}(G,M)$ by 
\begin{equation*}\begin{split}d^n(f)(g_1,\ldots,g_{n+1}):=g_1\cdot &f(g_2,\ldots,g_{n+1}) \\
&+\sum_{i=1}^n(-1)^if(g_1,\ldots,g_{i-1},g_ig_{i+1},g_{i+2},\ldots,g_{n+1}) \\
&+ (-1)^{n+1}f(g_1,\ldots,g_n)\end{split}\end{equation*}
Thus, we get a complex
\[	C^\bullet(G,M):0\rightarrow M=C^0(G,M)\rightarrow C^1(G,M)\rightarrow\cdots	\]
and we define $\H^n(G,M):=\ker d_n/\im d_{n-1}$.  If $M=\varinjlim_{i\in I}M_i$ is the filtered colimit of topological abelian groups equipped with continuous actions of $G$ (compatible with the transition maps), we define $\H^n(G,M):=\varinjlim_{i\in I}\H^n(G,M_i)$.

Suppose now that $M$ is actually a $\Q_p$-Banach space and the action of $G$ on $M$ is $\Q_p$-linear.  Exact sequences of $\Q_p$-Banach spaces of are $\Q_p$-linearly split, so a $G$-equivariant exact sequence of $\Q_p$-Banach spaces
\[	0\rightarrow M'\rightarrow M\rightarrow M''\rightarrow 0	\]
admits a continuous $\Q_p$-linear section $M''\rightarrow M$.  Therefore, there is a long exact sequence in cohomology
\[	0\rightarrow \H^0(G,M')\rightarrow \H^0(G,M)\rightarrow \H^0(G,M'')\rightarrow \H^1(G,M')\rightarrow\cdots	\]

If $A$ is an $E$-affinoid algebra and $M$ is a finite flat $A$-module and the action of $G$ is $A$-linear, Pottharst has shown that the cohomology groups $\H^i(G,M)$ satisfy a number of good properties.  In particular, $\H^i(G,M)$ is a finite $A$-module for all $i\geq 0$, by~\cite[Theorem 1.2]{pottharst}, and $\H^i(G,M)=0$ whenever $i>e$, by~\cite[Proposition 1.1]{pottharst}. 

Crucially, the finiteness of the cohomology groups $\H^i(G,M)$ makes it possible to deduce the following ``cohomology and base change'' result:
\begin{thm}[{\cite[Theorem 1.4]{pottharst}}]\label{base-change}
Let $M$ be a finite flat $A$-module.  Then if $A'$ is an $A$-affinoid algebra, there is a base change spectral sequence of $A'$-modules
\[	{\rm{E}}_2^{ij}=\Tor_{-i}^A(\H^j(G,M),A')\Rightarrow \H^{i+j}(G,M\otimes_AA')	\]
in which the edge map ${\rm{E}}_2^{0,j}=\H^j(G,M)\otimes_AA'\rightarrow \H^j(G,M\otimes_AA')$ is the natural map.
\end{thm}  
In particular, if $A'$ is flat over $A$, then the formation of continuous group cohomology commutes with base change.

\begin{remark}
The base change spectral sequence is a consequence of an isomorphism 
\[	C^{\bullet}(G,M)\otimes_A^{\mathbf{L}}A'\rightarrow C^{\bullet}(G,M\otimes_AA')	\]
in the bounded derived category $\D_{\rm{coh}}^b(A')$ of finite $A'$-modules.
\end{remark}

We will be primarily concerned with the Galois cohomology of groups $G$ of $p$-cohomological dimension $1$.  In that case, the base change theorem takes a particularly nice form.
\begin{cor}\label{coh-dim-1}
Suppose $G$ has $p$-cohomological dimension $1$.  Then
\begin{enumerate}
\item	The formation of $\H^1(\Gamma,M)$ commutes with affinoid base change on $A$.
\item	The spectral sequence degenerates at the $\rm{E}_3$ page
\item	There is a low-degree exact sequence
\begin{equation*}\begin{split}0\rightarrow \H^0(\Gamma,M)\otimes_AA'/\Tor_2^A&(\H^1(\Gamma,M),A')\rightarrow \H^0(\Gamma,M\otimes_AA')	\\
&\rightarrow \Tor_1^A(\H^1(\Gamma,M),A')\rightarrow 0\end{split}\end{equation*}
\end{enumerate}
\end{cor}

\subsection{Cohomology of semi-linear $G$-modules}

We will need to extend some of Pottharst's results.  Throughout this subsection, let $M$ be a finite flat $A[\![t]\!]$-module, equipped with its natural Fr\'echet topology (i.e., as the inverse limit $\varprojlim_k M/t^k$ of finite Banach $A$-modules), and suppose $M$ is equipped with a continuous $A[\![t]\!]$-semi-linear $G$-action, where $G$ acts on $A[\![t]\!]$ so that the action on $A$ is trivial and the action preserves the $t$-adic filtration.
\begin{prop}\label{h1-frechet}
Let $M$ be as above.  Then 
\begin{enumerate}
\item	$\H^i(G,M)=0$ for $i>e$, and
\item	if $A\rightarrow A'$ is a quotient of affinoid algebras, the formation of $\H^e(G,M)$ commutes with base change to $A'$, i.e., the natural map $\H^e(G,M)\otimes_AA'\rightarrow \H^e(G,M\widehat\otimes_AA')$ is an isomorphism.
\end{enumerate}
\end{prop}
\begin{proof}
\begin{enumerate}
\item	For each quotient $M/t^k$, we have the continuous cochain complex $C^\bullet(G,M/t^k)$, and the the transition maps $C^\bullet(G,M/t^{k+n})\rightarrow C^\bullet(G,M/t^k)$ are surjective.  Therefore, by~\cite[Theorem 3.5.8]{weibel}, for each $i$ we have an exact sequence
\[	0\rightarrow{\varprojlim_k}^1\H^{i-1}(G,M/t^k)\rightarrow \H^i(G,M)\rightarrow \varprojlim_k \H^i(G,M/t^k)\rightarrow 0	\]
Then for $i>e+1$, we have $\H^i(G,M/t^k)=0$ and $\H^{i-1}(G,M/t^k)=0$, by~\cite[Theorem 1.1(4)]{pottharst}.  Therefore, $\H^i(G,M)=0$.  If $i=e+1$, $\H^i(G,M/t^k)=0$ by~\cite[Theorem 1.1(4)]{pottharst}.  Then we use the long exact sequence associated to 
\[	0\rightarrow C^\bullet(G,t^nM/t^{k+n})\rightarrow C^\bullet(G,M/t^{k+n})\rightarrow C^\bullet(G,M/t^k)\rightarrow 0	\]
and the vanishing of $\H^{i}(G,t^nM/t^{k+n})$ to see that $\{\H^i(G,M/t^k)\}_k$ has surjective transition maps.  Therefore, $\varprojlim_k^1\H^{i-1}(G,M/t^k)$ and $\H^i(G,M)$ vanish as well.
\item	Let $A'=A/J$ be a quotient of $A$.  Since $A'$ is a finitely presented $A$-module, $-\otimes_AA'$ commutes with taking inverse limits with surjective transition maps.  It follows that
\[	C^\bullet(G,M)\otimes_AA'\xrightarrow{\simeq}\varprojlim_k\left(C^\bullet(G,M/t^k)\otimes_AA'\right)	\]
But the natural map $C^\bullet(G,M/t^k)\otimes_AA'\rightarrow C^\bullet(G,(M/t^k)\otimes_AA')$ is a quasi-isomorphism, by~\cite[Lemma 1.5]{pottharst}, so we obtain a quasi-isomorphism
\[	C^\bullet(G,M)\otimes_AA'\rightarrow \varprojlim_k\left(C^\bullet(G,(M/t^k)\otimes_AA')\right) = C^\bullet(G,M\widehat\otimes_AA')	\]
As $A'$ is finitely presented as an $A$-module, we may find a projective resolution $D^\bullet\rightarrow A'$.  Because the terms of $C^\bullet(G,M)$ are $A$-flat, the induced map $C^\bullet(G,M)\otimes_AD^\bullet\rightarrow C^\bullet(G,M)\otimes_AA'$ is a quasi-isomorphism, and we obtain a second-quadrant spectral sequence abutting to the homology of $C^\bullet(G,M\widehat\otimes_AA')$.  Consideration of the $E_2$-page yields the desired result.
\end{enumerate}
\end{proof}

\begin{prop}\label{h0-stab}
Let the notation be as above, and suppose in addition that $G$ is procyclic with topological generator $\gamma$, and that $\H^0(G,\gr^\bullet M)$ is $A$-finite, where $\gr^\bullet M$ is the associated graded module to $M$.  Then there is some $N$ such that $\H^0(G,M)\xrightarrow{\sim}\H^0(G,M/t^kM)$ for any $k\geq N$.
\end{prop}

Before we prove Proposition~\ref{h0-stab}, we record a few useful consequences about $\H^1(G,M)$.
\begin{cor}\label{h0-stab-cor}
Let the notation be as in Proposition~\ref{h0-stab}.  Then
\begin{enumerate}
\item\label{h0-stab-cor-1}	$\H^1(G,M)=\varprojlim_{k}\H^1(G,M/t^k)$, 
\item\label{h0-stab-cor-2}	for any $k\geq N$, the sequence
\[	0\rightarrow \H^1(G,t^kM)\rightarrow \H^1(G,M)\rightarrow \H^1(G,M/t^k)\rightarrow 0	\]
is exact, 
\item\label{h0-stab-cor-3}	for any $k\in\Z$ and $k'\in \N$, the kernel of the natural map $\H^1(G,t^{k+k'}M)\rightarrow \H^1(G,t^{k}M)$ is a quotient of the $A$-finite module $\H^0(G,t^kM/t^{k+k'})$, the cokernel is the $A$-finite module $\H^1(G,t^kM/t^{k+k'})$, and for all but finitely many $k, k'$, it is an injection, and
\item\label{h0-stab-cor-4}	$\varprojlim_k\H^1(G,t^kM)=0$
\end{enumerate}
\end{cor}
\begin{proof}
\begin{enumerate}
\item	Proposition~\ref{h0-stab} implies that the projective system $\{\H^0(G,M/t^k)\}_{k\geq 0}$ satisfies the Mittag-Leffler condition, so $\varprojlim_k^1 \H^0(G,M/t^k)=0$.  Together with the exact sequence
\[	0\rightarrow{\varprojlim_k}^1\H^{0}(G,M/t^k)\rightarrow \H^1(G,M)\rightarrow \varprojlim_k \H^1(G,M/t^k)\rightarrow 0	\]
this yields the desired result.
\item	For each $n\geq 0$, the exact sequence of Banach $A$-modules
\[	0\rightarrow t^kM/t^{k+n}\rightarrow M/t^{k+n}\rightarrow M/t^k\rightarrow 0	\]
induces a long exact sequence in cohomology.  If $k\geq N$, then Proposition~\ref{h0-stab} implies that $\H^0(G,M/t^{k+n})\rightarrow \H^0(G,M/t^k)$ is a surjection.  Therefore, the connecting homomorphism $\delta:\H^0(G,M/t^k)\rightarrow \H^1(G,t^kM/t^{k+n})$ is zero and 
\[	0\rightarrow \H^1(G,t^kM/t^{k+n})\rightarrow \H^1(G,M/t^{k+n})\rightarrow \H^1(G,M/t^k)\rightarrow 0	\]
is exact.  The first part of this corollary (applied to $\H^1(G,t^kM)$) implies that $\H^1(G,t^kM)=\varprojlim_n\H^1(G,t^kM/t^{k+n})$, and since the inverse system $\{\H^1(G,t^kM/t^{k+n})\}_n$ has surjective transition maps, we obtain an exact sequence
\[	0\rightarrow \H^1(G,t^kM)\rightarrow \H^1(G,M)\rightarrow \H^1(G,M/t^k)\rightarrow 0	\]
as desired.
\item	For every $n\geq 0$, the exact sequence 
\[	0\rightarrow t^{k+k'}M/t^{k+k'+n}\rightarrow t^kM/t^{k+k'+n}\rightarrow t^kM/t^{k+k'}\rightarrow 0	\]
induces an exact sequence
\[	\H^0(G,t^kM/t^{k+k'})\rightarrow \H^1(G,t^{k+k'}M/t^{k+k'+n})\rightarrow \H^1(G,t^kM/t^{k+k'+n})\rightarrow \H^1(G,t^kM/t^{k+k'})\rightarrow 0	\]
If $k\gg0$ or $k\ll0$, and $k'\gg0$, $\H^0(G,t^kM/t^{k+k'})=0$ and we may take the projective limit as $n\rightarrow\infty$ to obtain an injection $\H^1(G,t^{k+1}M)\rightarrow \H^1(G,t^kM)$.

Otherwise, let
\[	K_{k,k',n}:=\ker\left(\H^1(G,t^{k+k'}M/t^{k+k'+n})\rightarrow \H^1(G,t^kM/t^{k+k'+n})\right) = \H^0(G,t^kM/t^{k+k'})/\H^0(G,t^kM/t^{k+k'+n})	\]
To show that $\varprojlim_n K_{k,k',n}$ is $A$-finite, it suffices to show that the natural map
\[	\H^0(G,t^kM/t^{k+k'})\rightarrow \varprojlim_n K_{k,k',n}	\]
is a surjection.  But Proposition~\ref{h0-stab} (applied to $t^kM$) implies that $\{\H^0(G,t^kM/t^{k+k'+n})\}_n$ is stationary for $n\gg0$, so $\{\H^0(G,t^kM/t^{k+k'+n})/\H^0(G,t^{k+k'}M/t^{k+k'+n})\}_n$ satisfies the Mittag-Leffler condition.  This implies that
\[	0\rightarrow \varprojlim_n\left(\H^0(G,t^kM/t^{k+k'+n})/\H^0(G,t^{k+k'}M/t^{k+k'+n})\right)\rightarrow \H^0(G,t^kM/t^{k+k'})\rightarrow \varprojlim_n K_{k,k',n} \rightarrow 0	\]
is exact.

To identify the cokernel of $\H^1(G,t^{k+k'}M)\rightarrow \H^1(G,t^kM)$, we again consider the exact sequences
\[	0\rightarrow K_{k,k',n}\rightarrow \H^1(G,t^{k+k'}M/t^{k+k'+n})\rightarrow \H^1(G,t^kM/t^{k+k'+n})\rightarrow \H^1(G,t^kM/t^{k+k'})\rightarrow 0	\]
for $n\geq 0$.  As $n$ varies, the natural transition maps are surjections, so we see that the cokernel $\coker\left(\H^1(G,t^{k+k'}M)\rightarrow \H^1(G,t^kM)\right)$ is identified with the $A$-finite module $\H^1(G,t^kM/t^{k+k'})$.
\item	By Proposition~\ref{h0-stab}, there is some $N$ such that for $k\geq N$ and any $n\geq 0$, the sequence
\[	0\rightarrow \H^1(G,t^kM/t^{k+n})\rightarrow \H^1(G,M/t^{k+n})\rightarrow \H^1(G,M/t^k)\rightarrow 0	\]
is exact.  Taking the projective limit as $n\rightarrow \infty$ and applying part~\ref{h0-stab-cor-1} to $M$ and $t^kM$, we obtain an exact sequence
\[	0\rightarrow \H^1(G,t^kM)\rightarrow \H^1(G,M)\rightarrow \H^1(G,M/t^k)	\]
Taking the projective limit as $k\rightarrow \infty$, we obtain an exact sequence
\[	0\rightarrow \varprojlim_k\H^1(G,t^kM)\rightarrow \H^1(G,M)\rightarrow \varprojlim_k\H^1(G,M/t^k)	\]
But the map $\H^1(G,M)\rightarrow \varprojlim_k\H^1(G,M/t^k)$ is an isomorphism by part~\ref{h0-stab-cor-1}, so $\varprojlim_k\H^1(G,t^kM)=0$.
\end{enumerate}
\end{proof}

\begin{remark}
If we knew that surjections of $\Q_p$-Fr\'echet spaces admit continuous sections (so that short exact sequences of $G$-representations would yield long exact sequences in cohomology), the proof of Corollary~\ref{h0-stab-cor} could be simplified in a number of places.  As we are not aware of a result to that effect, we are instead forced to use the Mittag-Leffler property proved in Proposition~\ref{h0-stab}.
\end{remark}

We now turn to the proof of Proposition~\ref{h0-stab}; we will require a number of preliminaries.  We observe at the outset that the cohomology group $\H^0(G,M)$ is computed by $\H^0$ of the complex
\[	C_{\rm{alg}}^\bullet:0\rightarrow M\xrightarrow{\gamma-1}M\rightarrow 0	\]
For the remainder of this section, we therefore take $\H^\bullet(G,M)$ to be the homology of $C_{\rm{alg}}^\bullet$.  It is purely algebraic, with no input from the topology of $G$ or $M$.

\begin{lemma}\label{h0-inj}
Let $M$ and $G$ be as above.  Then
\begin{enumerate}
\item	there is some $N_0$ such that $t^kM/t^{k+1}M$ has no non-zero $G$-invariants for any $k\geq N_0$,
\item	for any $k\geq N_0$, $(t^kM)^{G=1}=\{0\}$,
\item	for any $k\geq N_0$, the natural maps 
$$(M/t^{k+1}M)^{G=1}\rightarrow (M/t^kM)^{G=1}\text{ and }M^{G=1}\rightarrow (M/t^kM)^{G=1}$$
are injections, and
\item	$\H^0(G,M)$ is finite.
\end{enumerate}
\end{lemma}
\begin{proof}
\begin{enumerate}
\item	This follows from the finiteness of $\H^0(G,\gr^\bullet M)$.
\item	Since $t^kM = \varprojlim_h t^kM/t^{k+h}M$ and taking $G$-invariants is left-exact, it is enought to show that for all $h\geq 0$, $(t^kM/t^{k+h}M)^{G=1}=0$.  But this follows from repeated applications of the exact sequence
\[	0\rightarrow (t^{k'+1}M/t^{k+h}M)^{G=1}\rightarrow (t^{k'}M/t^{k+h}M)^{G=1}\rightarrow (t^{k'}M/t^{k'+1}M)^{G=1}=0	\]
for $k\leq k'\leq k+h$.
\item	We have an exact sequence
\[	0\rightarrow (t^kM/t^{k+1}M)^{G=1}\rightarrow (M/t^{k+1}M)^{G=1}\rightarrow (M/t^{k}M)^{G=1}	\]
By the choice of $k$, $(t^kM/t^{k+1}M)^{G=1}=0$.  Similarly, we have an exact sequence
\[	0\rightarrow (t^kM)^{G=1}\rightarrow (M)^{G=1}\rightarrow (M/t^{k}M)^{G=1}	\]
By the choice of $k$, $(t^kM)^{G=1}=\{0\}$.
\item	We have seen that $\H^0(G,M)$ injects into $(M/t^{k}M)^{G=1}$ for sufficiently large $k$.  But $M/t^kM$ is a finite $A$-module, so $\H^0(G,M)$ is $A$-finite as well.
\end{enumerate}
\end{proof}

\begin{lemma}\label{ht-local-bound}
Let $M$ and $G$ be as above, and let
$$0\rightarrow I\rightarrow B\rightarrow B'\rightarrow 0$$ 
be a small extension of artin local $A$-algebras, and let $\mathfrak{m}_B$ be the maximal ideal of $B$.  Let $M_B$, $M_{B/\mathfrak{m}_B}$, and $M_{B'}$ denote $M\otimes_AB$, $M\otimes_AB/\mathfrak{m}_B$, and $M\otimes_AB'$, respectively.
Suppose that the natural maps
$$\H^0(G,M_{B/\mathfrak{m}_B}/t^{k+1})\rightarrow \H^0(G,M_{B/\mathfrak{m}_B}/t^k)\text{ and }\H^0(G,M_{B'}/t^{k+1})\rightarrow \H^0(G,M_{B'}/t^{k})$$ 
are isomorphisms.  Then $\H^0(G,M_B/t^{k+1})\rightarrow \H^0(G,M_B/t^{k})$ is an isomorphism, as well.
\end{lemma}

\begin{remark}
It is crucial for our application of Lemma~\ref{ht-local-bound} in the proof of Proposition~\ref{h0-stab} that $B$ and $B'$ are not assumed to be $\Q_p$-finite.  This is why we are working with the ``algebraic'' cohomology groups computed by the complex $0\rightarrow M\xrightarrow{\gamma-1}M\rightarrow 0$.
\end{remark}

\begin{proof}
Since $0\rightarrow I\rightarrow B\rightarrow B'\rightarrow 0$ is a small extension, $I$ is a principal ideal killed by $\mathfrak{m}_B$.  It follows that the complex 
$$0\rightarrow IM_B/t^{k+1}\xrightarrow{\gamma-1}IM_B/t^{k+1}\rightarrow 0$$
 (resp. $0\rightarrow IM_B/t^{k}\xrightarrow{\gamma-1}IM_B/t^{k}\rightarrow 0$) is isomorphic as a complex of $B/\mathfrak{m}_B$-vector spaces to the complex 
$$0\rightarrow (M_{B/\mathfrak{m}_B})/t^{k+1}\xrightarrow{\gamma-1}(M_{B/\mathfrak{m}_B})/t^{k+1}\rightarrow 0$$ 
(resp. $0\rightarrow M_{B/\mathfrak{m}_B}/t^{k}\xrightarrow{\gamma-1}(M_{B/\mathfrak{m}_B})/t^{k}\rightarrow 0$).  Then the hypothesis that 
$$\H^0(G,M_{B/\mathfrak{m}_B}/t^{k+1})\rightarrow \H^0(G,M_{B/\mathfrak{m}_B}/t^k)$$ 
is an isomorphism implies that $\H^0(G,IM_B/t^{k+1})\rightarrow \H^0(G,IM_B/t^{k})$ is an isomorphism, as well.

Since $M$ is $A$-flat, we have a commutative diagram 
$$\begin{CD}
0	@>>>	IM_B/t^{k+1} @>>> M_B/t^{k+1}	@>>> M_{B'}/t^{k+1}	@>>>	0	\\
@.	@VVV	@VVV	@VVV	@.	\\
0	@>>>	IM_B/t^{k} @>>> M_B/t^{k}	@>>> M_{B'}/t^{k}	@>>>	0
\end{CD}$$
where the rows are exact.

Taking $G$-invariants, we get a commutative diagram
$$\begin{CD}
0	@>>>	\!\!(IM_B/t^{k+1})^G\!\! @>>> \!\!(M_B/t^{k+1})^G\!\!	@>>> \!\!(M_{B'}/t^{k+1})^G\!\!	@>>>	\!\!\H^1(G,IM_B/t^{k+1})	\\
@.	\!\!@VVV\!\!	\!\!@VVV\!\!	\!\!@VVV\!\!	\!\!@VVV	\\
0	@>>>	\!\!(IM_B/t^{k})^G\!\! @>>> \!\!(M_B/t^{k})^G\!\!	@>>> \!\!(M_{B'}/t^{k})^G\!\!	@>>>	\!\!\H^1(G,IM_B/t^k)
\end{CD}$$

To show that $\H^0(G,M_B/t^{k+1})\rightarrow \H^0(G,M_B/t^{k})$ is an isomorphism, it suffices by the five lemma to show that $\H^1(G,IM_B/t^{k+1})\rightarrow \H^1(G,IM_B/t^k)$ is an isomorphism.   But $IM_B/t^{k+1}$ and $IM_B/t^{k}$ are finite $B/\mathfrak{m}_B$-vector spaces, so to show this, it is enough to show that they have the same dimension as $B/\mathfrak{m}_B$-vector spaces (since the map is {\it a priori} a surjection).  But
\begin{eqnarray*}
\dim \H^1(G,IM_B/t^{k+1}) &=& \dim IM_B/t^{k+1}-\dim IM_B/t^{k+1}+\dim \H^0(G,IM_B/t^{k+1})	\\
&=& \dim \H^0(G,IM_B/t^{k+1}) = \dim \H^0(G,IM_B/t^{k})	\\
&=& \dim IM_B/t^{k}-\dim IM_B/t^{k}+\dim \H^1(G,IM_B/t^{k})	\\
&=& \dim \H^1(G,IM_B/t^{k})
\end{eqnarray*}
\end{proof}

Now we are in a position to prove Proposition~\ref{h0-stab}.
\begin{proof}[Proof of Proposition~\ref{h0-stab}]
We use Noetherian induction on $\Spec(A)$.  By Lemma~\ref{h0-inj}, we may first choose $N_0$ such that for $k\geq N_0$, 
$$M^{G=1}\hookrightarrow 
\cdots\hookrightarrow(M/t^{k+1})^{G=1}\hookrightarrow (M/t^{k})^{G=1}\hookrightarrow\cdots\hookrightarrow(M/t^{N_0})^{G=1}$$
It follows that for any $k\geq N_0$, the cokernel of 
$(M/t^{k+1})^{G=1}\rightarrow (M/t^{k})^{G=1}$
is supported on a Zariski-closed subspace of $\Spec(A)$, namely the support of the cokernel of $M^{G=1}\hookrightarrow (M/t^{N_0})^{G=1}$.  Let $\{\mathfrak{q}_j\}$ be the (finitely many) primes corresponding to the irreducible components of this subspace.  We will find some $N_1\gg0$ such that for $k\geq N_1$, the natural map
$(M/t^{k+1})^{G=1}\otimes_A(\prod_jA_{\mathfrak{q}_j}^\wedge)\rightarrow (M/t^{k})^{G=1}\otimes_A(\prod_jA_{\mathfrak{q}_j}^\wedge)$
is an isomorphism.

Since $(M/t^{k+1})^{G=1}\otimes_AA_{\mathfrak{q}_j}^\wedge= ((M/t^{k+1})\otimes_AA_{\mathfrak{q}_j}^\wedge)^{G=1}$ (by flatness of $A\rightarrow A_{\mathfrak{q}_i}^\wedge$), it is enough to produce some $N_{1,j}$ such that 
$$((M/t^{k+1})\otimes_AA_{\mathfrak{q}_j}/{\mathfrak{q}_j^m})^{G=1}\rightarrow ((M/t^{k})\otimes_AA_{\mathfrak{q}_j}/{\mathfrak{q}_j^m})^{G=1}$$
is an isomorphism for all $k\geq N_{1,j}$ and for all $m$.  But this follows from Lemma~\ref{ht-local-bound} and the fact that any surjection of artin local rings can be factored into a sequence of small extensions.

Let $N_1=\max\{N_{1,j}\}$.  Then the natural map
$M^{G=1}\rightarrow (M/t^{N_1}M)^{G=1}$ is an injection with cokernel supported on a strictly smaller Zariski closed subspace of $\Spec(A)$.  If it is actually an isomorphism, we are done; otherwise, we repeat the argument with primes of $A$ corresponding to irreducible components of the support of the cokernel of $M^{G=1}\hookrightarrow (M/t^{N_1}M)^{G=1}$.

This process terminates in finitely many steps, so we find that for $k$ large enough, the natural maps $(M/t^{k+1}M)^{G=1}\rightarrow (M/t^{k}M)^{G=1}$ are isomorphisms of $A$-modules.  It follows that $\H^0(G,M)=\H^0(G,M/t^kM)$ for sufficiently large $k$.
\end{proof}

\section{The functors $\D_{\B_\ast}(V)$}\label{db-functors}

\subsection{Overview}

In this section, we discuss the functors $\D_{\HT}(V)$ and $\D_{\dR}(V)$; we relate them to $(\varphi,\Gamma)$-modules, and we prove they are coherent sheaves on $\Sp(A)$.  We also relate $\D_{\st}(V)$ and $\D_{\cris}(V)$ to $(\varphi,\Gamma)$-modules, and conjecture that they are coherent sheaves on $\Sp(A)$.

Throughout this section, we let $E$ and $K$ be finite extensions of $\Q_p$, and we let $X$ be a quasi-separated rigid analytic space over $E$.  

\begin{definition}
A \emph{family of Galois representations} over $X$ is a locally free $\mathscr{O}_X$-module $\mathscr{V}$ of rank $d$ together with an $\mathscr{O}_X$-linear action of $\Gal_K$ which acts continuously on $\Gamma(U,V)$ for every affinoid subdomain $U\subset X$.
\end{definition}

\begin{definition}
Let $\B_\ast$ be one of the period rings $\B_{\HT}$, $\B_{\dR}$, $\B_{\max}$, or $\B_{\st}$.  Then for any family of Galois representations, we define the presheaf
\[	\mathscr{D}_\ast^K(\mathscr{V})(U):=\left(\mathscr{B}_{X,\ast}(U)\otimes_{\mathscr{O}_X(U)}\mathscr{V}(U)\right)^{\Gal_K}	\]
where $\mathscr{B}_{X,\ast}$ is one of the sheaves of period rings defined in section~\ref{sheafy-rings}.  
We say that $\mathscr{V}$ is \emph{$\B_\ast$-admissible} (or simply Hodge--Tate, de Rham, semi-stable, or crystalline) if $\mathscr{D}_\ast^K(\mathscr{V})$ is a projective $\mathscr{O}_X\otimes\B_\ast^{\Gal_K}$-module of rank $d$, and the natural morphism
\[	\alpha_{\mathscr{V}}:\mathscr{B}_{X,\ast}\otimes_{\mathscr{O}_X\otimes\B_\ast^{\Gal_K}}\mathscr{D}_\ast^K(\mathscr{V})\rightarrow \mathscr{B}_{X,\ast}\otimes_{\mathscr{O}_X}\mathscr{V}	\]
is an isomorphism.
\end{definition}
Let $\{U_i\}_{i\in I}$ be an admissible covering of $X$.  Then because $\mathscr{V}$ and $\mathscr{B}_{X,\ast}$ are both sheaves on $X$, we have an exact sequence
$$0\rightarrow \Gamma(X,\mathscr{B}_{X,\ast}\otimes_{\mathscr{O}_X}\mathscr{V})\rightarrow \prod_{i\in I}\Gamma(U_i,\mathscr{B}_{X,\ast}\otimes_{\mathscr{O}_X}\mathscr{V})\rightarrow\prod_{i,j\in I}\Gamma(U_i\cap U_j,\mathscr{B}_{X,\ast}\otimes_{\mathscr{O}_X}\mathscr{V})$$
Each of these terms has a continuous action of $\Gal_K$, by assumption, and since the formation of $\Gal_K$-invariants is left-exact, we have an exact sequence
$$0\rightarrow \mathscr{D}_\ast^K(\mathscr{V})(X)\rightarrow\prod_{i\in I} \mathscr{D}_\ast^K(\mathscr{V})(U_i)\rightarrow\prod_{i,j\in I}\mathscr{D}_\ast^K(\mathscr{V})(U_i\cap U_j)$$
It follows that $\mathscr{D}_\ast^K(\mathscr{V})$ is actually a sheaf of $\mathscr{O}_X\otimes_{\Q_p}\B_\ast^{\Gal_K}$-modules.  However, we do not know at this stage that $\mathscr{D}_\ast^K(\mathscr{V})(U)$ is finite, let alone that $\mathscr{D}_\ast^K(\mathscr{V})$ is a coherent sheaf of $\mathscr{O}_X\otimes_{\Q_p}\B_\ast^{\Gal_K}$-modules.

Suppose that $X=\Sp(A)$, where $A$ is a $\Q_p$-finite artin ring, and $V$ is a finite projective $A$-module equipped with a continuous $A$-linear action of $\Gal_K$.  Then $\left((A\otimes_{\Q_p}\B_\ast)\otimes_AV\right)^{\Gal_K} = (\B_\ast\otimes_{\Q_p}V)^{\Gal_K}$ as $\B_\ast^{\Gal_K}$-vector spaces.  In fact, $V$ is $\B_\ast$-admissible as an $A$-linear representation in the sense above if and only if the underlying $\Q_p$-linear representation is $\B_\ast$-admissible:
\begin{prop}\label{artin-adm}
Let $A$ be a $\Q_p$-finite artin local ring with maximal ideal $\mathfrak{m}$, let $n:=\dim_{\Q_p}A$ and let $V$ be a finite free $A$-module of rank $d$ equipped with a continuous $A$-linear $\Gal_K$-action.  Then $V$ is $\B_\ast$-admissible as an $A$-representation if and only if its underlying $\Q_p$-representation is $\B_\ast$-admissible.
\end{prop}
\begin{proof}
It is clear that $\B_\ast$-admissibility over $A$ implies $\B_\ast$-admissibility over $\Q_p$.  For the converse, assume $V$ is $\B_\ast$-admissible when viewed as a $\Q_p$-representation, so $((A\otimes_{\Q_p} \B_\ast)\otimes_{A}V)^{\Gal_K}=(\B_\ast\otimes_{\Q_p} V)^{\Gal_K}$ is an $nd$-dimensional $\B_\ast^{\Gal_K}$-vector space and the natural map $\B_\ast\otimes_{\B_\ast^{\Gal_K}} \D_{\B_\ast}(V)\rightarrow \B_\ast\otimes_{\Q_p}V$ is an isomorphism.

We first assume that $A=E$ is a field.  In that case, $A\otimes_{\Q_p} \B_\ast^{\Gal_K}$ is a product of fields, so $\D_{\B_\ast}(V)$ is certainly locally free and therefore projective.  In addition, the isomorphism $\alpha_V:\B_\ast\otimes_{\B_\ast^{\Gal_K}} \D_{\B_\ast}(V)\rightarrow \B_\ast\otimes_{\Q_p}V$ tells us that the natural map
$$(A\otimes_{\Q_p} \B_\ast)\otimes_{A\otimes_{\Q_p} \B_\ast^{\Gal_K}} \D_{\B_\ast}(V)\rightarrow (A\otimes_{\Q_p} \B)\otimes_AV$$ 
is an isomorphism.  In particular, $\D_{\B_\ast}(V)$ is locally free of rank $d$ over $A\otimes\B_\ast^{\Gal_K}$.  Since $A\otimes_{\Q_p} \B_\ast^{\Gal_K}$ is semi-local, $\D_{\B_\ast}(V)$ is free over $A\otimes_{\Q_p} \B_\ast^{\Gal_K}$.

Now consider the general case.  We will factor the extension of $\Q_p$-finite artin rings $A\twoheadrightarrow A/\mathfrak{m}V$ as a sequence of small extensions and proceed by induction.  So suppose we have a small extension $f:A\twoheadrightarrow A'$, so that $\mathfrak{m}\ker(f)=0$ and $\ker(f)=(t)\cong A/\mathfrak{m}$, and suppose the result holds for $A'$-representations.

We have a surjection of $\Q_p$-representations $V\twoheadrightarrow V\otimes_AA'$, with kernel $tV$.  By the formalism of admissible representations, $V\otimes_AA'$ is $\B_\ast$-admissible and we have a surjection $\D_{\B_\ast}(V)\twoheadrightarrow \D_{\B_\ast}(V\otimes_AA')$ with kernel $\D_{\B_\ast}(tV)$.

We claim that the kernel of this surjection is $t\D_{\B_\ast}(V)$.  Clearly, $t\D_{\B_\ast}(V)\subset \D_{\B_\ast}(tV)$, since there is no Galois action on the coefficients.  On the other hand, suppose that $mv\in \D_{\B_\ast}(tV)$ for some $m\in \ker(f)$, $v\in \B_\ast\otimes V$.  Then $mv=mg(v)$ for any $g\in \Gal_K$, so $v=g(v)$ in $\B_\ast\otimes (V\otimes_A A/I)$ for all $g\in \Gal_K$, where $I$ is the ideal of elements of $A$ killed by $m$.  But again by the formalism of admissible representations, we have a surjection $\D_{\B_\ast}(V)\twoheadrightarrow \D_{\B_\ast}(V/IV)$, so there is some $\tilde v\in \D_{\B_\ast}(V)$ such that $\tilde v\cong v\mod{I}$.  Since $v$ and $\tilde v$ differ by an element of $\D_{\B_\ast}(IV)$ and $m$ kills $\D_{\B_\ast}(IV)$, $mv=m\tilde v\in t\D_{\B_\ast}(V)$, as desired.

By the assumption on $A'$-representations, $\D_{\B_\ast}(V\otimes_AA')$ is a free module of rank $d$ over $A'\otimes_{\Q_p}\B_\ast^{\Gal_K}$.  Furthermore, $\D_{\B_\ast}(V)\otimes_AA'=\D_{\B_\ast}(V\otimes_AA')$.  If $A_i$ is a local factor of the semi-local ring $A\otimes_{\Q_p}\B_\ast^{\Gal_K}$, then $A_i':=A_i/t$ is a local factor of $A'\otimes_{\Q_p}\B_\ast^{\Gal_K}$.  Therefore, Nakayama's lemma implies that $\D_{\B_\ast}(V)\otimes_AA_i$ is generated by $d$ elements for all $i$, so we have a surjection $(A\otimes_{\Q_p} \B_\ast^{\Gal_K})^d\twoheadrightarrow \D_{\B_\ast}(V)$.  But we also have an isomorphism of $\B_\ast$-modules $\alpha_V:\B_\ast\otimes_{\B_\ast^{\Gal_K}} \D_{\B_\ast}(V)\rightarrow \B_\ast\otimes_{\Q_p}V$, so by comparing the $\B_\ast^{\Gal_K}$-dimensions of $(A\otimes \B_\ast^{\Gal_K})^d$ and $\D_{\B_\ast}(V)$, we see that $\D_{\B_\ast}(V)$ is a free $A\otimes \B_\ast^{\Gal_K}$-module of rank $d$.
\end{proof}

\subsection{$(\varphi,\Gamma)$-modules and $\D_{\B_\ast}(V)$}

When $X=\Sp(\Q_p)$, the overconvergence of Galois representations is important in part because it allows us to recover the $p$-adic Hodge theoretic invariants $D_{\B_\ast}(V)$ from the $(\varphi,\Gamma)$-module.  This allows us to convert questions about Galois groups with cohomological dimension $2$ into questions about profinite groups with cohomological dimension $1$, at the cost of making the coefficients more complicated.

Specifically, when $X=\Sp(\Q_p)$, we have the following results.
\begin{thm}[\cite{sen}]
Let $V$ be a finite-dimensional $\Q_p$-linear representation of $\Gal_K$.  Then $\D_{\HT}^K(V)=\oplus_{i\in\Z}\left(\D_{\Sen}^K(V)\cdot t^i\right)^{\Gamma_K}$.
\end{thm}

\begin{thm}[\cite{fontaine}]
Let $V$ be a finite-dimensional $\Q_p$-linear representation of $\Gal_K$.  Then $\D_{\dR}^K(V)=\left(\D_{\dif}^K(V)\right)^{\Gamma_K}$.
\end{thm}

\begin{thm}[\cite{berger}]
Let $V$ be a finite-dimensional $\Q_p$-linear representation of $\Gal_K$.  Then
$\D_{\st}^K(V)=\left(\D_{\log}^\dagger(V)[1/t]\right)^{\Gamma_K}$ and $\D_{\cris}^K(V)=\left(\D_{\rig}^\dagger(V)[1/t]\right)^{\Gamma_K}$.
\end{thm}

\begin{remark}
Sen and Fontaine used different constructions of $\D_{\Sen}^K(V)$ and $\D_{\dif}^K(V)$ than the one we have given.  The equivalence of the two constructions is shown in~\cite{berger}.
\end{remark}

We will prove analogues of these results for families of Galois representations over an $E$-rigid analytic space $X$.

\begin{thm}\label{dsen-dht}
Let $\mathscr{V}$ be a family of representations of $\Gal_K$ of rank $d$.  Then 
$$\mathscr{D}_{\HT}^K(\mathscr{V})=\oplus_{i\in\Z}\left(\mathscr{D}_{\Sen}^K(\mathscr{V})\cdot t^i\right)^{\Gamma_K}$$
as subsheaves of $\mathscr{B}_{\HT}\otimes_{\mathscr{O}_X}\mathscr{V}$.
\end{thm}
\begin{proof}
Since both $\mathscr{D}_{\HT}^K(\mathscr{V})$ and $\oplus_{i\in\Z}\left(\mathscr{D}_{\Sen}^K(\mathscr{V})\cdot t^i\right)^{\Gamma_K}$ are subsheaves of $\mathscr{B}_{\HT}\otimes_{\mathscr{O}_X}\mathscr{V}$, we may work locally on $X$.  Therefore, we may assume that $X=\Sp(A)$ for some $E$-affinoid algebra $A$ and $V:=\Gamma(X,\mathscr{V})$ admits a free $\Gal_K$-stable $\mathscr{A}$-lattice $V_0$ of rank $d$, where $\mathscr{A}$ is some formal $\mathscr{O}_E$-model for $A$.

Since $\D_{\HT}^{K'}(V)=K'\otimes_K\D_{\HT}^K(V)$ for any finite extension $K'/K$, we may replace $K$ with any finite extension.  Let $L/K$ be a finite extension such that $\Gal_L$ acts trivially on $V_0/12pV_0$.  Then $\D_{\Sen}^{L_n}(V)$ is a finite free $A\otimes_{\Q_p}L_n$-module of rank $d$, and we have a natural Galois-equivariant isomorphism
$$(A\widehat\otimes\C_K)\otimes_{A\otimes_{\Q_p} L_n} \D_{\Sen}^{L_n}(V)\xrightarrow{\sim} (A\widehat\otimes\C_K)\otimes_AV$$
Taking $H_L$-invariants, we get
$$(A\widehat\otimes\C_K^{H_L})\otimes_{A\otimes L_n}\D_{\Sen}^{L_n}(V)\xrightarrow{\sim} \left((A\widehat\otimes\C_K)\otimes_AV\right)^{H_L}$$
since $H_L$ acts trivially on $\D_{\Sen}^{L_n}(V)$ by construction.  We need to take $\Gamma_{L_n}$-invariants of both sides.

It suffices to show that 
\[	\left((A\widehat\otimes\C_K)^{H_{L}}\otimes_{A\otimes L_n} \D_{\Sen}^{L_n}(V)\right)^{\Gamma_n}= \D_{\Sen}^{L_n}(V)^{\Gamma_n}	\]
To see this, we fix an $(A\otimes_{\Q_p} L_n)$-basis $(\ve{e}_1,\ldots ,\ve{e}_d)$ of $\D_{\Sen}^{L_n}(V)$ which is $c_3$-fixed by $\Gamma_{L_n}$ and choose some $\ve{x}\in \left((A\widehat\otimes\C_K)^{H_{L}}\otimes_{A\otimes_{\Q_p} L_n}\D_{\Sen}^{L_n}(V)\right)^{\Gamma_n}$.  Then $\ve{x}=\sum x_i\ve{e}_i$ for some $x_i\in (A\widehat\otimes\C_K)^{H_{L}}$.  By the semi-linearity of the Galois action, this means that for any $\gamma\in\Gamma_n$, $U_\gamma\cdot\gamma(\ve{x})=\ve{x}$, where $\ve{x}$ is the column vector of the $x_i$.  But then we may invoke~\cite[Lemme 3.2.5]{bc} with $V_1=U_\gamma^{-1}$ and $V_2=1$ to get that $\ve{x}\in A\otimes_{\Q_p} L_n$.

Since $\B_{\HT}=\C_K[t,t^{-1}]$, it follows that
$\D_{\HT}^K(V)=\oplus_{i\in\Z}\left(\D_{\Sen}^K(V)\cdot t^i\right)^{\Gamma_K}$,
as desired.
\end{proof}

\begin{lemma}\label{finite-banach-bound}
Let $M$ be a finite $A$-module, where $A$ is a Banach algebra whose value group is discrete, and let $m_1,\ldots,m_r$ generate $M$ over $A$.  Equip $M$ with the norm $|\cdot|_M$ induced by the natural quotient $A^{\oplus r}\twoheadrightarrow M$, where $A^{\oplus r}$ has the norm $|(a_1,\ldots,a_r)|=\max_i\{|a_i|\}$.
Let $T:M\rightarrow M$ be an $A$-linear map such that $|T(m_i)|\leq C|m_i|$ for all $m_i$.  Then the operator norm of $T$ on $M$ is at most $C$.
\end{lemma}
\begin{proof}
Let $m\in M$.  We wish to show that $|T(m)|_M\leq C|m|_M$.  Because the value group of $A$ is discrete, we can write $m=a_1m_1+\cdots+a_rm_r$ such that $|m|_M=\max_i \{|a_i|\}$.  Then 
$$|T(m)|_M\leq \max_i \{|a_i|\cdot |T(m_i)|_M\}\leq C\max_i \{|a_i| \cdot |m_i|_M\}\leq C\max_i\{|a_i|\}=C|m|$$
\end{proof}

\begin{lemma}\label{dsen-finite-orbit}
Let $V$ be a finite free $A$-module of rank $d$, equipped with a continuous $A$-linear action of $\Gal_K$.  Then the module generated by the $\Gamma_{L_n}$-orbit of $x\in \widehat L_{\infty}\widehat\otimes_{L_n}\D_{\Sen}^{L_n}(V)$ is $A$-finite if and only if $x\in \cup_{n'\geq n}L_{n'}\otimes_{L_n}\D_{\Sen}^{L_n}(V)$.
\end{lemma}
\begin{proof}
Certainly the $\Gamma_{L_n}$-orbit of any element of $L_{n'}\otimes_{L_n}\D_{\Sen}^{L_n}(V)$ generates an $A$-finite module.  Conversely, suppose that the $\Gamma_{L_n}$-orbit of $x\in \widehat L_{\infty}\widehat\otimes_{L_n}\D_{\Sen}^{L_n}(V)$ generates a finite $A$-module $M$. Let $\{\ve{e}_1,\ldots,\ve{e}_d\}$ be an $A\otimes_{\Q_p}L_n$-basis of $\D_{\Sen}^{L_n}(V)$, so that the action of $\gamma\in\Gamma_{L_n}$ with respect to $\ve{e}_1,\ldots,\ve{e}_d$ is given by a matrix $(a_{ij})$ with $a_{ij}\in A\otimes_{\Q_p} L_n$.  Write $x=\sum_i c_i\ve{e}_i$.

By assumption, $M$ is finite over $A\otimes_{\Q_p} L_n$, so it is generated by a finite collection $f_1,\ldots, f_r$ of elements of $(A\widehat\otimes \widehat L_\infty)\otimes_{A\otimes_{\Q_p} L_n}\D_{\Sen}^{L_n}(V)$.  Then the coefficients (with respect to $\{\ve{e}_i\}$) of elements of $M$ are contained in the $A\otimes_{\Q_p} L_n$-submodule of $A\widehat\otimes\widehat L_\infty$ generated by the coefficients of $f_1,\ldots,f_r$, which is finite.  But 
$$\gamma(x)=\sum \gamma(c_i)\gamma(\ve{e}_i) = \sum_j (\sum_i a_{ji}\cdot \gamma(c_i))\ve{e}_j$$
Since $\gamma$ is invertible, this shows that the $\Gamma_{L_n}$-orbit of the $c_i$ is in the $A\otimes_{\Q_p} L_n$-span of the coefficients with respect to $\{\ve{e}_i\}$ of $\Gamma_{L_n}\cdot x$.

Thus, we are reduced to the rank $1$ case.  That is, we need to show that if the $\Gamma_{L_n}$-orbit of $c\in A\widehat\otimes_{\Q_p}\widehat L_\infty$ generates an $A\otimes_{\Q_p} L_n$-finite module $M\subset A\widehat\otimes_{\Q_p}\widehat L_\infty$, then $c\in \cup_{n'\geq n}L_{n'}\otimes_{L_n}(A\otimes_{\Q_p} L_n)$.  

Choose a finite set $x_1,\ldots, x_r\in M$ which generates $M$, and give $M$ the quotient norm $|\cdot|_M$ coming from the natural surjection $A^{\oplus r}\twoheadrightarrow M$.  Since $A\widehat\otimes_{\Q_p}\widehat L_\infty$ is a potentially orthonormalizable $A$-module, $M$ is closed in $A\widehat\otimes_{\Q_p}\widehat L_\infty$ by~\cite[Lemma 2.3]{buzzard}, and therefore also acquires a $p$-adic norm $|\cdot|_p$.  All norms on a finite Banach module are equivalent by~\cite[Prop. 3.7.3/3]{bgr}, so $|\cdot|_M$ and $|\cdot|_p$ are equivalent, meaning that there are positive constants $C_1,C_2$ such that $C_1|x|_p\leq |x|_M\leq C_2|x|_p$ for all $x\in M$.  

Then for any $\varepsilon >0$, there is some $m_\varepsilon$ such that $|(\gamma^{p^m}-1)(x_i)|_p<\varepsilon |x_i|_p$ for all $i$ and any $m\geq m_\varepsilon$.  We choose $\varepsilon=\frac{1}{2}\frac{C_1^2}{C_2^2p^{c_3}}$.  This implies that
$$|(\gamma^{p^m}-1)(x_i)|_M\leq C_2|(\gamma^{p^m}-1)(x_i)|_p < C_2\varepsilon |x_i|_p\leq \frac{C_2\varepsilon}{C_1}\cdot |x_i|_M$$
By Lemma~\ref{finite-banach-bound}, $\gamma^{p^m}-1$ has operator norm at most $\frac{C_2\varepsilon}{C_1}$ with respect to $|\cdot|_M$.  But then 
$$|(\gamma^{p^m}-1)(x)|_p\leq \frac{1}{C_1}|(\gamma^{p^m}-1)(x)|_M < \frac{C_2\varepsilon}{C_1^2}|x|_M \leq \frac{C_2^2\varepsilon}{C_1^2}|x|_p$$
so $\gamma^{p^m}-1$ has operator norm at most $\frac{C_2^2\varepsilon}{C_1^2}$ with respect to $|\cdot|_p$.  

Next, we observe that for any integer $m\geq 1$, the kernel of $\gamma^{p^m}-1$ on $A\widehat\otimes_{\Q_p}\widehat L_\infty$ is $A\otimes_{\Q_p}L_{m+n}$.  Therefore, if $(\gamma^{p^m}-1)(M)=0$ for some $m\gg0$, we are done.  Now recall that by the third Tate-Sen axiom, for any $n'\geq n$, there is a $\Gamma_{L_{n'}}$-equivariant topological splitting 
$$A\widehat\otimes_{\Q_p}\widehat L_\infty=(A\otimes_{\Q_p}L_{n'})\oplus X_{H,n'}$$ 
and for $n'\gg_{m_\varepsilon} n$, $\gamma^{p^{m_\varepsilon}}-1$ acts invertibly on $X_{H,n'}$, with the norm of $(\gamma^{p^{m_\varepsilon}}-1)^{-1}$ bounded above by the constant $p^{c_3}$.  Since $\gamma^{p^{n'-n}}-1$ kills $A\otimes_{\Q_p}L_{n'}$, it follows that $(\gamma^{p^{n'-n}}-1)(M)\subset X_{H,n'}$.  But $(\gamma^{p^{n'-n}}-1)(M)\subset M$, so $\gamma^{p^{m_\varepsilon}}-1$ has $p$-adic operator norm at most $\frac{C_2^2\varepsilon}{C_1^2}$ on $(\gamma^{p^{n'-n}}-1)(M)$.  Then for any $x\in M$,
\begin{eqnarray*}
|(\gamma^{p^{n'-n}}-1)(x)|_p&=&|(\gamma^{p^{m_\varepsilon}}-1)^{-1}(\gamma^{p^{m_\varepsilon}}-1)(\gamma^{p^{n'-n}}-1)(x)|_p	\\
&\leq& p^{c_3}|(\gamma^{p^{m_\varepsilon}}-1)(\gamma^{p^{n'-n}}-1)(x)(x)|_p	\\
&\leq& p^{c_3}\frac{C_2^2\varepsilon}{C_1^2}|(\gamma^{p^{n'-n}}-1)(x)|_p = \frac{1}{2}|(\gamma^{p^{n'-n}}-1)(x)|_p
\end{eqnarray*}
This forces $|(\gamma^{p^{n'-n}}-1)(x)|_p=0$, so $(\gamma^{p^{n'-n}}-1)(x)=0$.  Therefore, 
$(\gamma^{p^{n'-n}}-1)(M)=0$
and we are done.
\end{proof}

We can bootstrap this result to relate $\mathscr{D}_{\dR}(\mathscr{V})$ and $\mathscr{D}_{\dif}(\mathscr{V})$, just as in the case when $X=\Sp(\Q_p)$.

\begin{thm}\label{ddif-ddr}
Let $\mathscr{V}$ be a family of representations of $\Gal_K$ of rank $d$.  Then $\mathscr{D}_{\dR}^K(\mathscr{V})=\left(\mathscr{D}_{\dif}^K(\mathscr{V})\right)^{\Gamma_K}$, as subsheaves of $\mathscr{B}_{\dR}\otimes_{\mathscr{O}_X}\mathscr{V}$.
\end{thm}
\begin{proof}
As before, we reduce to the case when $X=\Sp(A)$ for some $E$-affinoid algebra $A$ and $V:=\Gamma(X,\mathscr{V})$ admits a free $\Gal_K$-stable $\mathscr{A}$-lattice $V_0$ of rank $d$, where $\mathscr{A}$ is some formal $\mathscr{O}_E$-model for $A$.

Since $\D_{\dR}^{K'}(V)=K'\otimes_K\D_{\dR}^K(V)$ for any finite extension $K'/K$, we may again replace $K$ with any finite extension; we choose $L/K$ such that $\Gal_{L/K}$ acts trivially on $V_0/12pV_0$.  Then $\D_{\dif}^{L_n,+}(V)$ is a free $A\widehat\otimes L_n[\![t]\!]$-module of rank $d$, and we have a Galois-equivariant isomorphism
$$(A\widehat\otimes\B_{\dR}^+)\otimes_{A\widehat\otimes L_n[\![t]\!]}\D_{\dif}^{L_n,+}\xrightarrow{\sim} (A\widehat\otimes\B_{\dR}^+)\widehat\otimes_A V$$
which respects the $t$-adic filtration on both sides.

After twisting $V$ by some power of the cyclotomic character, it therefore suffices to show that
$$\left((A\widehat\otimes\L_{\dR}^+)\otimes_{A\widehat\otimes L_n[\![t]\!]} \D_{\dif}^{L_n,+}(V)\right)^{\Gamma_{L_n}}= \D_{\dif}^{L_n,+}(V)^{\Gamma_{L_n}}$$
where $\L_{\dR}^+:=(\B_{\dR}^+)^{H_K}$. 
In fact, it suffices to show that 
$$\left((A\widehat\otimes\L_{\dR}^+/t^m)\otimes_{A\otimes_{\Q_p} L_n[\![t]\!]}\D_{\dif}^{L_n,+}(V)\right)^{\Gamma_{L_n}}=\left(\D_{\dif}^{L_n,+}(V)/t^m\right)^{\Gamma_{L_n}}$$
for all $m$, because taking inverse limits commutes with taking $\Gamma_{L_n}$-invariants.  

We will do this by showing that if $x\in(A\widehat\otimes\L_{\dR}^+/t^m)\otimes_{A\otimes_{\Q_p} L_n[\![t]\!]}\D_{\dif}^{L_n,+}(V)$ and the $\Gamma_{L_n}$-orbit of $x$ generates a finite $A\otimes_{\Q_p} L_n[\![t]\!]/t^m$-module, then $x$ actually lives in $\bigcup_{n'\geq n}L_{n'}\otimes_{L_n}\D_{\dif}^{L_n,+}(V)/t^m$.  For then if $x$ is $\Gamma_{L_n}$-fixed, it lives in $L_{n'}\otimes_{L_n}\D_{\dif}^{L_n,+}(V)/t^m$ for some $n'\geq n$.  Since $\left(L_{n'}\otimes_{L_n}\D_{\dif}^{L_n,+}(V)/t^m\right)^{\Gamma_{L_n}}=\D_{\dif}^{L_n,+}(V)/t^m$, we conclude that $x\in\D_{\dif}^{L_n,+}(V)/t^m$.

We proceed by induction on $m$.  We first consider $m=1$.  Then we considering elements of $(A\widehat\otimes \widehat L_\infty)\otimes\D_{\Sen}^{L_n}(V)$ whose $\Gamma_{L_n}$-orbits generate finite $A\otimes_{\Q_p}L_n$-modules.  But such elements actually live in $\bigcup_{n'\geq n}L_{n'}\otimes_{L_n}\D_{\Sen}^{L_n}(V)$, by Lemma~\ref{dsen-finite-orbit}.

Now we assume the result for $m$, and we consider the exact sequence
\begin{equation*}
\begin{split}0\rightarrow t^m (\L_{\dR}^+/t^{m+1})\widehat\otimes_{L_n[\![t]\!]}\D_{\dif}^{L_n,+}(V)\rightarrow (\L_{\dR}^+/t^{m+1})&\widehat\otimes_{L_n[\![t]\!]}\D_{\dif}^{L_n,+}(V)\\
&\rightarrow (\L_{\dR}^+/t^{m})\widehat\otimes_{L_n[\![t]\!]}\D_{\dif}^{L_n,+}(V)\rightarrow 0\end{split}
\end{equation*}
If the $\Gamma_{L_n}$-orbit of $c\in (\L_{\dR}^+/t^{m+1})\widehat\otimes_{L_n[\![t]\!]}\D_{\dif}^{L_n,+}(V)$ generates a finite $A\otimes_{\Q_p}L_n$-module, then its image $\overline{c}$ in $(\L_{\dR}^+/t^{m})\widehat\otimes_{L_n[\![t]\!]}\D_{\dif}^{L_n,+}(V)$ does as well.  By the inductive hypothesis, $\overline{c}\in \bigcup_{n'\geq n}L_{n'}\otimes_{L_n}(\D_{\dif}^{L_n,+}(V)/t^m)$.  We may choose $\widehat{c}\in \bigcup_{n'\geq n}L_{n'}\otimes_{L_n}(\D_{\dif}^{L_n,+}(V)/t^{m+1})$ lifting $\overline{c}$, so that the $\Gamma_{L_n}$-orbit of $c-\hat{c}$ still generates a finite $A\otimes_{\Q_p}L_n$-module.  Then $c-\widehat{c}$ is an element of $t^m (\L_{\dR}^+/t^{m+1})\widehat\otimes_{L_n[\![t]\!]}\D_{\dif}^{L_n,+}(V)$, which is isomorphic to $t^m\cdot(\L_{\dR}^+/t)\widehat\otimes_{L_n[\![t]\!]}\D_{\dif}^{L_n,+}(V)$ as a $\Gamma_{L_n}$-representation.  But the $m=1$ case applies to this latter space, so we are done.
\end{proof}

We can similarly relate $\mathscr{D}_{\cris}^K(\mathscr{V})$ and $\mathscr{D}_{\st}^K(\mathscr{V})$ to the family of $(\varphi,\Gamma)$-modules $\mathscr{D}_{\rig,K}^\dagger(\mathscr{V})$, following~\cite{berger}.
\begin{thm}\label{dcris-phi-gamma}
Let $\mathscr{V}$ be a family of representations of $\Gal_K$.  Then $\mathscr{D}_{\cris}^K(\mathscr{V})=\left(\mathscr{D}_{\rig,K}^\dagger(\mathscr{V})[1/t]\right)^{\Gamma_K}$, and $\mathscr{D}_{\st}^K(\mathscr{V})=\left(\mathscr{D}_{\log,K}^\dagger(\mathscr{V})[1/t]\right)^{\Gamma_K}$.  The first equality is as subsheaves of $\widetilde{\mathscr{B}}_{\rig}^\dagger\otimes_{\mathscr{O}_X}\mathscr{V}$, and the second is as subsheaves of $\widetilde{\mathscr{B}}_{\log}^\dagger\otimes_{\mathscr{O}_X}\mathscr{V}$.
\end{thm}

We will need a number of preparatory results.  Throughout the proofs of these results, we will use freely the fact that if $A$ is a $\Q_p$-Banach algebra, then $A$ is potentially orthonormalizable in the sense of~\cite{buzzard}.  This follows from~\cite[Proposition 10.1]{schneider}, since $\Q_p$ is discretely valued.  This has the consequence that injections of Fr\'echet spaces are preserved under completed tensor products with $A$ over $\Q_p$.

\begin{lemma}
Let $A$ be an orthonormalizable $\Q_p$-Banach algebra, and let $\mathscr{A}$ be its unit ball.  Let $h$ be a positive integer.  Then 
$$\cap_{k=0}^\infty p^{-hk}(\mathscr{A}\widehat\otimes\widetilde{\A}^{\dagger,p^{-k}s})=\mathscr{A}\widehat\otimes\widetilde{\A}^+\text{ and }\cap_{k=0}^\infty p^{-hk}(\mathscr{A}\widehat\otimes\widetilde{\A}_{\rig}^{\dagger,p^{-k}s})\subset A\widehat\otimes \widetilde{\B}_{\rig}^+$$
\end{lemma}
\begin{proof}
This is an $A$-linear analogue of~\cite[Lemme 3.1]{berger}.  We prove the first assertion here; with this in place, the proof of the second carries over verbatim from~\cite{berger}.  Note that the first assertion is an equality of topological $\Z_p$-modules inside $A\widehat\otimes\widetilde{\B}^+$, not algebras, because we do not know that there is an algebra norm on $A$ making it into an orthonormalizable $\Q_p$-Banach space.



Choose an orthonormal basis $\{e_i\}_{i\in I}$ of $A$.  Then we may compute $\cap_{k=0}^\infty p^{-hk}(\mathscr{A}\widehat\otimes\widetilde{\A}^{\dagger,p^{-k}s})$ inside $A\widehat\otimes\widetilde\B\xrightarrow{\sim} c_I(\B)$.
But if 
$$x=\sum_{i\in I}a_ie_i\in p^{-hk}(\mathscr{A}\widehat\otimes\widetilde{\A}^{\dagger,p^{-k}s})$$ 
for all $k$, then $a_i\in p^{-hk}\widetilde{\A}^{\dagger,p^{-k}s}$ for all $k$, implying that $x\in \mathscr{A}\widehat\otimes\widetilde{\A}^+$, as desired.
\end{proof}

\begin{remark}
The completed tensor product $\mathscr{A}\widehat\otimes\widetilde{\A}^+$ appearing in the first assertion of Lemma~\ref{arig-brigplus} is with respect to the weak topology on $\widetilde{\A}^+$, not with respect to the $p$-adic topology.
\end{remark}

\begin{cor}\label{arig-brigplus}
Let $A$ be a $\Q_p$-Banach algebra, equipped with an algebra norm $|\cdot|$, and let $\mathscr{A}$ be its valuation ring.  Then $\cap_{k=0}^\infty p^{-hk}(\mathscr{A}\widehat\otimes\widetilde{\A}_{\rig}^{\dagger,p^{-k}s})\subset A\widehat\otimes \widetilde{\B}_{\rig}^+$.
\end{cor}
\begin{proof}
By~\cite[Proposition 10.1]{schneider}, there is an equivalent norm $|\cdot|'$ on $A$ with respect to which $A$ is orthonormalizable; let $\mathscr{A}'$ be the unit ball with respect to $|\cdot|'$.  Then there exists a constant $c\geq0$ such that $p^c\mathscr{A}\subset \mathscr{A}'$, so that $\cap_{k=0}^\infty p^cp^{-hk}(\mathscr{A}\widehat\otimes\widetilde{\A}_{\rig}^{\dagger,p^{-k}s})\subset A\widehat\otimes \widetilde{\B}_{\rig}^+$.  But $p$ is invertible in $\widetilde{\B}_{\rig}^+$, so $\cap_{k=0}^\infty p^{-hk}(\mathscr{A}\widehat\otimes\widetilde{\A}_{\rig}^{\dagger,p^{-k}s})\subset A\widehat\otimes \widetilde{\B}_{\rig}^+$.
\end{proof}


As in~\cite{berger}, ``Frobenius regularization'' follows immediately from Corollary~\ref{arig-brigplus}.
\begin{prop}[{\cite[Proposition 3.2]{berger}}]\label{frob-reg}
Let $d_1$, $d_2$, and $h$ be three positive integers, and let $M\in\Mat_{d_2\times d_1}(A\widehat\otimes\widetilde{\B}_{\log}^\dagger)$ be a matrix.  Suppose there exists $P\in\GL_{d_1}(A\otimes_{\Q_p}F)$ such that $M=\varphi^{-h}(M)P$.  Then $M\in \Mat_{d_2\times d_1}(A\widehat\otimes\widetilde{\B}_{\log}^+)$.
\end{prop}

\begin{cor}\label{gl-frob-reg}
Let $d$ be a positive integer, and let $M\in\GL_{d}(A\widehat\otimes\widetilde{\B}_{\log}^\dagger[t])$ be an invertible matrix.  Suppose there exists $P\in\GL_{d}(A\otimes_{\Q_p}F)$ such that $M=\varphi(M)P$.  Then $M\in \GL_{d}(A\widehat\otimes\widetilde{\B}_{\log}^+[1/t])$.
\end{cor}

\begin{remark}
Proposition~\ref{frob-reg} is stated and proved in~\cite{berger} for $h=1$.  However, the proof carries over verbatim for $h>1$.
\end{remark}

\begin{prop}\label{dlog-dst-field}
Let $A$ be a discretely valued $\Q_p$-Banach field with perfect residue field, and let $V$ be a an $A$-vector space of dimension $d$, equipped with a continuous $A$-linear action of $\Gal_K$.  Then the natural map
\[	\left((A\widehat\otimes\B_{\st}^+)\otimes_AV\right)^{\Gal_K}\rightarrow \left((A\widehat\otimes\widetilde{\B}_{\log}^\dagger)\otimes_AV\right)^{\Gal_K}	\]
is an isomorphism.
\end{prop}
\begin{proof}


Recall that for $n\gg0$, there is an injection $i_n:\widetilde{\B}_{\log}^{\dagger,s_n}\rightarrow \B_{\dR}^+$, where $s_n=p^ns_0=p^{n-1}(p-1)$.  Then $i_n$ yields an injection $\left((A\widehat\otimes\widetilde{\B}_{\log}^{\dagger,s_n})\otimes_AV\right)^{\Gal_K}\rightarrow \D_{\dR}^K(V)$.  Note that $V$ admits a free $\Gal_K$-stable $\mathscr{A}$-submodule of rank $d$; under these hypotheses, we will show below (in a non-circular way) in Corollary~\ref{dht-ddr-finite} that $\D_{\dR}^K(V)$ is a finite $A\otimes_{\Q_p}K$-module. 
Therefore, $\left((A\widehat\otimes\widetilde{\B}_{\log}^{\dagger,s_n})\otimes_AV\right)^{\Gal_K}$ is a finite $A\otimes_{\Q_p}K_0$-module.  

Further, we claim that there is some $s_n$ such that 
\[	\left((A\widehat\otimes\widetilde{\B}_{\log}^{\dagger,s_n})\otimes_AV\right)^{\Gal_K}=\left((A\widehat\otimes\widetilde{\B}_{\log}^{\dagger})\otimes_AV\right)^{\Gal_K}	\]
Indeed, $A\otimes_{\Q_p}K_0\cong \prod_i A_i$, where the $A_i$ are a finite collection of $\Q_p$-Banach fields which are finite extensions of $A$ (and isomorphic to each other, because $K_0/\Q_p$ is Galois), so that $\D_{\dR}^K(V)\cong \oplus_i (A_i\otimes_{K_0}K)^{\oplus d_i}$ for some integers $d_i\geq 0$.  It follows that $\left((A\widehat\otimes\widetilde{\B}_{\log}^{\dagger,s_n})\otimes_AV\right)^{\Gal_K}$ is an $A$-vector space of dimension at most $\sum_i d_i[K:K_0]\dim_AA_i$ for any $n$, so the same is true of $\left((A\widehat\otimes\widetilde{\B}_{\log}^{\dagger})\otimes_AV\right)^{\Gal_K}$.  Therefore, 
\[	\left((A\widehat\otimes\widetilde{\B}_{\log}^{\dagger})\otimes_AV\right)^{\Gal_K}:=\bigcup_n \left((A\widehat\otimes\widetilde{\B}_{\log}^{\dagger,s_n})\otimes_AV\right)^{\Gal_K}	\] 
is a finite module over the Noetherian ring $A$, so the conclusion follows.

Now let $D:=\left((A\widehat\otimes\widetilde{\B}_{\log}^{\dagger})\otimes_AV\right)^{\Gal_K}$, let $D_i:=D\otimes_AA_i$ be the factor of $D$ over $A_i$, let $v_1,\ldots, v_d$ be an $A$-basis of $V$, and let $w_1,\ldots,w_{d'}$ be an $A_i$-basis of $D_i$.  Then $v_1,\ldots,v_d$ is an $A\widehat\otimes\widetilde{\B}^{\dagger}$-basis of $(A\widehat\otimes\widetilde{\B}_{\log}^{\dagger})\otimes_AV$ and $w_j\in (A\widehat\otimes\widetilde{\B}_{\log}^{\dagger})\otimes_AV$, so  there is a matrix $M\in \Mat_{d\times d'}(A\widehat\otimes\widetilde{\B}_{\log}^\dagger)$ whose $j$th column is the coordinates of $w_j$ with respect to $v_1,\ldots,v_d$.  Let $P\in\GL_{d'}(A_i)$ be the matrix of $\varphi^{[K_0:\Q_p]}$ with respect to $w_1,\ldots,w_{d'}$.  To justify this, recall that $\varphi:\widetilde{\B}_{\log}^\dagger\rightarrow\widetilde{\B}_{\log}^\dagger$ is a bijection, and note that $\varphi$ cyclically permutes the $D_i$ so that $\varphi^{[K_0:\Q_p]}$ carries $D_i$ to itself.  Then $MP=\varphi^{[K_0:\Q_p]}(M)$, since $\varphi$ acts trivially on $v_1,\ldots,v_d$, so that $M=\varphi^{-[K_0:\Q_p]}(M)\varphi^{-[K_0:\Q_p]}(P)$.  Then by Frobenius regularization, $M$ has coefficients in $A\widehat\otimes\widetilde{\B}_{\log}^+\subset A\widehat\otimes\widetilde{\B}_{\st}^+$, so we are done.
\end{proof}

\begin{remark}
The conclusion of Proposition~\ref{dlog-dst-field} is used in the proof of \cite[Proposition 6.2.4]{bc}.  Since the proof requires some minor adjustments when $A$ is not $\Q_p$-finite, we have written out the details here.
\end{remark}

We can deduce the same result for Galois representations with affinoid coefficients, generalizing~\cite{berger}.
\begin{cor}\label{dlog-dst}
Let $A$ be an $E$-affinoid algebra and let $V$ be a finite free $A$-module of rank $d$ equipped with a continuous action of $\Gal_K$.  Then the natural map
$$\left((A\widehat\otimes\B_{\st}^+)\otimes_AV\right)^{\Gal_K}\rightarrow \left((A\widehat\otimes\widetilde{\B}_{\log}^\dagger)\otimes_AV\right)^{\Gal_K}$$
is an isomorphism.
\end{cor}
\begin{proof}
Let $A\rightarrow R=\prod_iR_i$ be a closed embedding into a finite product of artin rings, with $R_i$ a finite dimensional vector space over a complete discretely valued field $B_i$ with perfect residue field; this is possible by Lemma~\ref{embed}.  Then we have an exact sequence of $\Q_p$-Banach spaces $0\rightarrow V\rightarrow V_R\rightarrow V_R/V\rightarrow 0$.  Since $\Q_p$ is discretely valued, this exact sequence admits a continuous $\Q_p$-linear splitting, and we have a commutative diagram of Fr\'echet spaces
$$\begin{CD}
0 @>>>(A\widehat\otimes_{\Q_p}\widetilde{\B}_{\log}^\dagger)\otimes_AV @>>> (R\widehat\otimes_{\Q_p}\widetilde{\B}_{\log}^\dagger)\otimes_RV_R @>>> \widetilde{\B}_{\log}^\dagger\widehat\otimes_{\Q_p}(V_R/V) @>>> 0	\\
@.	@AAA	@AAA	@AAA	@.	\\
0 @>>>(A\widehat\otimes_{\Q_p}\widetilde{\B}_{\st}^+)\otimes_AV @>>> (R\widehat\otimes_{\Q_p}\widetilde{\B}_{\st}^+)\otimes_RV_R @>>> \widetilde{\B}_{\st}^+\widehat\otimes_{\Q_p}(V_R/V) @>>> 0
\end{CD}$$
where the rows are exact and the vertical maps are injections.  Moreover, the maps are $\Gal_K$-equivariant, so we have a commutative diagram of Banach spaces
$$\begin{CD}
0 @>>>\left(\widetilde{\B}_{\log}^\dagger\widehat\otimes_{\Q_p}V\right)^{\Gal_K} @>>> \left(\widetilde{\B}_{\log}^\dagger\widehat\otimes_{\Q_p}V_R\right)^{\Gal_K} @>>> \left(\widetilde{\B}_{\log}^\dagger\widehat\otimes_{\Q_p}(V_R/V)\right)^{\Gal_K}	\\
@.	@AAA	@AAA	@AAA	@.	\\
0 @>>>\left(\widetilde{\B}_{\st}^+\widehat\otimes_{\Q_p}V\right)^{\Gal_K} @>>> \left(\widetilde{\B}_{\st}^+\widehat\otimes_{\Q_p}V_R\right)^{\Gal_K} @>>> \left(\widetilde{\B}_{\st}^+\widehat\otimes_{\Q_p}(V_R/V)\right)^{\Gal_K}
\end{CD}$$
where the rows are still exact and the vertical maps are still injections.  For each idempotent factor $R_i$ of $R$, we can view $V_{R_i}$ as a finite-dimensional $B_i$-vector space and apply Proposition~\ref{dlog-dst-field}; we see that the inclusion $\left((R\widehat\otimes\widetilde{\B}_{\st}^+)\otimes_RV_R\right)^{\Gal_K}\subset\left((R\widehat\otimes\widetilde{\B}_{\log}^\dagger)\otimes_RV_R\right)^{\Gal_K}$ is an equality.  Then a diagram chase shows that the inclusion 
$\left((A\widehat\otimes\widetilde{\B}_{\st}^+)\otimes_AV\right)^{\Gal_K}\subset\left((A\widehat\otimes\widetilde{\B}_{\log}^\dagger)\otimes_AV\right)^{\Gal_K}$
is an equality, as well.
\end{proof}

We now define 
\[	\D_{\log,K}^{\dagger,s}(V):=(A\widehat\otimes\B_{\log,K}^{\dagger,s})\otimes_{A\widehat\otimes\B_{\rig,K}^{\dagger,s}}\D_{\rig,K}^{\dagger,s}(V)\text{ and }\D_{\log,K}^{\dagger}(V):=\cup_s\D_{\rig,K}^{\dagger,s}(V)	\]
as well as 
\[	\widetilde{\D}_{\log,K}^{\dagger,s}(V):=(A\widehat\otimes\widetilde{\B}_{\log,K}^{\dagger,s})\otimes_{A\widehat\otimes\B_{\log,K}^{\dagger,s}}\D_{\log,K}^{\dagger,s}(V)\text{ and }\widetilde{\D}_{\log,K}^{\dagger}(V):=\cup_s\widetilde{\D}_{\log,K}^{\dagger,s}(V)	\]
\begin{prop}\label{dlog-dlog}
Let $A$ be a Noetherian $\Q_p$-Banach algebra with valuation ring $\mathscr{A}$, and let $V$ be a finite free $A$-module of rank $d$ equipped with a continuous $A$-linear action of $\Gal_K$ such that $V$ admits a free $\Gal_K$-stable $\mathscr{A}$-submodule of rank $d$.  Then the natural map 
\[	\left(\D_{\log,K}^{\dagger}(V)[1/t]\right)^{\Gamma_K}\rightarrow\left((A\widehat\otimes\widetilde{\B}_{\log}^\dagger)[1/t]\otimes_AV\right)^{\Gal_K}	\]
is an isomorphism.
\end{prop}
\begin{proof}
It suffices to prove this with $K$ replaced by a finite extension, so we may assume $\D_{\log,K}^\dagger(V)$ is free.  After twisting $V$ by some power of the cyclotomic character, we may assume that $\left(\D_{\log,K}^{\dagger}(V)[1/t]\right)^{\Gamma_K}=\left(\D_{\log,K}^{\dagger}(V)\right)^{\Gamma_K}$ and consider only the map 
$$\left(\D_{\log,K}^{\dagger}(V)\right)^{\Gamma_K}\rightarrow\left((A\widehat\otimes\widetilde{\B}_{\log,K}^\dagger)\otimes_AV\right)^{\Gal_K}$$
Furthermore, we observe that 
\begin{align*}
\left((A\widehat\otimes\widetilde{\B}_{\log}^\dagger)\otimes_AV\right)^{H_K} &= (A\widehat\otimes\widetilde{\B}_{\log,K}^\dagger)\otimes_{A\widehat\otimes\B_{\log,K}^\dagger}\D_{\log,K}^{\dagger}(V)	\\
&= \widetilde\D_{\log,K}^\dagger(V)
\end{align*}
Since $\left(\widetilde{\D}_{\log,K}^{\dagger}(V)\right)^{\Gamma_K}$ and $\left(\D_{\log,K}^{\dagger}(V)\right)^{\Gamma_K}$ are finite modules over the Noetherian Banach algebra $A\otimes_{\Q_p}K_0$, we see that $\left(\D_{\log,K}^{\dagger}(V)\right)^{\Gamma_K}$ is a closed submodule of $\left(\widetilde{\D}_{\log,K}^{\dagger}(V)\right)^{\Gamma_K}$, by~\cite[Prop. 3.7.3/1]{bgr}.  Thus, it suffices to show that $\left(\D_{\log,K}^{\dagger}(V)\right)^{\Gamma_K}$ is dense in $\left(\widetilde{\D}_{\log,K}^{\dagger}(V)\right)^{\Gamma_K}$.  

We will actually do something slightly different.  For $s\gg0$ and any integer $k\geq 0$, we consider the $A\widehat\otimes\varphi^{-k}(\B_{\log,K}^{\dagger,p^ks})$-submodule $\varphi^{-k}(\D_{\log,K}^{\dagger,p^ks}(V))\subset \widetilde{\D}_{\log,K}^{\dagger,s}(V)$.  Since $\left(\widetilde{\D}_{\log,K}^{\dagger,s}(V)\right)^{\Gamma_K}$ is a finite $A$-module, it follows that $\left(\varphi^{-k}(\D_{\log,K}^{\dagger,p^ks}(V))\right)^{\Gamma_K}$ is a closed submodule of $\left(\widetilde{\D}_{\log,K}^{\dagger,s}(V)\right)^{\Gamma_K}$.  

We claim that $\bigcup_k \left(\varphi^{-k}(\D_{\log,K}^{\dagger,p^ks}(V))\right)^{\Gamma_K}= \left(\widetilde{\D}_{\log,K}^{\dagger,s}(V)\right)^{\Gamma_K}$.  If we choose a basis $\ve{v}_1,\ldots,\ve{v}_d$ of $\D_{\log,K}^{\dagger,s}(V)$, then for any $\Gamma_K$-fixed element $m\in\widetilde{\D}_{\log,K}^{\dagger,s}(V)$, we may write $m=a_1\ve{v}_1+\cdots+a_d\ve{v}_d$.  Recall that there are $\Gamma_K$-equivariant maps 
$$R_k:\widetilde{\B}_{\rig,K}^{\dagger,s}\rightarrow \varphi^{-k}(\B_{\rig,K}^{\dagger,p^ks})$$ 
which are sections to the inclusions $\varphi^{-k}(\B_{\rig,K}^{\dagger,p^ks})\subset \widetilde{\B}_{\rig,K}^{\dagger,s}$, and extend to maps 
$$R_k:\widetilde{\B}_{\log,K}^{\dagger,s}\rightarrow \varphi^{-k}(\B_{\log,K}^{\dagger,p^ks})$$
For each $k$, let $m_k=R_k(m)=R_k(a_1)\ve{v}_1+\cdots+R_k(a_d)\ve{v}_d$.  Then $m_k$ is a $\Gamma_K$-fixed element of $\varphi^{-k}(\D_{\log,K}^{\dagger,p^ks}(V))$, because $R_k$ is $\Gamma_K$-equivariant.  Since $\lim_{k\rightarrow\infty}R_k(a)=a$ for any $a\in\widetilde{\B}_{\log,K}^{\dagger,s}$, it follows that $\lim_{k\rightarrow \infty}m_k=m$.
Thus, $\bigcup_k \left(\varphi^{-k}(\D_{\log,K}^{\dagger,p^ks}(V))\right)^{\Gamma_K}$ is dense in $\left(\widetilde{\D}_{\log,K}^{\dagger,s}(V)\right)^{\Gamma_K}$.  Since it is also closed in $\left(\widetilde{\D}_{\log,K}^{\dagger,s}(V)\right)^{\Gamma_K}$ (as it is a submodule of a finite module over a Noetherian Banach algebra), equality follows.  

Next, we note that 
\[	\varphi^{k+1}\left(\varphi^{-k}(\D_{\log,K}^{\dagger,p^ks}(V))\right)\subset \varphi(\D_{\log,K}^{\dagger,p^ks}(V))\subset \D_{\log,K}^{\dagger,p^{k+1}s}(V)	\]
This implies that $\left(\varphi^{-k}(\D_{\log,K}^{\dagger,p^ks}(V))\right)^{\Gamma_K}\subset \left(\varphi^{-(k+1)}(\D_{\log,K}^{\dagger,p^{k+1}s}(V))\right)^{\Gamma_K}$, and therefore that
$\left(\widetilde{\D}_{\log,K}^{\dagger,s}(V)\right)^{\Gamma_K} = \bigcup_k \left(\varphi^{-k}(\D_{\log,K}^{\dagger,p^ks}(V))\right)^{\Gamma_K}$
is a rising union.  Since $\left(\widetilde{\D}_{\log,K}^{\dagger,s}(V)\right)^{\Gamma_K}$ is $A$-finite, there is some $k$ such that $\left(\varphi^{-k}(\D_{\log,K}^{\dagger,p^ks}(V))\right)^{\Gamma_K}= \left(\widetilde{\D}_{\log,K}^{\dagger,s}(V)\right)^{\Gamma_K}$.  

But we have $A$-linear isomorphisms
$\varphi^k:\left(\varphi^{-k}(\D_{\log,K}^{\dagger,p^ks}(V))\right)^{\Gamma_K}\rightarrow \left(\D_{\log,K}^{\dagger,p^ks}(V)\right)^{\Gamma_K}$ and
$\varphi^k\!:\!\widetilde{\D}_{\log,K}^{\dagger,s}(V)\xrightarrow{\sim}\!\widetilde{\D}_{\log,K}^{\dagger,p^ks}(V)$, so we conclude that $\left(\D_{\log,K}^{\dagger,p^ks}(V)\right)^{\Gamma_K}\!\!\!=\! \left(\widetilde{\D}_{\log,K}^{\dagger,p^ks}(V)\right)^{\Gamma_K}$, as desired.
\end{proof}

Now we can prove Theorem~\ref{dcris-phi-gamma}.
\begin{proof}[Proof of Theorem~\ref{dcris-phi-gamma}]
We may assume that $X=\Sp(A)$ for some $E$-affinoid algebra $A$, and that $V:=\mathscr{V}(A)$ is $A$-free of rank $d$ and admits a $\Gal_K$-stable integral lattice.  Then $\D_{\st}^K(V)=\left(\D_{\log,K}(V)[1/t]\right)^{\Gamma_K}$ by Corollary~\ref{dlog-dst} and Proposition~\ref{dlog-dlog}.  Since 
$$\D_{\cris}(V)=\D_{\st}(V)^{N=0}\text{ and }\D_{\rig,K}(V)=\D_{\log,K}(V)^{N=0}$$
it follows that $\D_{\cris}^K(V)=\left(\D_{\rig,K}(V)[1/t]\right)^{\Gamma_K}$.
\end{proof}

\subsection{Properties of $\D_{\B_\ast}(V)$}\label{properties}

Now we can combine Theorems~\ref{dsen-dht} and~\ref{ddif-ddr} with ``cohomology and base change'' to deduce various useful properties of the functors $V\mapsto \D_{\HT}(V)$ and $V\mapsto \D_{\dR}(V)$.
\begin{thm}
Let $X$ and $\mathscr{V}$ be as above.  Then 
\begin{enumerate}
\item	$\mathscr{D}_{\HT}^K(\mathscr{V})$ and $\mathscr{D}_{\dR}^K(\mathscr{V})$ are coherent sheaves of $\mathscr{O}_X\otimes_{\Q_p}K$-modules.  More generally, their formation commutes with flat base change on $X$.
\item	$\mathscr{D}_{\HT}^K(\mathscr{V})$ and $\mathscr{D}_{\dR}^K(\mathscr{V})$ take values in the categories of graded coherent sheaves and filtered coherent sheaves, respectively.  If $\mathscr{V}$ is $\B_{\HT}$-admissible, then $\mathscr{D}_{\HT}^K(\mathscr{V})$ is a graded vector bundle over $\mathscr{O}_X\otimes_{\Q_p}K$, and if $\mathscr{V}$ is $\B_{\dR}$-admissible, then $\mathscr{D}_{\HT}^K(\mathscr{V})$ is a filtered vector bundle over $\mathscr{O}_X\otimes_{\Q_p}K$.
\end{enumerate}
\end{thm}

As before, we reduce immediately to the case when $X=\Sp(A)$ and $V:=\Gamma(X,\mathscr{V})$ is a free $A$-linear representation of $\Gal_K$ which admits a free $\Gal_K$-stable $\mathscr{A}$-lattice $V_0$ of rank $d$.

\begin{prop}\label{dht-ddr-finite}
Let $A$ be a Noetherian $\Q_p$-Banach algebra with valuation ring $\mathscr{A}$, let $V_0$ be a free $\mathscr{A}$-module of rank $d$ equipped with a continuous $\mathscr{A}$-linear action of $\Gal_K$, and let $V:=V_0[1/p]$.  Then $\D_{\HT}^K(V)$ and $\D_{\dR}^K(V)$ are finite $A\otimes_{\Q_p}K$-modules.
\end{prop}
\begin{proof}
Recall that 
\[	\D_{\HT}^K(V) = \oplus_{i\in \Z}\left(\D_{\Sen}^K(V)\cdot t^i\right)^{\Gamma_K}	\]
Now $\left(\D_{\Sen}^K(V)\cdot t^i\right)^{\Gamma_K} = \left(\D_{\Sen}^K(V)\right)^{\Gamma_K=\chi^{-i}}$ for every $i\in\Z$, so $\D_{\HT}^K(V)\subset \D_{\Sen}^K(V)$.  But $\D_{\Sen}^K(V)$ is a finite module over the Noetherian ring $A\otimes_{\Q_p} K$, so $\D_{\HT}^K(V)$ is $A$-finite as well.

Moreover, we observe that the summands of $\D_{\HT}^K(V)$ have pairwise trivial intersection.  Therefore, only finitely many of them are non-zero.

To see that $\D_{\dR}^K(V)$ is finite over $A\otimes_{\Q_p}K$, we observe that 
\[	\gr^\bullet\D_{\dR}^K(V)\hookrightarrow \left(\gr^\bullet((A\widehat\otimes\B_{\dR})\otimes_AV)\right)^{\Gal_K} = \D_{\HT}^K(V)	\]
In fact, we claim that there exist integers $i_0, i_1$ such that $\Fil^i\D_{\dR}^K(V)=\Fil^{i_0}\D_{\dR}^K(V)$ for all $i\leq i_0$ and $\Fil^i\D_{\dR}^K(V)=\Fil^{i_1}\D_{\dR}^K(V)=0$ for all $i\geq i_1$.  Indeed, 
$$\gr^i\D_{\dR}^K(V)\hookrightarrow \left(\gr^i((A\widehat\otimes\B_{\dR})\otimes_AV)\right)^{\Gal_K} = \left(\D_{\Sen}^K(V)\cdot t^i\right)^{\Gal_K}$$
But the right-most term is one of the summands of $\D_{\HT}^K(V)$, and only finitely many such summands such summands are non-zero.  Therefore, $\gr^i\D_{\dR}^K(V)=0$ for $i\ll0$ and $i\gg0$, and $\D_{\dR}^K(V)$ is $A\otimes_{\Q_p}K$-finite.
\end{proof}

\begin{definition}
The \emph{Hodge--Tate weights} of $V$ are those integers $i$ such that 
$$\left(\D_{\Sen}^K(V)\cdot t^i\right)^{\Gal_K}\neq0$$
\end{definition}

\begin{remark}
This is a slight departure from the traditional definition of Hodge--Tate weights, which are usually only defined for representations which are Hodge--Tate.  However, we will find this abuse of terminology convenient.
\end{remark}

\begin{prop}\label{h1-finite}
Suppose that $V$ is a free $A$-module of rank $d$, equipped with a continuous, $A$-linear action of $G_K$.  Then $\H^1(\Gamma_K,\oplus_{k\in\Z}t^k\D_{\Sen}(V))$ and $\H^1(\Gamma_K,\D_{\dif}(V))$ are $A$-finite if and only if there is some interval $[a,b]$ such that the Hodge--Tate weights of the fibral representations all lie in the interval $[a,b]$.
\end{prop}
\begin{proof}
We first reduce to the case where $\Gamma_K$ is procyclic.  In general, $\Gamma_K\cong \Delta\times\Gamma_K'$, where $\Delta$ is a finite abelian group and $\Gamma_K'$ is procyclic.  The statement about the fibral Hodge--Tate weights can be checked after restriction to a finite index subgroup of $\Gamma_K$, and in particular, after restriction to $\Gamma_K'$.  On the other hand, taking $\Delta$-invariants on $\Q_p$-vector spaces is an exact functor, so $\H^1(\Gamma_K',\oplus_{k\in\Z}t^k\D_{\Sen}(V))=\H^1(\Gamma_K,\oplus_{k\in\Z}t^k\D_{\Sen}(V))^{\Delta}$ and $\H^1(\Gamma_K',\D_{\dif}(V))=\H^1(\Gamma_K,\D_{\dif}(V))^{\Delta}$.  It follows that we can also check the finiteness of $\H^1(\Gamma_K,\oplus_{k\in\Z}t^k\D_{\Sen}(V))$ and $\H^1(\Gamma_K,\D_{\dif}(V))$ can after restriction to $\Gamma_K'$.  We may therefore assume that $\Gamma_K$ is torsion-free and apply the results of section~\ref{cohomology}.

The statement is clear for $\H^1(\Gamma_K,\oplus_{k\in\Z}t^k\D_{\Sen}(V))$.  

Suppose first that the fibral Hodge--Tate weights are bounded in an interval $[a,b]$.  The natural map 
\[	\H^1(\Gamma_K,\D_{\dif}^{+}(V)/t^{k+1})\rightarrow \H^1(\Gamma_K,\D_{\dif}^{+}(V)/t^k)	\]
is a surjection for all $k\geq 0$, and its kernel is surjected onto by $\H^1(\Gamma_K,t^k\D_{\Sen}(V))$.  But the formation of $\H^1(\Gamma_K,t^k\D_{\Sen}(V))$ commutes with arbitrary base-change on $A$, so the hypothesis on the fibral Hodge--Tate weights implies that if $k> b$, $\H^1(\Gamma_K,t^k\D_{\Sen}(V))$ is trivial when reduced modulo any power of any maximal ideal of $A$.  Therefore, $\H^1(\Gamma_K,t^k\D_{\Sen}(V))$ is itself trivial, and $\H^1(\Gamma_K,\D_{\dif}^{+}(V))\cong \H^1(\Gamma_K,\D_{\dif}^{+}(V)/t^{\max{\{0,b\}}+1})$, which is $A$-finite.  This implies that for any $k\geq 0$, $\H^1(\Gamma_K,t^{-k}\D_{\dif}^+(V))\cong \H^1(\Gamma_K,t^{-k}\D_{\dif}^+(V)/t^{\max{\{0,b\}}+1})$ is $A$-finite, as well.  Further, the proof of Corollary~\ref{h0-stab-cor}(3) shows that for any $k\in\Z$, the cokernel of the natural map 
\[	\H^1(\Gamma_K,t^{-k}\D_{\dif}^{+}(V))\rightarrow \H^1(\Gamma_K,t^{-(k+1)}\D_{\dif}^{+}(V))	\]
is $\H^1(\Gamma_K,t^{-(k+1)}\D_{\Sen}(V))$.  But the hypothesis on the fibral Hodge--Tate weights implies that this is $0$ for $k\geq -a$, so $\H^1(\Gamma_K,\D_{\dif}(V))$ is $A$-finite.

Conversely, suppose $\H^1(\Gamma_K,\D_{\dif}(V))$ is $A$-finite.  We need to show that $\H^1(\Gamma_K,t^k\cdot \D_{\Sen}(V))=0$ for $k\gg0$ and $k\ll0$.   By Corollary~\ref{h0-stab-cor}(3), there exist $N_0, N_0'\in\Z$ such that for $k\geq N_0$ or $k\leq N_0'$, the transition maps $\H^1(\Gamma_K,t^{k+1}\D_{\dif}^+(V))\rightarrow\H^1(\Gamma_K,t^{k}\D_{\dif}^+(V))$ are injective.  These transition maps moreover always have $A$-finite kernels and cokernels.  Since $A$ is Noetherian, this implies that $\H^1(\Gamma_K,t^{k}\D_{\dif}^+(V))$ is finite for all $k\in\Z$.   
%

Let $x\in\Sp(A)$, and let $\kappa(x)$ denote the residue field of $A$ at $x$.  By Proposition~\ref{h1-frechet}, 
\[	\H^1(\Gamma_K,t^k\D_{\dif}^+(V))\otimes_A\kappa(x)\cong \H^1(\Gamma_K,t^k\D_{\dif}^+(V)\widehat\otimes_A\kappa(x))	\]
for all $k\geq 0$; it follows that there is some $N_{1,x}\geq 0$ such that $\H^1(\Gamma_K,t^k\D_{\dif}^+(V))\otimes_A\kappa(x)=0$ for all $k\geq 0$.  Since $\H^1(\Gamma_K,t^k\D_{\dif}^+(V))$ is a finite $A$-module for all $k$, there is some Zariski open $U_x\subset \Sp(A)$ such that $\H^1(\Gamma_K,t^k\D_{\dif}^+(V))|_{U_x}=0$ for all $k\geq N_{1,x}$.  Since $\Spec A$ is quasi-compact, it follows that there is some $N_1\gg0$ such that $\H^1(\Gamma_K,t^k\D_{\dif}^+(V))=0$ for all $k\geq N_1$.  Thus, $\H^1(\Gamma_K,t^k\D_{\Sen}(V))=0$ for $k\geq N_1$.

Finally, the finiteness of $\H^1(\Gamma_K,\D_{\dif}(V))$ implies that there exists $N_1'\in \Z$ such that for $k\leq N_0'$, the transition map $\H^1(\Gamma_K,t^{k+1}\D_{\dif}^+(V))\rightarrow\H^1(\Gamma_K,t^{k}\D_{\dif}^+(V))$ has vanishing cokernel.  But this cokernel is $\H^1(\Gamma_K,t^k\D_{\Sen}(V))$, so it follows that the fibral Hodge--Tate weights are bounded.
\end{proof}

Now we can deduce that $\D_{\cris}^K(V)$ and $\D_{\st}^K(V)$ are finite modules, as well.

\begin{cor}\label{dcris-dst-finite}
$\D_{\cris}^K(V)$ and $\D_{\st}^K(V)$ are finite $A\otimes_{\Q_p}K_0$-modules.
\end{cor}
\begin{proof}
Recall that there is an injection $\B_{\max}\rightarrow \B_{\dR}$.  Since $A$ is a Banach space over the discretely valued field $\Q_p$ (and therefore potentially orthonormalizable), this extends to an injection $A\widehat\otimes\B_{\max}\hookrightarrow A\widehat\otimes\B_{\dR}$.  It follows that $\D_{\cris}^{K}(V)\hookrightarrow \D_{\dR}^{K}(V)$ and $\D_{\cris}^{K}(V)$ is $A$-finite.

Similarly, $\B_{\st}$ can be injected into $\B_{\dR}$ (although this depends on a choice of $p$-adic logarithm), so $\D_{\st}^K(V)\hookrightarrow \D_{\dR}^K(V)$.  Thus, $\D_{\st}^K(V)$ is $A$-finite.
\end{proof}

Note that $\D_{\cris}^K(V)$ and $\D_{\st}^K(V)$ are equipped with semilinear actions of Frobenius $\varphi$ (over $1\otimes\varphi$ on $A\otimes_{\Q_p}K_0$) coming from the coefficients, and $\D_{\st}^K(V)$ has a monodromy operator $N$ coming from the coefficients and satisfying $N\circ\varphi=p\varphi\circ N$.

We turn to base change properties of the functors $\D_{\B_\ast}(V)$.

\begin{prop}
Let $f:A\rightarrow A'$ be a flat morphism of $E$-affinoid algebras.  Then 
\begin{enumerate}
\item	$A'\otimes_A\D_{\HT}^K(V)\xrightarrow{\sim}\D_{\HT}^K(V\otimes_AA')$
\item	$A'\otimes_A\D_{\dR}^K(V)\xrightarrow{\sim}\D_{\dR}^K(V\otimes_AA')$.
\end{enumerate}
It follows that $U\mapsto\D_{\HT}^K(V_U)$ and $U\mapsto\D_{\dR}^K(V_U)$ are coherent sheaves on $\Sp(A)$.
\end{prop}
\begin{proof}
\begin{enumerate}
\item	This follows by noting that $\D_{\HT}^{L_n}(V)=\varinjlim_{h\rightarrow\infty} \left(\oplus_{k=-h}^ht^k\D_{\Sen}^{L_n}(V)\right)^{\Gamma_{L_n}}$, and the formation of $\left(\oplus_{k=-h}^ht^k\D_{\Sen}^{L_n}(V)\right)^{\Gamma_{L_n}}$ commutes with flat base change, by Theorem~\ref{base-change}.  Since $\D_{\HT}^{L_n}(V)=L_n\otimes_K\D_{\HT}^K(V)$, we are done.
\item	We apply Proposition~\ref{h0-stab} to $M=\D_{\dif}^{L_n,+}$ to see that 
$$\D_{\dR}^{L_n,+}(V)=\left(\D_{\dif}^{L_n,+}(V)/t^N\right)^{\Gamma_{L_n}=1}$$ for some $N\geq 0$.  Similarly, $\D_{\dR}^{L_n,+}(V\otimes_AA')=\left(\D_{\dif}^{L_n,+}(V\otimes_AA')/t^{N'}\right)^{\Gamma_{L_n}=1}$.  Since the modules $\D_{\dif}^{L_n,+}(V)/t^k$ are finite $A$-modules, 
\begin{eqnarray*}
\D_{\dR}^{L_n,+}(V\otimes_AA')&=&\H^0(\Gamma_{L_n},\D_{\dif}^{L_n,+}(V\otimes_AA')/t^{\max\{N,N'\}})	\\
&=&\H^0(\Gamma_{L_n},\D_{\dif}^{L_n,+}(V)/t^{\max\{N,N'\}})\otimes_AA'=\D_{\dR}^{L_n,+}(V)\otimes_AA'
\end{eqnarray*}
where the second equality again follows from Theorem~\ref{base-change}. 
\end{enumerate}
\end{proof}

\begin{conjecture}\label{dcris-sheaf}
The formation of $\D_{\cris}^K(V)$ and $\D_{\st}^K(V)$ commutes with flat base change on $A$.
\end{conjecture}

\begin{remark}
Conjecture~\ref{dcris-sheaf} does not follow automatically from Theorem~\ref{base-change}, because the $(\varphi,\Gamma)$-module $\D_{\rig,K}^\dagger(V)$ is not $A$-finite.  It is also difficult, in general, to verify this conjecture for particular examples of families of Galois representations.

However, life is considerably better when considering a \emph{trianguline} family, because one can compute with rank-$1$ families.  
The papers \cite{hellmann2}, \cite{kpx}, and \cite{liu} prove triangulation results for certain families of Galois representations arising from eigenvarieties.  In particular, in~\cite[Theorem 6.3.9]{kpx}, the authors show that if $\mathscr{V}$ is the family of Galois representations on the normalization $X$ of the eigencurve, then away from the image of the $\theta^{k-1}$-map, the associated family $\mathscr{D}_{\rig}^\dagger(\mathscr{V})$ has a global triangulation
\[	0\rightarrow \mathscr{D}_1\rightarrow \mathscr{D}_{\rig}^\dagger(\mathscr{V})\rightarrow \mathscr{D}_2\rightarrow 0	\]
Here the $\mathscr{D}_i=\mathcal{R}_X(\delta_i)\otimes_X\mathscr{L}_i$ are rank-$1$ families of $(\varphi,\Gamma)$-modules over $X$.  We have used the notation of~\cite{kpx}: $\delta_i:\Q_p^\times\rightarrow \Gamma(X,\mathscr{O}_X^\times)$ are continuous characters, $\mathcal{R}_X(\delta_i)$ is the free rank-$1$ $(\varphi,\Gamma_{\Q_p})$-module with basis $\mathbf{e}$ such that $\varphi(\mathbf{e})=\delta(p)\mathbf{e}$ and $\gamma(\mathbf{e})=\delta(\chi(\gamma))\mathbf{e}$ for $\gamma\in\Gamma_{\Q_p}$, and $\mathscr{L}_i$ are line bundles on $X$ with no action of $\varphi$ or $\Gamma_{\Q_p}$.  

In this case, $\delta_1|_{\Z_p^\times}$ is trivial, while $\delta_2|_{\Z_p^\times}$ is the weight-nebentypus character.  As a result, for any affinoid $U\subset X$ which trivializes $\mathscr{L}_2$, $\mathscr{D}_2(U)[1/t]^{\Gamma_{\Q_p}}=0$ and so by Theorem~\ref{dcris-phi-gamma}
\[	\mathscr{D}_{\cris}(\mathscr{V})(U)=\mathscr{D}_1(U)[1/t]^{\Gamma_{\Q_p}}	\]
Moreover, by construction 
\[	\mathscr{D}_1(U)[1/t]^{\Gamma_{\Q_p}} = \mathscr{D}_1(U)[1/t]^{\varphi=\delta(p), \Gamma_{\Q_p}=1} = \mathscr{D}_{\cris}(\mathscr{V})(U)^{\varphi=\delta(p)}	\]
It follows from \cite{kpx} or \cite{liu} that $\mathscr{D}_{\cris}(\mathscr{V})^{\varphi=\delta(p)}$ is a coherent $\mathscr{O}_X$-module, and hence so is $\mathscr{D}_{\cris}(\mathscr{V})$.  This is a very natural example of a family $\mathscr{V}$ of Galois representations such that $\mathscr{D}_{\cris}(\mathscr{V})$ is a coherent sheaf.
\end{remark}


\section{$\B_\ast$-admissible loci}\label{b-adm-loci}

\subsection{Overview}

In this section, we fix a family $\mathscr{V}$ of rank-$d$ representations of $\Gal_K$ over an $E$-analytic space $X$, and we study the loci on $X$ where $\mathscr{V}$ is $\B_\ast$-admissible for various period rings $\B_\ast$.  We have the following theorem.

\begin{thm}\label{adm-locus}
Let $X$ and $\mathscr{V}$ be as above, and let $\ast\in\{\HT,\dR,\st,\cris\}$.  Then there is a closed subspace $X_{\B_\ast}^{[a,b]}\hookrightarrow X$ such that for any $E$-finite artin local ring $B$, a map $x:\Sp(B)\rightarrow X$ factors through $X_{\B_\ast}^{[a,b]}$ if and only if the induced $B$-linear Galois representation $\mathscr{V}_x$ is $\B_\ast$-admissible with Hodge--Tate weights in the interval $[a,b]$.
\end{thm}

If the family $\mathscr{V}$ is $\B_\ast$-admissible, then certainly for every morphism $f:X'\rightarrow X$, the base change $f^\ast\mathscr{V}$ is $\B_\ast$-admissible.  We prove a converse theorem in two parts.

\begin{thm}\label{pointwise-adm}
Let $X$ and $\mathscr{V}$ be as above, and let $\ast\in\{\HT,\dR,\st,\cris\}$.  Suppose that $\mathscr{V}_x$ is $\B_\ast$-admissible with Hodge--Tate weights in the interval $[a,b]$ for every morphism $x:\Sp(B)\rightarrow X$, where $B$ is an $E$-finite artin local ring.  Then
\begin{enumerate}
\item	the sheaf $\mathscr{D}_{\B_\ast}^K(\mathscr{V})$ is a sheaf of projective $\mathscr{O}_X\otimes_{\Q_p}\B_\ast^{\Gal_K}$-modules of rank $d$, and
\item	the formation of $\mathscr{D}_{\B_\ast}^K(\mathscr{V})$ commutes with arbitrary base change on $X$.
\end{enumerate}
\end{thm}

In each case, we then use the base change property to finish proving that $\mathscr{V}$ is a $\B_\ast$-admissible family of Galois representations:
\begin{thm}\label{comp-isom}
Let $X$ and $\mathscr{V}$ be as above, and let $\ast\in\{\HT,\dR,\st,\cris\}$.  Suppose that $\mathscr{V}_x$ is $\B_\ast$-admissible with Hodge--Tate weights in the interval $[a,b]$ for every morphism $x:\Sp(B)\rightarrow X$, where $B$ is an $E$-finite artin local ring.  Then the natural map $\mathscr{B}_{X,\ast}\otimes_{\mathscr{O}_X\otimes\B_\ast^{\Gal_K}}\mathscr{D}_{\B_\ast}^K(\mathscr{V})\rightarrow \mathscr{B}_{X,\ast}\otimes_{\mathscr{O}_X}\mathscr{V}$ is an isomorphism.
\end{thm}

\begin{remark}
We do not know whether assuming that $\mathscr{D}_{\B_\ast}^K(\mathscr{V})$ is a sheaf of projective $\mathscr{O}_X\otimes_{\Q_p}\B_\ast^{\Gal_K}$-modules of rank $d$ implies that the formation of $\mathscr{D}_{\B_\ast}^K(\mathscr{V})$ commutes with base change.  This is why our definition of $\B_\ast$-admissibility of a family includes the condition that $\mathscr{B}_{X,\ast}\otimes_{\mathscr{O}_X\otimes\B_\ast^{\Gal_K}}\mathscr{D}_{\B_\ast}^K(\mathscr{V})\rightarrow \mathscr{B}_{X,\ast}\otimes_{\mathscr{O}_X}\mathscr{V}$ is an isomorphism.

If the natural base change morphism $B\otimes_{\mathscr{O}_X}\mathscr{D}_{\B_\ast}^K(\mathscr{V})\rightarrow \mathscr{D}_{\B_\ast}^K(\mathscr{V}\otimes_{\mathscr{O}_X}B)$ were injective for all morphisms $\Sp(B)\rightarrow X$, with $B$ an $E$-finite artin local ring, we could deduce that $\mathscr{V}$ is $\B_\ast$-admissible.  However, the low-degree exact sequence in \ref{coh-dim-1} shows that there is an obstruction to such injectivity when $\ast\in\{\HT,\dR\}$, at least \emph{a priori}.
\end{remark}

When $\B_{\ast}=\B_{\HT}$ or $\B_{\dR}$, we can prove finer results.  Fix an interval $[a,b]$, and define 
\[	\mathscr{D}_{\HT}^{[a,b]}(\mathscr{V}):=\left((\mathscr{O}_X\otimes_{\Q_p}\oplus_{i=a}^b\C_K\cdot t^i)\otimes_{\mathscr{O}_X}\mathscr{V}\right)^{\Gal_K}	\]
\[	\mathscr{D}_{\dR}^{[a,b]}(\mathscr{V}):=\left((\mathscr{O}_X\otimes_{\Q_p}t^a\B_{\dR}/t^b)\otimes_{\mathscr{O}_X}\mathscr{V}\right)^{\Gal_K}	\]
We think of these coherent sheaves as sheaves of ``periods in Hodge--Tate weight $[a,b]$''.  Fix an integer $0\leq d'\leq d$.

\begin{definition}
A morphism $f:X\rightarrow X'$ is a \emph{Zariski-locally closed immersion} if there is a Zariski-open subspace $U\subset X'$ such that $f$ factors through a Zariski-closed immersion $X\hookrightarrow U$.
\end{definition}

\begin{thm}\label{dht-ddr-periods}
Let $X$ and $\mathscr{V}$ be as above, and let $\ast\in\{\HT,\dR\}$.  There is a Zariski-locally closed immersion $X_{\B_\ast,d'}^{[a,b]}\hookrightarrow X$ such that $x:\Sp(B)\rightarrow X$ factors through $X_{\B_\ast,d'}^{[a,b]}$ if and only if $\mathscr{D}_{\B_\ast}^{[a,b]}(\mathscr{V}_x)$ is a free $B\otimes_{\Q_p}K$-module of rank $d'$, where $B$ is an $E$-finite artin local ring and $\mathscr{V}_B:=\mathscr{V}\otimes_AB$.  In fact, $X_{\B_\ast,d'}^{[a,b]} = X_{\B_\ast,\leq d'}^{[a,b]}\cap X_{\B_\ast,\geq d'}^{[a,b]}$, where $X_{\B_\ast,\leq d'}^{[a,b]}\subset X$ is Zariski-open and $X_{\B_\ast,\geq d'}^{[a,b]}\hookrightarrow X$ is Zariski-closed.
\end{thm}
In fact, the $X_{\B_\ast,d'}^{[a,b]}$ give a stratification of $X$, in the sense that $X_{\B_\ast,\leq d'-1}^{[a,b]}=X_{\B_\ast,\leq d'}^{[a,b]}\smallsetminus X_{\B_\ast,d'}^{[a,b]}$ and $X=X_{\B_\ast,\leq d}^{[a,b]}$.

\begin{thm}\label{pointwise-dht-ddr-periods}
Let $X$ and $\mathscr{V}$ be as above, and let $\ast\in\{\HT,\dR\}$. Suppose that for every $E$-finite artinian point $x:A\rightarrow B$, the $B\otimes_{\Q_p}K$-module $\mathscr{D}_{\B_\ast}^{[a,b]}(\mathscr{V}_x)$ is free of rank $d'$, where $0\leq d'\leq d$.  Then
\begin{enumerate}
\item	$\mathscr{D}_{\B_\ast}^{[a,b]}(\mathscr{V})$ is a locally free $\mathscr{O}_X\otimes_{\Q_p}K$-module of rank $d'$, and each $\left(t^k\cdot\mathscr{D}_{\Sen}^{K_n}(\mathscr{V})\right)^{\Gamma_K}$ is a locally free $\mathscr{O}_X\otimes_{\Q_p} K$-module, 
\item	the formation of $\mathscr{D}_{\B_\ast}^{[a,b]}(\mathscr{V})$ commutes with arbitrary base change $f:X'\rightarrow X$
\end{enumerate}
If $d'=d$, then
\begin{enumerate}
\item[(3)]	$\mathscr{D}_{\B_\ast}(\mathscr{V})=\mathscr{D}_{\B_\ast}^{[a,b]}(\mathscr{V})$,
\item[(4)]	the natural morphism 
$$\alpha_{\mathscr{V}}:\mathscr{B}_\ast\otimes_{\mathscr{O}_X\otimes_{\Q_p}K}\mathscr{D}_{\B_\ast}^K(\mathscr{V})\rightarrow \mathscr{B}_\ast\otimes_{\mathscr{O}_X}\mathscr{V}$$
is an isomorphism.
\end{enumerate}
\end{thm}

Before we begin proving Theorems~\ref{adm-locus}, \ref{pointwise-adm}, and \ref{comp-isom} and their refinements, we discuss some consequences.

First of all, the subspaces $X_{\B_\ast}^{[a,b]}\hookrightarrow X$ have strong functorial properties.
\begin{cor}\label{b-adm-functorial}
Let $X$ and $\mathscr{V}$ be as above, and let $f:X'\rightarrow X$ be a morphism of rigid analytic spaces.  Then 
\begin{enumerate}
\item	$f$ factors through $X_{\B_\ast}^{[a,b]}$ if and only if $f^\ast\mathscr{V}$ is $\B_\ast$-admissible with Hodge--Tate weights in the interval $[a,b]$,
\item	the subspace ${X'}_{\B_\ast}^{[a,b]}\hookrightarrow X'$ is induced via base change from $X_{\B_\ast}^{[a,b]}\hookrightarrow X$.
\end{enumerate}
\end{cor}
\begin{proof}
\begin{enumerate}
\item	By Theorems~\ref{adm-locus} and \ref{pointwise-adm}, $f^\ast\mathscr{V}$ is $\B_\ast$-admissible with Hodge--Tate weights in the interval $[a,b]$ if and only if $x^\ast f^\ast\mathscr{V}$ is $\B_\ast$-admissible with Hodge--Tate weights in the interval $[a,b]$ for every $E$-finite artin local point $x:\Sp(B)\rightarrow X'$.  We may assume that $X$ and $X'$ are affinoid, with $X=\Sp(A)$ and $X'=\Sp(A')$ for $E$-affinoid algebras $A$ and $A'$.

Suppose that $V_{A'}$ is $\B_\ast$-admissible with Hodge--Tate weights in $[a,b]$.  Then for every maximal ideal $\mathfrak{m}'\subset A'$ and every integer $n\geq 0$, the composition $A\rightarrow A'\rightarrow A'/\mathfrak{m}'^n$ factors through $A_{\B_\ast}^{[a,b]}$.  In other words, if $x\in\ker(A\twoheadrightarrow A_{\B_\ast}^{[a,b]})$, then $f(x)\in\mathfrak{m}'^n$ for all $\mathfrak{m}'$ and all $n$.  Since $\cap_{n,\mathfrak{m}'} \mathfrak{m}'^n=0$, we see that $f(x)=0$.

Suppose conversely that $f$ factors through $A\twoheadrightarrow A_{\B_\ast}^{[a,b]}$, and consider an $E$-finite artinian point $x:A'\rightarrow B'$.  By assumption, the composition $A\rightarrow A'\rightarrow B'$ factors through $A\twoheadrightarrow A_{\B_\ast}^{[a,b]}$, so the induced representation $V_{B'}$ is $\B_\ast$-admissible with Hodge--Tate weights in $[a,b]$.  Thus, $V_{A'}$ is Hodge--Tate.
\item	This follows from the first part, and from the universal property of fiber products of rigid analytic spaces.
\end{enumerate}
\end{proof}

Similarly, if $\ast\in\{\HT,\dR\}$, the subspaces $X_{\B_\ast,d'}^{[a,b]}$ are functorial for any $0\leq d'\leq d$.
\begin{cor}
Let $X$ and $\mathscr{V}$ be as above, and let $f:X'\rightarrow X$ be a morphism of rigid analytic spaces.  Then 
\begin{enumerate}
\item	$f$ factors through $X_{\B_\ast,d'}^{[a,b]}$ if and only if $\mathscr{D}_{\B_\ast}^{[a,b]}(f^\ast\mathscr{V})$ is a locally free $\mathscr{O}_X\otimes_{\Q_p}K$-module of rank $d'$,
\item	the subspace ${X'}_{\B_\ast,d'}^{[a,b]}\hookrightarrow X'$ is induced via base change from $X_{\B_\ast,d'}^{[a,b]}\hookrightarrow X$.
\end{enumerate}
\end{cor}
The proof proceeds identically to the proof of Corollary~\ref{b-adm-functorial}.

We can also refine our conclusions about the structure of $\mathscr{D}_{\B_\ast}^K(\mathscr{V})$.
\begin{cor}
Let $\mathscr{V}$ be a family of Galois representations such that for any $E$-finite artinian point $x:\Sp(B)\rightarrow X$, the specialization $\mathscr{V}_x$ is $\B_{\ast}$-admissible.  Then the vector bundle $\mathscr{D}_{\B_\ast}^K(\mathscr{V})$ is $X$-locally free.  
\end{cor}
\begin{remark}
The hypothesis on $\mathscr{V}$ is phrased in terms of a pointwise condition because this corollary is used in the proof of Theorem~\ref{comp-isom} when $\ast\in\{\st,\cris\}$.
\end{remark}
\begin{proof}
It suffices to show that for any point $x\in X$, the completed stalk $\mathscr{D}_{\B_\ast}^K(\mathscr{V})_x^\wedge$ is a free $\mathscr{O}_{X,x}^\wedge\otimes_{\Q_p}\B_\ast^{\Gal_K}$-module.  Since $\mathscr{V}$ is assumed $\B_\ast$-admissible, the formation of $\mathscr{D}_{\B_\ast}^K(\mathscr{V})$ commutes with arbitrary base change on $X$, and hence
\[	\mathscr{D}_{\B_\ast}^K(\mathscr{V})_x^\wedge = \varprojlim_n\mathscr{D}_{\B_\ast}^K(\mathscr{V}\otimes_{\mathscr{O}_X}\mathscr{O}_{X,x}/\mathfrak{m}_x^n)	\]
But again by the $\B_\ast$-admissibility of $\mathscr{V}$, $\mathscr{D}_{\B_\ast}^K(\mathscr{V}\otimes_{\mathscr{O}_X}\mathscr{O}_{X,x}/\mathfrak{m}_x^n)$ is a free $\mathscr{O}_{X,x}/\mathfrak{m}_x^n\otimes_{\Q_p}\B_\ast^{\Gal_K}$-module of rank $d$, and the transition maps are simply reduction modulo $\mathfrak{m}_x^n$.  Therefore, $\mathscr{D}_{\B_\ast}^K(\mathscr{V})_x^\wedge$ is $\mathscr{O}_{X,x}^\wedge\otimes_{\Q_p}\B_\ast^{\Gal_K}$-free of rank $d$, as desired.
\end{proof}

If $\mathscr{V}$ is Hodge--Tate (resp. de Rham), the graded pieces $\gr^i\mathscr{D}_{\HT}^K(\mathscr{V})$ (resp. the submodules $\Fil^i\mathscr{D}_{\dR}^K(\mathscr{V})$) need not be $X$-locally free as $\mathscr{O}_X\otimes_{\Q_p}K$-modules; an example is given in~\cite[Remarque 3.1.1.4]{bm}.  However, we can use Theorem~\ref{adm-locus} to cut out the locus where $\mathscr{D}_{\dR}(\mathscr{V})$ is filtered by subbundles and has specified Hodge polygon.

\begin{definition}
Let $A$ be an $E$-affinoid algebra, and let $D$ be an $A$-locally free $A\otimes_{\Q_p}K$-module equipped with a separated exhaustive decreasing filtration $\Fil^\bullet D$ by $A$-locally free sub-bundles.  Let $\{i_0<i_2<\cdots <i_k\}$ be the distinct $i$ such that $\gr^i(D)\neq 0$.  The \emph{Hodge polygon} $\Delta_D$ of $D$ is the convex polygon in the plane with left-most endpoint $(0,0)$ and $\rk_{A\otimes_{\Q_p}K}\gr^{i_j}(D)$ segments of horizontal distance $1$ and slope $i_j$ for $0\leq j\leq k$.  The \emph{Hodge number} $t_H$ of $D$ is the $y$-coordinate of the right-most point of $\Delta_D$, i.e., $\sum_j i_j\cdot\rk_{A\otimes_{\Q_p}K}\gr^{i_j}(D)$.
\end{definition}

\begin{cor}
Let $\mathscr{V}$ be a rank $d$ family of $\Gal_K$-representations, and let $\Delta$ be a convex polygon in the plane with left-most endpoint $(0,0)$ and right-most endpoint $(d,t_H)$ for some $t_H\in\Z$.  Then there is a closed immersion $X_{\B_\ast}^\Delta\hookrightarrow X$ such that for any $E$-finite artin local ring $B$, a map $\Sp(B)\rightarrow X$ factors through $X_{\B_\ast}^{\Delta}$ if and only if the induced $B$-linear Galois representation $\mathscr{V}_B$ is $\B_\ast$-admissible and the Hodge polygon of $\D_{\dR}(\mathscr{V}_B)$ is $\Delta$.  In fact, $X_{\B_\ast}^\Delta$ is a union of connected components of $X_{\B_\ast}^{[a,b]}$, where $a=i_0$ and $b=i_k$.  If $X=X_{\B_\ast}^\Delta$, then the Hodge polygon of $\mathscr{D}_{\dR}(\mathscr{V})$ is $\Delta$.
\end{cor}
\begin{proof}
We may assume that $X=X_{\B_\ast}^{[a,b]}$.  For each $i\in[a,b]$, we have an exact sequence of vector bundles
$$0\rightarrow \Fil^{i+1}\D_{\dR}(\mathscr{V})\rightarrow \Fil^i\D_{\dR}(\mathscr{V})\rightarrow \gr^i\D_{\dR}(\mathscr{V})\rightarrow 0$$
These vector bundles have locally constant rank on $X$ (though not necessarily globally constant rank), and their formations commute with base change on $X$, by Theorem~\ref{pointwise-adm}, so we take the union of the connected components where $\gr^i\D_{\dR}(\mathscr{V})$ has the correct dimension.
\end{proof}

We can also stratify the spaces $X_{\B_{\dR},d'}^{[a,b]}$ by Hodge polygon, although we do not get a decomposition into connected components, because we have no {\it a priori} interpretation of the graded pieces of $\mathscr{D}_{\dR}(\mathscr{V})$ when $\mathscr{V}$ is not de Rham.

\begin{cor}
Let $\mathscr{V}$ be a rank $d$ family of $\Gal_K$-representations, let $d'$ be an integer $0\leq d'\leq d$, and let $\Delta$ be a convex polygon in the plane with left-most endpoint $(0,0)$ and right-most endpoint $(d',t_H)$ for some $t_H\in\Z$.  Then there is a Zariski locally-closed immersion $X_{\dR}^\Delta\hookrightarrow X$ such that for any $E$-finite artin local ring $B$, a map $\Sp(B)\rightarrow X$ factors through $X_{\dR}^\Delta$ if and only if $\Fil^{i}\D_{\dR}^{[a,b]}(V_B)$ is projective for all $i$ and $\D_{\dR}^{[a,b]}(V_B)$ has Hodge polygon $\Delta$, where $a=i_0$ and $b=i_k$.  If $X=X_{\dR}^\Delta$, then the Hodge polygon of $\mathscr{D}_{\dR}^{[a,b]}(\mathscr{V})$ is $\Delta$. 
\end{cor}
\begin{proof}
Let $c_j$ be the $x$-coordinate of the right endpoint of the $j$th segment of $\Delta$ (where we count segments starting with $0$), and for $i_j< i\leq i_{j+1}$, let $d_i=\sum_{\ell=j}^k c_\ell$.  Then we are looking for the locus of $E$-finite artin points $x:\Sp(B)\rightarrow X$ where for all $a\leq i\leq b$, if $i_j<i\leq i_{j+1}$ then $\D_{\dR}^{[i,b]}(\mathscr{V}_B)$ is a free $B\otimes_{\Q_p}K$-module of rank $i$.

We use Theorem~\ref{dht-ddr-periods} to construct the desired locus.  Indeed, for every $i\in [a,b]$, Theorem~\ref{dht-ddr-periods} gives us a Zariski-open subspace $X_{\B_{\dR},\leq d_i}^{[i,b]}$ and a Zariski-closed subspace $X_{\B_{\dR},\geq d_i}^{[i,b]}$ such that $X_{\B_{\dR},d_i}^{[i,b]}:=X_{\B_{\dR},\leq d_i}^{[i,b]}\cap X_{\B_{\dR},\geq d_i}^{[i,b]}$ represents the condition ``$\D_{\dR}^{[i,b]}(\mathscr{V}_B)$ is free of rank $d_i$''.  Then we put
$$X_{\dR}^\Delta:=\left(\cap_{i\in[a,b]}X_{\B_{\dR},\leq d_i}^{[i,b]}\right)\cap\left(\cap_{i\in[a,b]}X_{\B_{\dR},\geq d_i}^{[i,b]}\right)$$

The second assertion follows similarly, by repeated application of Theorem~\ref{pointwise-dht-ddr-periods}.
\end{proof}

\begin{remark}
It is possible to define an ordering on Hodge polygons so that the spaces $X_{\dR}^\Delta$ yield a Zariski-stratification of $X_{\dR}^{[a,b]}$.  This is done in~\cite[\textsection3]{shrenik}.
\end{remark}

We can also show that when restricted to the category of $\B_\ast$-admissible representations, the functor $\mathscr{D}_{\B_\ast}^K$ is well-behaved with respect to exact sequences, tensors, and duals.
\begin{cor}
\begin{enumerate}
\item	Let $\Rep_{X}^{\B_\ast}(\Gal_K)$ be the category of $\B_\ast$-admissible families of representations of $\Gal_K$ over $X$.  Then $\mathscr{D}_{\B_\ast}:\Rep_{X}^{\B_\ast}(\Gal_K)\rightarrow \Proj_{\mathscr{O}_X\otimes_{\Q_p}\B_\ast^{\Gal_K}}$ is exact and faithful, where $\Proj_{\mathscr{O}_X\otimes_{\Q_p}\B_\ast^{\Gal_K}}$ is the category of sheaves of projective $\mathscr{O}_X\otimes_{\Q_p}\B_\ast^{\Gal_K}$-modules, and any subrepresentation or quotient of a $\B$-admissible family of representations is itself $\B$-admissible.
\item	The sub-category $\Rep_{X}^{\B_\ast}(\Gal_K)$ is stable under formation of tensor products and duals, and the functor $\mathscr{V}\mapsto\mathscr{D}_{\B_\ast}^K(\mathscr{V})$ commutes with these operations when restricted to $\Rep_{X}^{\B_\ast}(\Gal_K)$.
\item	If $\B_\ast=\B_{\HT}$ (resp. $\B_{\dR}$), then the grading (resp. filtration) on $\mathscr{D}_{\B_\ast}^K(\mathscr{V})$ is also exact and tensor compatible.
\end{enumerate}
\end{cor}
\begin{proof}
These statements all follow from the corresponding statements with artinian coefficients, because for $\mathscr{V}$ a $\B_\ast$-admissible family of representations of $\Gal_K$, the formation of $\mathscr{D}_{\B_\ast}^K(\mathscr{V})$ commutes with arbitrary base change on $X$.  Since $\mathscr{D}_{\B_\ast}^K(\mathscr{V})$ is a finite $\mathscr{O}_X\otimes_{\Q_p}\B_\ast^{\Gal_K}$-module, we can check isomorphisms on thickenings of closed points of $X$.
\end{proof}

Finally, we can use the existance of the closed subspace $X_{\B_\ast}^{[a,b]}\hookrightarrow X$ to deduce $\B_\ast$-admissibility on a Zariski-open neighborhood of $x\in X$ from $\B_\ast$-admissibility on infinitesimal neighborhoods of $x$.  
\begin{cor}
Let $V$ be a continuous $\mathscr{O}_X$-linear representation of $\Gal_K$ as above, and suppose $x\in X$ is a point such that $V\otimes_{\mathscr{O}_X}\mathscr{O}_{X,x}/\mathfrak{m}_x^n$ is $\B_\ast$-admissible with Hodge--Tate weights in $[a,b]$ for all $n\geq 0$.  Then there is a Zariski-open neighborhood $U$ of $x$ such that $V|_{U}$ is $\B_\ast$-admissible with Hodge--Tate weights in $[a,b]$.
\end{cor}
\begin{proof}
We may assume that $X=\Sp(A)$ for some $E$-affinoid algebra $A$, so that $X_{\B_\ast}^{[a,b]}=\Sp(A_{\B_\ast}^{[a,b]})$ for some quotient $A\twoheadrightarrow A_{\B_\ast}^{[a,b]}$.  The assumption on the infinitesimal neighborhoods of $x$ implies that the natural map $A\rightarrow A_{\mathfrak{m}_x}^\wedge$ factors through $A_{\B_\ast}^{[a,b]}$.  This implies that the complete local rings of $A$ and $A_{\B_\ast}^{[a,b]}$ at $\mathfrak{m}_x$ are the same, which in turn implies that $\Sp(A_{\B_\ast}^{[a,b]})$ contains an Zariski-open neighborhood of $x$.
\end{proof}
\begin{remark}
In fact, since $X_{\B_\ast}^{[a,b]}$ is a closed subspace containing an admissible open neighborhood of $x$, if $X=\Sp(A)$ for some $E$-affinoid algebra $A$, then $X_{\B_\ast}^{[a,b]}$ contains all irreducible components of $\Sp(A)$ passing through $x$.
\end{remark}

\subsection{Hodge--Tate and de Rham loci}

As above, we let $X$ be a quasi-separated $E$-analytic space and we let $\mathscr{V}$ be a finite locally free $\mathscr{O}_X$-module of rank $d$, equipped with a continuous $\mathscr{O}_X$-linear action of $\Gal_K$.  

In order to construct quotients $\mathscr{O}_X\twoheadrightarrow \mathscr{O}_{X_{\HT}^{[a,b]}}$ and $\mathscr{O}_X\twoheadrightarrow \mathscr{O}_{X_{\dR}^{[a,b]}}$ for Theorem~\ref{adm-locus} for $\B_\ast=\B_{\HT}$ or $\B_{\dR}$, we work locally on $X$ and construct a suitable coherent ideal sheaf.  Similarly, we work locally on $X$ to construct the Zariski-locally closed immersions $X_{\HT,d'}^{[a,b]}\hookrightarrow X$ and $X_{\dR,d'}^{[a,b]}\hookrightarrow X$ required by Theorem~\ref{dht-ddr-periods}.

In order to prove Theorems~\ref{pointwise-adm} and~\ref{pointwise-dht-ddr-periods} for $\B_\ast=\B_{\HT}$ and $\B_\ast=\B_{\dR}$, it likewise suffices to work locally on $X$.  Thus, we may reduce to the case where $X=\Sp(A)$ for $A$ an $E$-affinoid algebra and $V:=\Gamma(X,\mathscr{V})$ is a finite free $A$-module of rank $d$ equipped with a continuous $A$-linear action of $\Gal_K$ which admits a free $\Gal_K$-stable $\mathscr{A}$-lattice $V_0$ for some formal $\mathscr{O}_E$-model $\mathscr{A}$ of $A$.

Before we begin, we prove a useful lemma.
\begin{lemma}\label{artin-local-free}
Let $R$ be an artin ring, let $M$ be a free $R$-module of rank $r$ equipped with an endomorphism $T:M\rightarrow M$, and suppose that
\[	0\rightarrow M'\rightarrow M\xrightarrow{T}M\rightarrow M''\rightarrow 0	\]
is exact.  Then $M'$ is free of rank $d$ over $R$ if and only if $M''$ is.
\end{lemma}
\begin{proof}
Since $R$ is an artin ring, it is semi-local.  The assertions ``$M'$ is free of rank $d$'' and ``$M''$ is free of rank $d$'' can each be checked by passing to local factors of $R$, so we may assume that $R$ is a local ring with maximal ideal $\mathfrak{m}$. 

It suffices to show that for an exact sequence $0\rightarrow M'\rightarrow M\rightarrow N\rightarrow 0$ of $R$-modules (with $M$ free of rank $r$), $M'$ is free of rank $d$ if and only if $N$ is free of rank $r-d$.  

If $N$ is free of rank $r-d$, it is projective, so the exact sequence splits and $M'$ is free of rank $d$ (as $R$ is local).  

Conversely, suppose that $M'$ is free of rank $d$.  We will prove that $M/M'$ is free of rank $r-d$ by induction on $d$.  If $d=0$, there is nothing to prove.  So suppose that $M'$ is free of rank $d$, and suppose we know the result for submodules of rank $d-1$.  Choose bases $\{e_1,\ldots,e_d\}$ and $\{f_1,\ldots,f_r\}$ for $M'$ and $M$, respectively, and consider the image $a_1f_1+\cdots a_rf_r$ of $e_1$.  At least one of the $a_i$ is a unit in $R$, because otherwise by injectivity of $M'\rightarrow M$, the element $e_1\in M'$ would be killed by $\ann_R(\mathfrak{m})\neq 0$ (which is impossible, as $e_1$ is part of a basis of $M'$).  Thus, without loss of generality, we assume that $a_1$ is a unit.  Then $\{e_1,f_2,\ldots,f_r\}$ is a basis of $M$, and 
$$0\rightarrow M'/\langle e_1\rangle\rightarrow M/\langle e_1\rangle\rightarrow N\rightarrow 0$$
is exact, so by the inductive hypothesis, $N$ is free of rank $(r-1)-(d-1)=r-d$.
\end{proof}

\subsubsection{Hodge--Tate locus}

Let $L/K$ be a finite extension such that $\Gal_L$ acts trivially on $V_0/12pV_0$.  Then for $n\gg0$, $\D_{\Sen}^{L_n}(V)$ is finite free of rank $d$ over $A\otimes_{\Q_p}L_n$ and carries a linear action of $\Gamma_{L_n}$.  If necessary, we increase $n$ so that $\Gamma_{L_n}$ is procyclic, with topological generator $\gamma$.  Moreover, the formation of $\D_{\Sen}^{L_n}(V)$ commutes with arbitrary $E$-affinoid base change on $A$ and $\left(\D_{\Sen}^{L_n}(V)\right)^{\Gamma_{L_n}} = \left((\C_K\otimes A)\otimes_A V\right)^{\Gal_{L_n}}$.  As a consequence,
$$\D_{\HT}^{L_n}(V) = \left(\oplus_k t^k\cdot\D_{\Sen}^{L_n}(V)\right)^{\Gamma_{L_n}=1}$$
Recall that we have defined 
$\D_{\HT}^{K,[a,b]}(V) = \left(\oplus_{i\in[a,b]}\D_{\Sen}^{L_n}(V)\cdot t^i\right)^{\Gamma_K}$.

\begin{thm}\label{ht-adm-locus}
Let $V$ be as above.  Then for every $0\leq d'\leq d$, there is a Zariski-locally closed immersion $X_{\HT,d'}^{[a,b]}\hookrightarrow X$ such that for any $E$-finite artin local ring $B$ and $V_B:=V\otimes_AB$, $x:\Sp(B)\rightarrow X$ factors through $X_{\HT,d'}^{[a,b]}$ if and only if $\D_{\HT}^{K,[a,b]}(V_B)$ is a free $B\otimes_{\Q_p}K$-module of rank $d'$.  In fact, $X_{\HT,d'}^{[a,b]} = X_{\HT,\leq d'}^{[a,b]}\cap X_{\HT,\geq d'}^{[a,b]}$, where $X_{\HT,\leq d'}^{[a,b]}\subset X$ is Zariski-open and $X_{\HT,\geq d'}^{[a,b]}\hookrightarrow X$ is Zariski-closed.  If $d'=d$, there is a quotient $A\twoheadrightarrow A_{\HT}^{[a,b]}$ such that an $E$-finite artinian point $x:A\rightarrow B$ factors through $A_{\HT}^{[a,b]}$ if and only if $V_B$ is Hodge--Tate with Hodge--Tate weights in the interval $[a,b]$.
\end{thm}
\begin{proof}
First we note that because $\D_{\HT}^{L_n}(V)=L_n\otimes \D_{\HT}^K(V)$, it is enough consider $V_B$ as a representation of $\Gal_{L_n}$.

Next let $M=\oplus_{i=a}^b \D_{\Sen}^{L_n}(V)\cdot t^i$.  We note that $\H^0(\Gamma_{L_n},M\otimes_A B)$ is free over $B\otimes_{\Q_p}L_n$ of rank $d'$ if and only if $\H^1(\Gamma_{L_n},M\otimes_A B)$ is, because the continuous $\Gamma_{L_n}$-cohomology of $M\otimes_A B$ is computed by the complex
$$0\rightarrow M\otimes_A B \xrightarrow{\gamma-1} M\otimes_A B\rightarrow 0$$
and we can apply Lemma~\ref{artin-local-free}.  Further, the formation of $\H^1$ commutes with arbitrary base change on $A$, by Corollary~\ref{coh-dim-1}.  Since $\H^1(\Gamma_{L_n},M)$ is a coherent $A\otimes_{\Q_p}L_n$-module and 
\[	M\otimes_AB\xrightarrow{\gamma-1}M\otimes_AB\rightarrow \H^1(\Gamma_{L_n},M\otimes_A B)\rightarrow 0	\]
is a finite presentation of $\H^1(\Gamma_{L_n},M\otimes_A B)$, it follows that 
\[	M\xrightarrow{T=\gamma-1}M\rightarrow \H^1(\Gamma_{L_n},M)\rightarrow 0	\]
is a finite presentation of $\H^1(\Gamma_{L_n},M)$.

We use the theory of Fitting ideals to cut out the locus where $\H^1(\Gamma_{L_n},M\otimes_A B)$ is free of rank $d'$ over $B\otimes_{\Q_p}L_n$.  By~\cite[Prop. 20.8]{eisenbud}, $\H^1(\Gamma_{L_n},M\otimes_A B)$ is projective of rank $d'$ over $B\otimes_{\Q_p}L_n$ if and only if $\Fitt_{d'}(\H^1(\Gamma_{L_n},M\otimes_A B))=B\otimes_{\Q_p}L_n$ and $\Fitt_{d'-1}(\H^1(\Gamma_{L_n},M\otimes_A B))=0$.  The latter is a closed condition defined by the $(d'-1)\times (d'-1)$-minors of $T$.  The former is a Zariski-open condition defined by inverting the $d'\times d'$-minors of $T$.

If $d'=d$, we claim the open condition can be ignored on the complement of the zero locus of ideal generated by the $(d-1)\times (d-1)$-minors of $T$.  First, we note that $\Fitt_d(\H^1(\Gamma_{L_n},M\otimes_A B))=B\otimes_{\Q_p}L_n$ if and only if $\H^1(\Gamma_{L_n},M\otimes_A B)$ can be generated by $d$ elements, by~\cite[Prop. 20.6]{eisenbud}.  More precisely, $B\otimes_{\Q_p}L_n$ is a semi-local ring, while \cite[Prop. 20.6]{eisenbud} applies to modules over local rings.  But $\H^1(\Gamma_{L_n},M\otimes_A B)$ can be generated by $d$ elements over $B\otimes_{\Q_p}L_n$ if and only if the same is true after passing to idempotent factors of $B\otimes_{\Q_p}L_n$.  Similarly, since the formation of Fitting ideals commutes with base change, we can check that $\Fitt_d(\H^1(\Gamma_{L_n},M\otimes_A B))=(1)$ after passing to idempotent factors of $B\otimes_{\Q_p}L_n$.

Moreover, the formation of $\H^1(\Gamma_{L_n},M\otimes_A B)$ commutes with base change on $B$, so by Nakayama's lemma, $\H^1(\Gamma_{L_n},M\otimes_A B)$ can be generated by lifts of generators of $\H^1(\Gamma_{L_n},M\otimes_A B/\mathfrak{m}_B)$.  But if the Fitting ideal $\Fitt_{d-1}(\H^1(\Gamma_{L_n},M\otimes_A B/\mathfrak{m}_B))$ vanishes, then $\H^1(\Gamma_{L_n},M\otimes_A B/\mathfrak{m}_B)$ cannot be generated by $d-1$ elements at any point of $B/\mathfrak{m}_B\otimes_{\Q_p}L_n$.  Therefore, the $\Q_p$-dimension of $\H^0(\Gamma_{L_n},M\otimes_A B/\mathfrak{m}_B)$ (which is the same as the $\Q_p$-dimension of $\H^1(\Gamma_{L_n},M\otimes_A B/\mathfrak{m}_B)$) is at least $d\cdot\dim_{\Q_p}B\cdot\dim_{\Q_p}L_n$.  But then the formalism of admissible representations implies that the $\Q_p$-dimension of $\H^0(\Gamma_{L_n},M\otimes_A B/\mathfrak{m}_B)$ is exactly $d\cdot\dim_{\Q_p}B\cdot\dim_{\Q_p}L_n$ and Proposition~\ref{artin-adm} implies that $\H^0(\Gamma_{L_n},M\otimes_A B/\mathfrak{m}_B)$ is a free $B/\mathfrak{m}_B\otimes_{\Q_p}L_n$-module of rank $d$, so $\H^1(\Gamma_{L_n},M\otimes_A B/\mathfrak{m}_B)$ is as well, by Lemma~\ref{artin-local-free}.  Clearly this implies that $\H^1(\Gamma_{L_n},M\otimes_A B/\mathfrak{m}_B)$ can be generated by $d$ elements.

Thus, the condition that $V_B$ be Hodge--Tate with Hodge--Tate weights in the appropriate range is cut out by the $(d-1)\times(d-1)$-minors of $T$.  Since $T$ has coefficients in $A\otimes_{\Q_p}L_n$ and $E\otimes_{\Q_p}L_n$ is finite free over $E$, we obtain the quotient of $A$ we sought.
\end{proof}


We turn to the converse question.

\begin{thm}\label{ht-adm}
Let $V$ be as above.  Suppose that for every $E$-finite artinian point $x:A\rightarrow B$, the $B\otimes_{\Q_p}K$-module $\D_{\HT}^{K,[a,b]}(V_x)$ is free of rank $d'$, where $0\leq d'\leq d$.  Then
\begin{enumerate}
\item	$\D_{\HT}^{K,[a,b]}(V)$ is a locally free $A\otimes_{\Q_p}K$-module of rank $d'$, and each $\left(t^k\cdot\D_{\Sen}^{K_n}(V)\right)^{\Gamma_K}$ is a locally free $A$-module,
\item	the formation of $\D_{\HT}^{K,[a,b]}(V)$ commutes with arbitrary base change $f:A\rightarrow A'$,
\end{enumerate}
If $d'=d$, then
\begin{enumerate}
\item[(3)]	$\D_{\HT}^{K}(V)=\left(\oplus_{k=a}^bt^k\cdot\D_{\Sen}^{K_n}(V)\right)^{\Gamma_K}$, and the formation of $\D_{\HT}^K(V)$ commutes with arbitrary base change $f:A\rightarrow A'$,
\item[(4)]	the natural map $\alpha_V:(A\widehat\otimes\B_{\HT})\otimes_{A\otimes_{\Q_p}K}\D_{\HT}^K(V)\rightarrow (A\widehat\otimes\B_{\HT})\otimes_AV$ is a Galois-equivariant isomorphism of graded $A\widehat\otimes\B_{\HT}$-modules, and so $V$ is Hodge--Tate.
\end{enumerate}
\end{thm}
\begin{proof}
It is enough to consider $V$ as a representation of $\Gal_{L_n}$, since $\D_{\HT}^{L_n}(V)=L_n\otimes \D_{\HT}^K(V)$.  Since $\D_{\Sen}^{K_n}(V) = \left(\D_{\Sen}^{L_n}(V)\right)^{H_K}$, the decomposition of $\D_{\HT}^K(V)$ follows from the decomposition of $\D_{\HT}^{L_n}(V)$.

We will use the base change spectral sequence of Theorem~\ref{base-change} for the continuous $\Gamma_{L_n}$-cohomology of $M:=\oplus_{i=a}^b \D_{\Sen}^{L_n}(V)\cdot t^i$.

By assumption, $\H^1(\Gamma_{L_n},M\otimes_A B)$ is a free $B\otimes_{\Q_p}L_n$-module of the same rank as $\H^0(\Gamma_{L_n},M\otimes_A B)$.  Since $M$ is $A$-finite, Corollary~\ref{coh-dim-1} implies that the formation of $\H^1(\Gamma_{L_n},M)$ commutes with base change $A\rightarrow B$.  It follows that $\H^1(\Gamma_{L_n},M)$ is a projective $A$-module and $\Tor_{-p}^A(\H^1(\Gamma_{L_n},M),A')$ vanishes for all homomorphisms $f:A\rightarrow A'$ and for all $p<0$.

But if we consider the low-degree exact sequence of Corollary~\ref{coh-dim-1}, the vanishing of the $\Tor$ terms shows that formation of $\H^0(\Gamma_{L_n},M)$ commutes with arbitrary base change $A\rightarrow A'$, and in particular with the base change $x:A\rightarrow B$.  Thus, 
\[	M^{\Gamma_{L_n}}\subset \D_{\HT}^{L_n}(V)=\left(\oplus_{i\in\Z} \D_{\Sen}^{L_n}(V)\cdot t^i\right)^{\Gamma_{L_n}}	\] 
is a locally free $A\otimes_{\Q_p}L_n$-module of rank $d'$.  This proves the first two parts.

Now we assume that $d'=d$.  We claim that $M^{\Gamma_{L_n}}=\D_{\HT}^{L_n}(V)$.  Suppose to the contrary that there is some non-zero $y\in\left(\D_{\Sen}^{L_n}(V)\cdot t^i\right)^{\Gamma_{L_n}}$ for some $i\not\in [0,h]$.  Then there is some $E$-finite artinian point $x:A\rightarrow B$ such that $y$ is non-zero in $\D_{\Sen}^{L_n}(V_x)\cdot t^i$.  But $M_x^{\Gamma_{L_n}}=\D_{\HT}^{L_n}(V_x)$ because $M_x^{\Gamma_{L_n}}\subset\D_{\HT}^{L_n}(V_x)$ and the $L_n$-dimensions of the two sides agree, contradicting the assumed $\Gamma_{L_n}$-invariance of $y$.

Let $f:A\rightarrow A'$ be a morphism of $E$-affinoid algebras.  We have already seen that the formation of $\H^0(\Gamma_{L_n},M)$ commutes with arbitrary affinoid base change on $A$, so $\H^0(\Gamma_{L_n},M\otimes_AA')=\H^0(\Gamma_{L_n},M)\otimes_AA'$.  If $\D_{\HT}^{L_n}(V\otimes_AA')=\H^0(\Gamma_{L_n},M\otimes_AA')$, we are done.  But for any $E$-finite artinian point $x:A'\rightarrow B'$, the induced $B'$-linear representation $(x\circ f)^\ast V$ is Hodge--Tate with Hodge--Tate weights in $[a,b]$.  Then we have just seen that
$\D_{\HT}^{L_n}(V\otimes_AA')=\left(\oplus_{k=a}^bt^k\cdot\D_{\Sen}^{K_n}(V\otimes_AA')\right)^{\Gamma_{L_n}}$,
as desired.

Finally, we show that 
$(A\widehat\otimes\B_{\HT})\otimes_{A\otimes_{\Q_p}K}\D_{\HT}^{K}(V)\rightarrow (A\widehat\otimes\B_{\HT})\otimes_AV$
is a Galois-equivariant isomorphism.  Since the natural map
\[	(A\widehat\otimes\C_K)\otimes_{A\otimes_{\Q_p}L_n}\D_{\Sen}^{L_n}(V)\rightarrow (A\widehat\otimes\C_K)\otimes_AV	\]
is a Galois-equivariant isomorphism, it suffices to show that the natural map
\[	(A\otimes_{\Q_p}{L_n}[t,t^{-1}])\otimes_{A\otimes_{\Q_p}{L_n}}\D_{\HT}^{L_n}(V)\rightarrow \oplus_{i\in\Z}\D_{\Sen}^{L_n}(V)\cdot t^i	\]
is a Galois-equivariant isomorphism of graded $A\otimes_{\Q_p}L_n$-modules.  We may further reduce to checking that the natural map
\[	\gr^i\left((A\otimes_{\Q_p}{L_n}[t,t^{-1}])\otimes_{A\otimes_{\Q_p}{L_n}}\D_{\HT}^{L_n}(V)\right)\rightarrow \D_{\Sen}^{L_n}(V)\cdot t^i	\]
is a Galois-equivariant isomorphism of $A\otimes_{\Q_p}L_n$-modules for all $i$.

Now we have a map of $A$-finite modules, so it suffices to check our desired isomorphism modulo powers of maximal ideals of $A$.  But when $A$ is an $E$-finite artin ring, this follows from the formalism of admissible representations, so we are done.
\end{proof}

\subsubsection{de Rham locus}

To treat the de Rham case, we work with $\D_{\dif}^{L_n}(V)$ instead of $\D_{\Sen}^{L_n}(V)$.  
Recall that we have defined
$$\D_{\dR}^{K,[a,b]}(V) = \left((A\otimes_{\Q_p}t^a\B_{\dR}/t^b)\otimes_{A}V\right)^{\Gamma_K} = \left(t^a\D_{\dif}^{K,+}(V)/t^b\D_{\dif}^{K,+}(V)\right)^{\Gamma_K}$$

\begin{thm}
Let $V$ be as above.  Then for every $0\leq d'\leq d$, there is a Zariski-locally closed immersion $X_{\dR,d'}^{[a,b]}\hookrightarrow X$ such that for any $E$-finite artin local ring $B$ and $V_B:=V\otimes_AB$, $x:\Sp(B)\rightarrow X$ factors through $X_{\dR,d'}^{[a,b]}$ if and only if $\D_{\dR}^{K,[a,b]}(V)$ is a free $B\otimes_{\Q_p}K$-module of rank $d'$.  In fact, $X_{\dR,d'}^{[a,b]} = X_{\dR,\leq d'}^{[a,b]}\cap X_{\dR,\geq d'}^{[a,b]}$, where $X_{\dR,\leq d'}^{[a,b]}\subset X$ is Zariski-open and $X_{\dR,\geq d'}^{[a,b]}\hookrightarrow X$ is Zariski-closed.  If $d'=d$, there is a quotient $A\twoheadrightarrow A_{\dR}^{[a,b]}$ such that an $E$-finite artinian point $x:A\rightarrow B$ factors through $A_{\dR}^{[a,b]}$ if and only if $V_B$ is de Rham with Hodge--Tate weights in the interval $[a,b]$.
\end{thm}
\begin{proof}
Because $\D_{\dR}^{L_n}(V)=L_n\otimes_K \D_{\dR}^K(V)$, it is enough to cut out the locus where $V_B$ is de Rham as a representation of $\Gal_{L_n}$ (with weights in the appropriate range).  Then the proof of Theorem~\ref{ht-adm-locus} carries over verbatim with $M$ re-defined as $M=t^a\D_{\dif}^{L_n}(V)/t^b$, and we obtain the desired result.
\end{proof}

Now we treat the converse question.

\begin{thm}\label{dr-adm}
Let $V$ be as above.  Suppose that for every $E$-finite artinian point $x:A\rightarrow B$, the $B\otimes_{\Q_p}K$-module $\D_{\dR}^{K,[a,b]}(V_x)$ is free of rank $d'$, where $0\leq d'\leq d$.  Then
\begin{enumerate}
\item	$\D_{\dR}^{K,[a,b]}(V)$ is a locally free $A\otimes_{\Q_p}K$-module of rank $d'$,
\item	the formation of $\D_{\dR}^{K,[a,b]}(V)$ commutes with arbitrary base change $f:A\rightarrow A'$,
\end{enumerate}
If $d'=d$, then
\begin{enumerate}
\item[(3)]	$\D_{\dR}^{K}(V)=\D_{\dR}^{K,[a,b]}$, and the formation of $\D_{\dR}^K(V)$ commutes with arbitrary base change $f:A\rightarrow A'$,
\item[(4)]	the natural map $\alpha_V:(A\widehat\otimes\B_{\dR})\otimes_{A\otimes_{\Q_p}K}\D_{\dR}^K(V)\rightarrow (A\widehat\otimes\B_{\dR})\otimes_AV$ is a Galois-equivariant isomorphism of graded $A\widehat\otimes\B_{\dR}$-modules, and so $V$ is de Rham.
\end{enumerate}
\end{thm}
\begin{proof}
It is enough to consider $V$ is as a representation of $\Gal_{L_n}$, since $\D_{\dR}^{L_n}(V)=L_n\otimes_K \D_{\dR}^K(V)$.  Then by the same arguments as in the proof of Theorem~\ref{ht-adm} applied with $M=t^a\D_{\dif}^{L_n,+}(V)/t^b$, we see that $M^{\Gamma_{L_n}}$ is a locally free $A\otimes_{\Q_p}L_n$-module of rank $d'$.

Now we assume $d'=d$.  We claim that $M^{\Gamma_{L_n}}=\D_{dR}^{L_n,+}(V)=\D_{\dR}^{L_n}(V)$.  Since $V$ is Hodge--Tate with Hodge--Tate weights in $[0,h]$, $\left(t^k\D_{\Sen}^{L_n}(V)\right)^{\Gamma_{L_n}}=0$ for $k\not\in [0,h]$.  Moreover, $\H^1(\Gamma_{L_n},t^k\cdot \D_{\Sen}^{L_n}(V\otimes_AB))=0$ for any artinian point $A\rightarrow B$ and $k\not\in [0,h]$, so $\H^1(\Gamma_{L_n},t^k\cdot \D_{\Sen}^{L_n}(V))=0$ for $k\not\in[0,h]$.  Then for any $k>h$, the long exact sequence associated to the short exact sequence of $\Gamma_{L_n}$-modules
$$0\rightarrow t^k\cdot \D_{\Sen}^{L_n}(V)\rightarrow \D_{\dif}^{L_n,+}(V)/t^{k+1}\rightarrow \D_{\dif}^{L_n,+}(V)/t^k\rightarrow 0$$
shows that $\left(\D_{\dif}^{L_n,+}(V)/t^{k+1}\right)^{\Gamma_{L_n}}\rightarrow \left(\D_{\dif}^{L_n,+}(V)/t^k\right)^{\Gamma_{L_n}}$ is an isomorphism.  It follows that 
\[	\left(\D_{\dif}^{L_n,+}(V)\right)^{\Gamma_{L_n}}=\varprojlim_k\left(\D_{\dif}^{L_n,+}(V)/t^k\right)^{\Gamma_{L_n}} = \left(\D_{\dif}^{L_n,+}(V)/t^{h+1}\right)^{\Gamma_{L_n}}	\]

Let $f:A\rightarrow A'$ be a morphism of $E$-affinoid algebras.  Then we have seen that the formation of $\H^0(\Gamma_{L_n},M)$ commutes with arbitrary base change on $A$, so $\H^0(\Gamma_{L_n},M\otimes_AA')=\H^0(\Gamma_{L_n},M)\otimes_AA'$.  If $\D_{\dR}^{L_n}(V\otimes_AA')=\H^0(\Gamma_{L_n},M\otimes_AA')$, we are done.

But for any $E$-finite artinian point $x:A'\rightarrow B'$, the induced $B'$-linear representation $(x\circ f)^\ast V$ is de Rham with Hodge--Tate weights in $[a,b]$.  Then we have just seen that 
\[	\D_{\dR}^{L_n}(V\otimes_AA')=\left(t^a\D_{\dif}^{K_n,+}(V\otimes_AA')/t^b\right)^{\Gamma_{L_n}}	\]
as desired.

Finally, we show that $V$ is de Rham.  Recall that the natural map
\[	(A\widehat\otimes\B_{\dR}^+)\otimes_{A\widehat\otimes L_n[\![t]\!]}\D_{\dif}^{L_n,+}(V)\rightarrow (A\widehat\otimes\B_{\dR}^+)\otimes_AV	\]
is a Galois-equivariant isomorphism respecting the filtration on each side.  Therefore, it suffices to show that the natural map
\[	\Fil^0\left((A\widehat\otimes L_n(\!(t)\!))\otimes_{A\otimes_{\Q_p}L_n}\D_{\dR}^{L_n,+}(V)\right)\rightarrow \D_{\dif}^{L_n,+}(V)	\]
is a Galois-equivariant isomorphism respecting the filtrations.  For this, it further suffices to show that for every $k\geq 0$, the natural map
\[	\Fil^0\left((A\widehat\otimes L_n(\!(t)\!))\otimes_{A\otimes_{\Q_p}L_n}\D_{\dR}^{L_n,+}(V)\right)/t^k\rightarrow \D_{\dif}^{L_n,+}(V)/t^k	\]
is a Galois-equivariant isomorphism respecting the filtrations.  

Now we have a morphism of $A$-finite modules, so it suffices to check this modulo every power of every maximal ideal of $A$.  But when $A$ is an $E$-finite artin ring, this follows from the formalism of admissible representations, so we are done.
\end{proof}

\subsection{Semi-stable and crystalline loci}

We will produce similar quotients of $A$ parametrizing the semi-stable and crystalline loci.  However, our method of proof, which follows the argument in~\cite{bc}, is quite different.  Instead of using cohomology and base change arguments, we will prove that ``de Rham implies uniformly potentially semi-stable'' and then deduce our desired results.

\subsubsection{Uniform potential semi-stability}

Suppose that $A=\Q_p$ and $\rho:\Gal_K\rightarrow \GL(V)$ is a de Rham representation of $\Gal_K$.  Then there is a finite extension $L/K$ such that $\rho|_{\Gal_L}$ is a semi-stable representation.  This is known as the $p$-adic local monodromy theorem, and it follows from a theorem of Berger, combined with Crew's conjecture, which was proved separately by Andr\'e~\cite{andre}, Mebkhout~\cite{mebkhout}, and Kedlaya~\cite{kedlaya-monodromy}.  Berger then associated to any de Rham $p$-adic Galois representation a $p$-adic differential equation, and characterized semi-stability of the representation in terms of unipotence of the associated differential equation~\cite{berger}.  Crew's conjecture states that any $p$-adic differential equation becomes semi-stable over a finite extension.

Neither Berger's construction nor Crew's conjecture work naively when $A$ is an affinoid algebra, so we cannot proceed directly.  However, both pieces are known when the coefficients are a general complete discretely valued field with perfect residue field of characteristic $p>0$.

The version of the $p$-adic monodromy theorem we need is the following:
\begin{thm}[{\cite[Corollaire 6.2.5]{bc}}]\label{monodromy}
Let $B$ be a complete discretely valued field with perfect residue field of characteristic $p>0$, and let $V$ be a $B$-representation of $\Gal_K$ of dimension $d$ which is de Rham.  Then there is a finite extension $L/K$ such that the $B\widehat\otimes \hat\Q_p^{\rm{nr}}$-module $((B\widehat\otimes\B_{\st})\otimes_BV)^{I_L}$ is free of rank $d$ and the map 
$$L\otimes_{L_0}((B\widehat\otimes\B_{\st})\otimes_BV)^{I_L}\rightarrow((B\widehat\otimes\B_{\dR})\otimes_BV)^{I_L}$$
is an isomorphism.
\end{thm}

We return to our general set-up.  Let $A$ be an $E$-affinoid algebra, and let $V$ be a finite free $A$-module equipped with a continuous $A$-linear action of $\Gal_K$.  We will embed $A$ isometrically into a finite product $B=\prod_iB_i$ of artin local rings $B_i$ and apply Theorem~\ref{monodromy} to $V\otimes_AB$.  

However, we need to be able to compare the semi-stability of $V$ (an $A$-linear representation of $\Gal_K$) and the semi-stability of $V_B=\prod_{i=1}^rV_{B_i}$ (as a $B$-linear representation of $\Gal_K$).

Recall that we have defined $\B_{\st}^+$ to be $\B_{\max}^+[\log([\overline\pi])]$, rather than the usual semi-stable period ring.  We further define $\B_{\st}^{+,h}:=\oplus_{i=0}^h\B_{\max}^+\log([\overline\pi])^i$, so that $\B_{\st}^{+,h}$ is the kernel of $N^{h+1}$ on $\B_{\st}^+$ (here $N$ is the monodromy operator).

Further, by Remark~\ref{bdr-frechet} that there is an isomorphism of $K$-Fr\'echet spaces $\B_{\dR}^+\xrightarrow{\sim}\C_K[\![T]\!]$ which defines compatible isomorphisms 
$$\B_{\max}^+\xrightarrow{\sim}\C_K\langle T\rangle\text{ and }\B_{\st}^+\xrightarrow{\sim}\C_K\langle T\rangle[\log(1+T)]$$ 
(as well as $\B_{\st}^{+,h}\xrightarrow{\sim}\oplus_{i=0}^h\C_K\langle T\rangle\log(1+T)^i$).

\begin{prop}\label{bdr-bst}
Let $A$ be an $E$-affinoid algebra, and let $x:A\rightarrow B$ be a closed embedding of Banach algebras.  Then if $a\in A\widehat\otimes\B_{\dR}^+$ and $x(a)\in B\widehat\otimes(L\otimes_{L_0}\B_{\st}^{+,h})$, $a$ is actually in $A\widehat\otimes(L\otimes_{L_0}\B_{\st}^{+,h})$.
\end{prop}

This follows as in~\cite[Lemme 6.3.1]{bc}.

Combined with the $p$-adic local monodromy theorem, we have the following:

\begin{thm}\label{pst}
Let $A$ be an $E$-affinoid algebra, $V$ an $A$-linear representation of $\Gal_K$ on a finite free $A$-module of rank $d$, and $[a,b]$ an interval such that $V_x$ is a de Rham representation of $\Gal_K$ with Hodge--Tate weights in $[a,b]$ for every $E$-finite artinian point $x$ of $A$. Then $V$ is potentially semi-stable: there is a finite Galois extension $L/K$ such that the $A\otimes_{\Q_p} L_0$-module $\D_{\st}^L(V)$ is locally free of rank $d$ and 
$$(A\otimes_{\Q_p} L)\otimes_{A\otimes_{\Q_p} L_0}\D_{\st}^L(V)=\D_{\dR}^L(V)$$
In addition, for any homomorphism of $E$-affinoid algebras $A\rightarrow A'$, the natural map $A'\otimes_A\D_{\st}^L(V)\rightarrow \D_{\st}^L(V\otimes_AA')$ is an isomorphism.
\end{thm}
When $A$ is reduced, this is~\cite[Th\'eor\`eme 6.3.2]{bc}.
\begin{proof}
We first apply Lemma~\ref{embed} to find a closed embedding $A\rightarrow \prod_i B_i$ into a finite product of artin rings.  Here $B_i$ is an $E_i$-finite algebra, where $E_i$ is a complete discretely valued field with perfect residue field of characteristic $p$ (and $B_i$ is topologized as a finite-dimensional $E_i$-vector space).

Now we can apply the $p$-adic monodromy theorem to the representations $V_{B_i}$, because each of them can be viewed as a finite-dimensional $E_i$-representation.  In other words, there is a finite extension $L/K$ such that for each $i$ the natural map
$$L\otimes_{L_0}((B_i\widehat\otimes\B_{\st})\otimes_{B_i}V_{B_i})^{I_L}\rightarrow ((B_i\widehat\otimes\B_{\dR})\otimes_{B_i}V_{R_i})^{I_L}$$
is an isomorphism.  Theorem~\ref{monodromy} only produces an isomorphism of $E_i\otimes_{\Q_p}L$-modules, but the natural map respects the $B_i$-linear structure on each side, and there is no kernel or cokernel (because it is an isomorphism of $E_i\otimes_{\Q_p}L$-modules), so it is actually an isomorphism of underlying $B_i\otimes_{\Q_p}L$-modules.

But we know by Theorem~\ref{pointwise-adm} applied to $\B_\ast=\B_{\dR}$ that $\D_{\dR}^L(V)$ is a locally free $A\otimes_{\Q_p} L$-module of rank $d$.  We have an injective map 
$$\D_{\dR}^L(V)\rightarrow ((B\widehat\otimes\B_{\dR})\otimes_AV)^{I_L}$$ 
with $B=\prod B_i$, and for $y\in \D_{\dR}^L(V)$, we can write $y=\sum_{j=1}^dy_j\otimes v_j$ with $y_j\in A\widehat\otimes\B_{\dR}$ and $\{v_j\}$ an $A$-basis for $V$.  By the isomorphism above, each of the $y_j$ lies in $B\widehat\otimes(L\otimes_{L_0}\B_{\st})$.  Then by Proposition~\ref{bdr-bst}, each lies in $A\widehat\otimes(L\otimes_{L_0}\B_{\st})$, so
$$\D_{\dR}^L(V)=(A\widehat\otimes(L\otimes_{L_0}\B_{\st})\otimes_AV)^{\Gal_L}=L\otimes_{L_0}\D_{\st}^L(V)$$
and hence $\D_{\st}^L(V)$ is locally free of rank $d$ over $A\otimes_{\Q_p} L_0$.  

For the last part, consider the natural map $A'\otimes_A\D_{\st}^L(V)\rightarrow \D_{\st}^L(V\otimes_A A')$.  We can extend scalars from $A'\otimes_{\Q_p}L_0$ to $A'\otimes_{\Q_p} L$ to get a map 
$$\begin{CD}
((A'\otimes_{\Q_p} L)\otimes_{A'\otimes_{\Q_p} L_0}\left(A'\otimes_A\D_{\st}^L(V)\right) 
@>>>
(A'\otimes_{\Q_p} L)\otimes_{A'\otimes_{\Q_p}L_0}\left(\D_{\st}^L(V\otimes_AA')\right)	\\
@|		@VVV	\\
A'\otimes_A\D_{\dR}^L(V)	@.	\D_{\dR}^L(V\otimes_AA')
\end{CD}$$
Since $A'\otimes_{\Q_p}L_0\rightarrow A'\otimes_{\Q_p}L$ is a faithfully flat extension and 
$$A'\otimes_A\D_{\dR}^L(V)\xrightarrow{\sim} \D_{\dR}^L(V\otimes_AA')$$
is an isomorphism, our map must have been an isomorphism to begin with.
\end{proof}

\subsubsection{Semi-stability and crystallinity}

We will use uniform potential semi-stability of families of de Rham representations to define quotients $A_{\dR}^{[a,b]}\twoheadrightarrow A_{\st}^{[a,b]}$ and $A_{\dR}^{[a,b]}\twoheadrightarrow A_{\cris}^{[a,b]}$ cutting out the semi-stable and crystalline loci in $\Sp(A)$, respectively.

Note that if $V$ becomes semi-stable over $L$, then we get a representation $\rho$ of the inertia group $I_{L/K}\subset \Gal_{L/K}$ on $\D_{\st}^L(V)$, a locally free $A\otimes_{\Q_p}L_0$-module of rank $d$.  If $A=\Q_p$, $\rho$ is trivial precisely when $V$ is semi-stable as a representation of $\Gal_K$.  This is because $L_0K/K$ is an unramified extension, and $\D_{\st}^{L_0K}(V)=L_0\otimes_{K_0}\D_{\st}^K(V)$.  Thus, although the Galois group $\Gal_{L/K}$ acts semi-linearly on $\D_{\st}^L(V)$ over $L_0$, only the $L_0$-linear action of $I_{L/K}$ matters for checking semi-stability.

\begin{lemma}\label{trace-artin}
Let $B$ be an $E$-finite artin local ring, with maximal ideal $\mathfrak{m}$ and residue field $k$, so that we may view $B$ as a $k$-algebra.  Let $V$ be a free $B$-module of rank $d$, equipped with a $B$-linear action of a finite group $G$.  Then $V$ is isomorphic as a representation of $G$ to $(V\otimes_Bk)\otimes_k B$.  In particular, $\Tr V=\Tr (V\otimes_Bk)$.
\end{lemma}
\begin{proof}
This follows from the fact that $\H^i(G,\mathfrak{m}\otimes_k\ad(V\otimes_Ek))=0$ for $i>0$, since $G$ is finite and the coefficients are a characteristic $0$ vector space.
\end{proof}

\begin{prop}\label{trace-constant}
Let $A$ be an $E$-affinoid algebra such that $\Sp(A)$ is connected.  Let $V$ be a free $A$-module of rank $d$, and let $G$ be a finite group.  Let $\rho:G\rightarrow \GL(V)$ be a representation of $G$ on $V$.  Then for any closed points $x_1$ and $x_2$ of $X$, there are finite \'etale extensions $i_1:k(x_1)\hookrightarrow E'$, $i_2:k(x_2)\hookrightarrow E'$ such that $\rho_{x_1}\otimes_{i_1}E'$ and $\rho_{x_2}\otimes_{i_2}E'$ are isomorphic as $E'$-valued representations of $G$.
\end{prop}
\begin{proof}
Fix an algebraic closure $\overline{E}/E$.  There are only finitely many isomorphism classes of $d$-dimensional representations of $G$ over $\overline{E}$, call them $\rho_1,\ldots,\rho_k$, so there is some subfield $F\subset \overline{\Q_p}$, finite over $\Q_p$, such that they are all defined over $F$.

Now consider $A':=A\otimes_{E}F$ and the representation $\rho_{A'}:G\rightarrow \GL(V_{A'})$.  The conditions $\Tr(\rho)=\Tr(\rho_i)$ each define pairwise disjoint closed subspaces of $X_F$ whose set-theoretic union is all of $X_F$.  Therefore they are all open, as well, and the function $\Tr(\rho)$ is constant on connected components of $X_F$.

Now let $X'$ be any connected component of $X_F$.  It is finite \'etale over $X$, and in particular, the map $X'\rightarrow X$ is surjective.  This gives the desired result.
\end{proof}

In the situation of interest to us, Lemma~\ref{trace-artin} and Proposition~\ref{trace-constant} have the following consequence:
\begin{thm}\label{pst-galois-type}
Let $V$ be a de Rham representation of $\Gal_K$ on a projective $A$-module of rank $d$, and let $L/K$ be the finite Galois extension provided by Theorem~\ref{pst}.  Let $\tau:I_{L/K}\rightarrow \GL_d(\overline{E})$ be a representation of the inertia group of $L/K$.  There is a quotient $A\twoheadrightarrow A_{\dR,\tau}$ such that an $E$-finite artinian point $x:A\rightarrow B$ factors through $A_{\dR,\tau}$ if and only if the representation of $I_{L/K}$ on $\D_{\st}^L(V_x)$ is equivalent to $\tau$.  In particular, there is a quotient $A\twoheadrightarrow A_{\st}$ corresponding to $\tau$ being the trivial representation.
\end{thm}
\begin{remark}
Since there are only finitely many isomorphism classes of representations $\tau:I_{L/K}\rightarrow \GL_d(\overline{E})$, $\Sp(A_{\dR,\tau})$ is a union of connected components of $\Sp(A)$.
\end{remark}

\begin{cor}
Let $A$ and $V$ be as above.  If for every $E$-finite artinian point $x:A\rightarrow B$, $V_x$ is semi-stable with Hodge--Tate weights in a fixed interval $[a,b]$, then $\D_{\st}^K(V)$ is a locally free $A\otimes_{\Q_p}K_0$-module of rank $d$.
\end{cor}
\begin{proof}
The assumptions imply that for every $E$-finite artinian point $x:A\rightarrow B$, $V_x$ is de Rham with Hodge--Tate weights in the interval $[a,b]$.  Then by Theorem~\ref{pointwise-adm} applied to $\B_\ast=\B_{\dR}$, $V$ is a de Rham representation.  

Let $L$ be the finite Galois extension provided by Theorem~\ref{pst}.  Then the assumption on artinian points implies that $A\twoheadrightarrow A_{\st}$ is the identity map.  For every closed point $x\in X$ and every $n\geq 0$, 
\[	\D_{\st}^L(V\otimes_AA/\mathfrak{m}_x^n)=L_0\otimes_{K_0}\D_{\st}^K(V\otimes_AA/\mathfrak{m}_x^n)	\]
and in particular, every element $v\in \D_{\st}^L(V)$ is fixed by $g\in I_{L/K}$ modulo all $\mathfrak{m}_x^n$.  But then $v$ is actually fixed by all $g\in I_{L/K}$, so $I_{L/K}$ acts trivially on $\D_{\st}^L(V)$.  We can then write $\D_{\st}^L(V)=(A\otimes_{\Q_p}L_0)\otimes_{A\otimes_{\Q_p}K_0}\D_{\st}^K(V)$, so $\D_{\st}^K(V)$ is a locally free $A\otimes_{\Q_p}K_0$-module of rank $d$.
\end{proof}

\begin{cor}\label{b-log-isom}
Let $A$ and $V$ be as above.  If for every $E$-finite artinian point $x:A\rightarrow B$, $V_x$ is semi-stable with Hodge--Tate weights in a fixed interval $[a,b]$, then the natural map $(A\widehat\otimes {\B}_{\log}^\dagger)[1/t]\otimes_{A\otimes K_0}\D_{\st}(V)\rightarrow (A\widehat\otimes\D_{\log}^\dagger)(V)[1/t]$ is an isomorphism.
\end{cor}
\begin{proof}
Since there is some $s_n$ such that $\D_{\st}(V)=\left(\D_{\log,K}^{\dagger,s_n}(V)[1/t]\right)^{\Gamma}$, we are reduced to showing that 
\[	(A\widehat\otimes {\B}_{\log}^{\dagger,s_n})[1/t]\otimes_{A\otimes K_0}\D_{\st}(V)\rightarrow \D_{\log,K}^{\dagger,s_n}(V)[1/t]	\]
is an isomorphism.

The left- and right-sides are coherent modules over $A\widehat\otimes\B_{\log,K}^{\dagger,s_n}[1/t]=A\widehat\otimes\B_{\rig,K}^{\dagger,s_n}[\log([\overline\pi]),1/t]$, and the homomorphism respects the grading given by powers of $\log([\overline\pi])$.  We are thus further reduced to considering homomorphisms of coherent modules over $A\widehat\otimes\B_{\rig,K}^{\dagger,s_n}[1/t]$, which is the ring of global sections of a quasi-Stein space.  The quasi-Stein space in question is the product of $\Sp(A)$ with a half-open annulus (associated to $\B_{\rig,K}^{\dagger,s_n}$) minus the divisor of $t$, which we denote $Y$.  

Since every point of $\Sp(A)\times Y$ sits over a point of $\Sp(A)$, to prove surjectivity of the desired map, it suffices to check on artinian thickenings of closed points of $\Sp(A)$.  But this holds by~\cite[Proposition 3.7]{berger}.  Furthermore, since $\D_{\log,K}^{\dagger,s_n}(V)[1/t]$ is finite projective, we may check injectivity on points, as well, and this again follows from~\cite[Proposition 3.7]{berger}.
\end{proof}

\begin{cor}\label{tilde-b-log-isom}
Let $A$ and $V$ be as above.  If for every $E$-finite artinian point $x:A\rightarrow B$, $V_x$ is semi-stable with Hodge--Tate weights in a fixed interval $[a,b]$, then the natural map $(A\widehat\otimes \widetilde{\B}_{\log}^\dagger)[1/t]\otimes_{A\otimes K_0}\D_{\st}(V)\rightarrow (A\widehat\otimes\widetilde{\B}_{\log}^\dagger)[1/t]\otimes_AV$ is an isomorphism.
\end{cor}
\begin{proof}
Since $(A\widehat\otimes\widetilde{\B}_{\log}^\dagger)[1/t]\otimes_{A\widehat\otimes\B_{\rig,K}^\dagger}\D_{\rig,K}^\dagger(V)\cong (A\widehat\otimes\widetilde{\B}_{\log}^\dagger)[1/t]\otimes_AV$, it suffices to show that the natural map
\[	(A\widehat\otimes {\B}_{\log}^\dagger)[1/t]\otimes_{A\otimes K_0}\D_{\st}(V)\rightarrow \D_{\log,K}^\dagger(V)[1/t]	\]
is an isomorphism, and follows from Corollary~\ref{b-log-isom}
\end{proof}

\begin{cor}
Let $A$ and $V$ be as above.  If for every $E$-finite artinian point $x:A\rightarrow B$, $V_x$ is semi-stable with Hodge--Tate weights in a fixed interval $[a,b]$, then the natural map $(A\widehat\otimes\B_{\st})\otimes_{A\otimes K_0}\D_{\st}(V)\rightarrow (A\widehat\otimes\B_{\st})\otimes_AV$ is an isomorphism.
\end{cor}
\begin{proof}
We first show that the natural morphism
\[	(A\widehat\otimes\B_{\log}^+)[1/t]\otimes_{A\otimes K_0}\D_{\st}(V)\rightarrow (A\widehat\otimes\B_{\log}^+)[1/t]\otimes_AV	\]
is an isomorphism.
Injectivity follows from Corollary~\ref{tilde-b-log-isom}, since $A\widehat\otimes\B_{\log}^+\rightarrow A\widehat\otimes\widetilde{\B}_{\log}^\dagger$ is injective and $\D_{\st}(V)$ is flat.  For surjectivity, we choose bases $\{v_i\}$ and $\{w_j\}$ of $V$ and $\D_{\st}(V)$, respectively (since $\D_{\st}(V)$ is $A$-locally free, we may assume that both $V$ and $\D_{\st}(V)$ are free).  Let $M\in\Mat_{d\times d}(A\widehat\otimes\widetilde{\B}_{\log}^\dagger)[1/t])$ be the matrix whose $j$th column is the coordinates of $w_j$ with respect to $\{v_i\}$, and let $P\in \GL(A\otimes_{\Q_p}K_0)$ be the matrix of Frobenius on $\D_{\st}(V)$ with respect to $\{w_j\}$.  As in the proof of Proposition~\ref{dlog-dst-field}, $M$ and $P$ satisfy the relation $MP=\varphi(M)$, so by Corollary~\ref{gl-frob-reg}, $M\in\GL_d(\A\widehat\otimes\widetilde\B_{\log}^+[1/t])$.  

Finally, since $\B_{\log}^+[1/t]\subset\B_{\st}$, we may extend scalars to see that that
\[	(A\widehat\otimes\B_{\st})\otimes_{A\otimes K_0}\D_{\st}(V)\rightarrow (A\widehat\otimes\B_{\st})\otimes_AV	\]
is an isomorphism.
\end{proof}

\begin{remark}
A similar isomorphism is proved in~\cite{hartl-hellmann}, using a different sheaf of semi-stable period rings.
\end{remark}

\begin{cor}
Suppose $V$ is semi-stable with Hodge--Tate weights in the interval $[a,b]$, and let $f:A\rightarrow A'$ be a homomorphism of $E$-affinoid algebras.  Then $V\otimes_AA'$ is semi-stable with Hodge--Tate weights in the interval $[a,b]$.
\end{cor}

To handle crystalline representations, we note that $\D_{\cris}(V)=\D_{\st}(V)^{N=0}$.  Then the results below follow easily.

\begin{thm}
Let $A$ be an $E$-affinoid algebra and let $V$ be a free $A$-module of rank $d$ equipped with a continuous $A$-linear action of $\Gal_K$.  Let $[a,b]$ be a finite interval.  Then there is a quotient $A\twoheadrightarrow A_{\cris}^{[a,b]}$ such that an $E$-finite artin point $x:A\rightarrow B$ factors through $A_{\cris}^{[a,b]}$ if and only if $V_x$ is crystalline, with Hodge--Tate weights in the interval $[a,b]$.
\end{thm}

\begin{thm}
Let $A$ be an $E$-affinoid algebra and let $V$ be a free $A$-module of rank $d$ equipped with a continuous $A$-linear action of $\Gal_K$.  Suppose that for every $E$-finite artinian point $x:A\rightarrow B$ the representation $V_x$ is crystalline with Hodge--Tate weights in an interval $[a,b]$.  Then $\D_{\cris}(V)$ is a locally free $A\otimes_{\Q_p}K_0$-module of rank $d$, the formation of $\D_{\cris}(V)$ commutes with base change on $A$, and the natural homomorphism
\[	(A\widehat\otimes\B_{\max})\otimes_{A\otimes K_0}\D_{\st}(V)\rightarrow (A\widehat\otimes\B_{\max})\otimes_AV	\]
is an isomorphism.
\end{thm}

\appendix

\section{Rings of $p$-adic Hodge theory}\label{rings}

Most of the definitions and properties of the rings we use are given in~\cite{berger}, and we refer to it freely.  However, we describe the construction and topologies of the period rings $\B_{\HT}$, $\B_{\dR}$, $\B_{\st}$, and $\B_{\max}$ with some care so that we can define sheaves of period rings in~\ref{sheafy-rings}.

\subsection{Period rings}

Let $V$ be a finite-dimensional $\Q_p$-vector space equipped with a continuous $\Q_p$-linear action of $\Gal_K$, for some finite extension $K/\Q_p$.  Then for any period ring $\B_\ast$ listed above, we define 
$\D_{\B_\ast}^K(V):=\left(\B_\ast\otimes_{\Q_p}V\right)^{\Gal_K}$.  
For every choice of $\B_\ast$, $\B_\ast^{\Gal_K}$ is a field; we say that $V$ is \emph{$\B_\ast$-admissible} if $\dim_{\B_\ast^{\Gal_K}}= \dim_{\Q_p}V$.

\begin{remark}
There is a general formalism of period rings developed in~\cite[\textsection 1]{fontaine2}.  However, because of issues related to the topologies on various rings, it is not clear to us that this formalism generalizes in any meaningful way to the study of arithmetic families of Galois representations.  Thus, we content ourselves with giving a list of period rings of interest to us.
\end{remark}

\subsubsection*{$\B_{\HT}$}

We define $\B_{\HT}$ to be the polynomial ring $\C_p[t,t^{-1}]$.  This ring is graded by powers of $t$.  For any $K/\Q_p$, the Galois group $\Gal_K$ acts on $\B_{\HT}$ via the natural action on $\C_p$ and via $g\cdot t=\chi(g)t$.  


\subsubsection*{$\B_{\dR}$}

The construction of $\B_{\dR}$ is more complicated.  Recall the existence of a Galois-equivariant map $\theta:\widetilde\B^+\rightarrow \C_K$ characterized by $\theta(\sum[c_n]p^n) = \sum c_n^{(0)}p^n$.  It is continuous with respect to the weak topology on $\widetilde\B^+$ and the $p$-adic topology on $\C_K$, and its kernel is the principal ideal generated by $[\widetilde p]-p$.  Then $\B_{\dR}^+$ is by definition $\varprojlim \widetilde\B^+/\ker(\theta)^h$.  

We are grateful to Laurent Berger for providing the following definition of the topology on $\B_{\dR}^+$.  Since $\theta$ is Galois-equivariant, the Galois action on $\widetilde\B^+$ induces a Galois action on $\B_{\dR}^+$.  We want to topologize the quotients $\widetilde\B_h:=\widetilde\B^+/\ker(\theta)^h$ so that this action is continuous.  We could make $\widetilde\B_h$ into a $p$-adic Banach space with unit ball $\A_h:=\widetilde\A^+/\ker(\theta)^h\cap\widetilde\A^+$, but then the action of Galois would not be obviously continuous, so instead we try to use the weak topology, i.e.,  the topology on the image of $\widetilde\A^+$ generated by the images of the $U_{k,n}$.  For $n\geq h$, though,
$$[\widetilde p]^n = \left(([\widetilde p]-p)+p\right)^n = \left(([\widetilde p]-p)+p\right)^h\left(([\widetilde p]-p)+p\right)^{n-h}\in p\A_h$$
In particular, this shows that $p\widetilde\A_h$ is an open ideal.

A priori the $p$-adic topology on $\widetilde\A_h$ has more open sets than the weak topology does.  But we have just shown that every open set of the $p$-adic topology is actually open in the weak topology, so the two topologies must be the same.  In particular, the weak topology on $\widetilde\A_h$ is Hausdorff and complete.

The upshot is that $\widetilde\B_h=\cup_{i\geq0}p^{-i}\widetilde\A_h$ has a natural structure of a $p$-adic Banach space with unit ball $\widetilde\A_h$, so $\B_{\dR}^+$ has a natural structure of a $p$-adic Fr\'echet space.

\subsubsection*{$\B_{\max}$}

We will use the ring $\B_{\max}$ to study crystalline representations, rather than $\B_{\cris}$, because the topology on $\B_{\max}$ is much nicer.  



\begin{remark}
There is a closely related ring $\widetilde\B_{\rig}^+:=\cap_n\varphi^n(\B_{\max}^+)=\cap_n\varphi^n(\B_{\cris})$; $\B_{\max}$ and $\widetilde\B_{\rig}^+[1/t]$ define the same functor from the category of $\Q_p$-representations to the category of filtered $K_0$-isocrystals.  However, we prefer to work with $\B_{\max}^+$ because it is a Banach space, while $\widetilde\B_{\rig}^+$ is a Fr\'echet space.
\end{remark}

\subsubsection*{$\B_{\st}$}

After choosing a value for $\log(p)$, the power series defining $\log(\overline\pi^{(0)})+\log([\overline\pi]/\overline\pi^{(0)})$ converges in $\B_{\dR}^+$.
We let $\B_{\st}:=\B_{\max}[\log[\overline\pi]]$.  We let $\log(p)=0$.


\begin{remark}
This is not the standard definition of $\B_{\st}$, but it is the usage of~\cite{berger} and~\cite{bc}.  It defines the same functor on representations of $\Gal_K$ as the usual $\B_{\st}$, because $\B_{\cris}$ and $\B_{\max}$ define the same functor.  If we define $\widetilde{\B}_{\log}:=\widetilde{\B}_{\rig}[\log[\overline\pi]]$, then $\widetilde{\B}_{\log}$ defines the same functor, as well.
\end{remark}

\begin{remark}
We have defined $\B_{\st}$ as a subring of $\B_{\dR}$ in terms of a choice of a branch of the $p$-adic logarithm.  A different choice would lead to a different subring.  It is also possible to define $\B_{\st}$ intrinsically as an abstract ring and use the choice of a $p$-adic logarithm to define an embedding of $\B_{\st}$ in $\B_{\dR}$.  This approach makes clear that $\B_{\st}$-admissibility of a representation does not depend on any choices.  We have not taken this approach here; for details about the construction of the usual $\B_{\st}$ as an extension of $\B_{\cris}$, see~\cite[\textsection 9.2]{brinon-conrad}.
\end{remark}

\begin{remark}\label{bdr-frechet}
Let $L/K$ be a finite extension.  Then we obtain a map 
$$L\otimes_{L_0}\B_{\max}^+\rightarrow \B_{\dR}^+$$ 
by extending the inclusion $\B_{\max}^+\hookrightarrow \B_{\dR}^+$ by $L$-linearity.  Then Colmez has shown in~\cite[Proposition 7.14]{colmez2} that this map is an injection.  In the course of the proof, he showed that it is possible to write down an isomorphism of $K$-Fr\'echet spaces $\B_{\dR}^+\cong \C_K[\![t]\!]$ so that $L\otimes_{L_0}\B_{\max}^+$ is carried isomorphically to the Banach space $\C_K\langle T\rangle$ and $L\otimes_{L_0}\B_{\st}^+$ is carried isomorphically to $\C_K\langle T\rangle[\log(1+T)]$.    Note that these are isomorphisms as vector spaces, not as rings!
\end{remark}

\subsection{Sheaves of period rings}\label{sheafy-rings}

As we wish to study $p$-adic families of Galois representations, we need to define versions of these rings with ``coefficients'' in Banach algebras, rather than simply $\Q_p$.

Let $(X,\mathscr{O}_X)$ be a quasi-compact quasi-separated rigid space over a finite extension $E/\Q_p$.

\begin{definition}
Let $B$ be a $\Q_p$-Banach algebra.  Then we define the presheaf $\mathscr{B}_X$ on $X$ by setting
$$\mathscr{B}_X(U):=\mathscr{O}_X(U)\widehat\otimes_{\Q_p}B$$
where $U$ is an admissible affinoid open of $X$.
\end{definition}

By~\cite[Lemma 3.3]{kl}, $\mathscr{B}_X$ is actually a sheaf on $U$ when $U$ is affinoid.  Therefore, it extends to a sheaf on $X$.  We wish to extend this to Fr\'echet algebras in the role of $B$.

We first record some basic functional analysis results, which will be useful for checking that exactness properties are preserved under completed tensor products.

Let $I$ be a set (not necessarily countable), and let $N$ be a Fr\'echet space equipped with a countable family of seminorms $\{q_j\}$ (in particular, $N$ could be a Banach space).  We define the space $c_I(N)$ to be the set of functions $f:I\rightarrow N$ such that $\lim_{i\in I} q_j(f(i))=0$ for each seminorm $q_j$.  That is, for each $j$ and each $\varepsilon>0$, the set $\{i\in I| q_j(f(i))>\varepsilon\}$ is finite.  We equip $c_I(N)$ with the seminorms $q_{j,\infty}$ defined by $q_{j,\infty}(f):=\sup_{i\in I}q_j(f(i))$, making $c_I(N)$ into a Fr\'echet space.  Following \cite{buzzard}, we say that a Banach module $N$ over a Banach algebra $A$ is \emph{potentially orthonormalizable} if there is some Banach norm on $N$ making it is isomorphic to $c_I(A)$ for some index set $I$.  For example, all Banach spaces over discretely valued fields are potentially orthonormalizable~\cite[Proposition 10.1]{schneider}.

\begin{lemma}
Let $k$ be a non-archimedean field, let $M$ be a potentially orthonormalizable $k$-Banach space, and let $N$ be a $k$-Fr\'echet space, with countable family of seminorms $\{q_j\}$.  Write $M\cong c_I(k)$.  Then the natural $\Q_p$-linear map $M\widehat\otimes_k N\rightarrow c_I(N)$ is an isomorphism, functorially in $N$.
\end{lemma}
\begin{proof}
The natural map $M\widehat\otimes N\rightarrow c_I(N)$ is induced by the bilinear map $M\times N\rightarrow c_I(N)$ sending $(f:I\rightarrow k,b)$ to $\sum_{i\in I}f(i)b$.  The sum converges because $q_j(f(i)b)\leq |f(i)|\cdot q_j(b)$ and $\lim_{i\in I}a_i=0$.  

To construct a map in the other direction, we observe that any element $f\in c_I(N)$ can be written as the limit of elements of the form $f|_S$, where $S\subset I$ is a finite subset.  More precisely, the set of finite subsets $S\subset I$ is a directed set, and $S\mapsto f|_S$ is a net converging to $f$.  For any finite set $S\subset I$, we write $\mathbf(1)_S\in c_I(k)$ for the characteristic function of $S$.  Now consider $\sum_{i\in S}\mathbf{1}_{\{i\}}\otimes f(i)\in c_I(k)\otimes_k N$.  We have 
\[	q_j\left(\sum_{i\in S}\mathbf{1}_{\{i\}}\otimes f(i)\right) \leq \max_{i\in S}q_j(f(i))	\]
so the net $S\mapsto\sum_{i\in S}\mathbf{1}_{\{i\}}\otimes f(i)$ converges in $c_I(k)\widehat\otimes N$.  The map $f\mapsto \lim_S\sum_{i\in S}\mathbf{1}_{\{i\}}\otimes f(i)$ provides an inverse to the map $M\widehat\otimes_k N\rightarrow c_I(N)$.
\end{proof}

\begin{cor}
Let $k$ and $M$ be as above, and let $N\rightarrow N'$ be a continuous injection of $k$-Fr\'echet spaces.  Then the natural map $M\widehat\otimes N\rightarrow M\widehat\otimes N'$ is injective.
\end{cor}

\begin{definition}
Let $B=\varprojlim_nB_n$ be a $\Q_p$-Fr\'echet algebra, where the $B_n$ are $\Q_p$-Banach algebras.  Then we define the presheaf $\mathscr{B}_X$ on $X$ by setting
$$\mathscr{B}_X(U):=\mathscr{O}_X(U)\widehat\otimes V = \varprojlim_n\mathscr{O}_X(U)\widehat\otimes_{\Q_p}B_n$$
when $U$ is an admissible affinoid open of $X$.
\end{definition}

\begin{lemma}
$\mathscr{B}_X$ is a sheaf on $X$.
\end{lemma}
\begin{proof}
It suffices to prove this when $X=\Sp(A)$ is affinoid, for $A$ some $E$-affinoid algebra.  Further, it suffices to check the sheaf property on Laurent coverings of $\Sp(A)$.  That is, we need to check that the sequence
\[	0\rightarrow \varprojlim_n A\widehat\otimes_{\Q_p}B_n\rightarrow \varprojlim_n A\langle f\rangle\widehat\otimes_{\Q_p}B_n\times \varprojlim_n A\langle f^{-1}\rangle\widehat\otimes_{\Q_p}B_n\rightarrow \varprojlim_n A\langle f,f^{-1}\rangle\widehat\otimes_{\Q_p}B_n	\]
is exact.  But 
\[	0\rightarrow A\rightarrow A\langle f\rangle\times A\langle f^{-1}\rangle\rightarrow A\langle f,f^{-1}\rangle\rightarrow 0	\]
is exact, and the quotient admits a section by~\cite[Proposition 10.5]{schneider}, since $\Q_p$-affinoid algebras are countable type over $\Q_p$.  It follows that 
$$0\rightarrow A\widehat\otimes_{\Q_p}B_n\rightarrow A\langle f\rangle\widehat\otimes_{\Q_p}B_n\times A\langle f^{-1}\rangle\widehat\otimes_{\Q_p}B_n\rightarrow A\langle f,f^{-1}\rangle\widehat\otimes_{\Q_p}B_n\rightarrow 0$$
is exact for each $n$, and inverse limits are left-exact.
\end{proof}

Thus, taking $B$ to be $\widetilde\B_K^{\dagger,s}$, $\B_K^{\dagger,s}$, $\widetilde\B_{\rig,K}^{\dagger,s}$, $\B_{\rig,K}^{\dagger,s}$, $\C_K$, $\B_{\dR}^+$, or $\B_{\max}^+$, we get a sheaf of rings.  Furthermore, since taking rising unions is exact, we see that if $U\subset X=\Sp(A)$ is an affinoid subdomain with coordinate ring $A_U$,
\[	U\mapsto \cup_s A_U\widehat\otimes \B_{\rig,K}^{\dagger,s}	\]
is a sheaf on $X$.  Similarly, we also get sheaves associated to $\B_{\HT} = \C_K[t,t^{-1}]$, $\B_{\dR} = \cup_it^{-i}\B_{\dR}^+$, $\B_{\max} = \cup_it^{-i}\B_{\max}^+$, and $\B_{\st} = \B_{\max}[\log[\overline\pi]]$.  Each of these sheaves carries the additional structures, such as Galois action, Frobenius action, grading, filtration, or monodromy action, of the absolute ring.

\begin{prop}
\begin{enumerate}
\item	$\mathscr{B}_{X,\HT}$ is a graded sheaf of rings over $X$, equipped with an action of $\Gal_K$, and $\mathscr{B}_{X,HT}^{\Gal_K} = \mathscr{O}_X\otimes_{\Q_p}K$
\item	$\mathscr{B}_{X,\dR}$ is a filtered sheaf of rings over $X$, equipped with an action of $\Gal_K$, and $\mathscr{B}_{X,\dR}^{\Gal_K} = \mathscr{O}_X\otimes_{\Q_p}K$
\item	$\mathscr{B}_{X,\max}$ is a sheaf of rings over $X$, equipped with an action of $\Gal_K$ and an action of $\varphi$, and $\mathscr{B}_{X,\max}^{\Gal_K} = \mathscr{O}_X\otimes_{\Q_p}K_0$
\item	$\mathscr{B}_{X,\st}$ is a sheaf of rings over $X$, equipped with an action of $\Gal_K$ and an action of $\varphi$ and $N$, and $\mathscr{B}_{X,\st}^{\Gal_K} = \mathscr{O}_X\otimes_{\Q_p}K_0$
\end{enumerate}
\end{prop}
\begin{proof}
For all of these, it suffices to consider the case when $X=\Sp(A)$ is affinoid for some $E$-affinoid algebra $A$.  Then $A$ is countable type over $\Q_p$, so we can choose a Schauder basis for $A$ with index set $I$.  Under the resulting isomorphism $A\cong c_I(\Q_p)$, $A\widehat\otimes B\cong c_I(B)$ for a $\Q_p$-Fr\'echet space $B$.  The Galois action on $A\widehat\otimes B$ is $g\cdot (b_i)_{i\in I} = (g\cdot b_i)_{i\in I}$, where $I$ is the index set for the Schauder basis, so the assertions follow from the corresponding classical results.
\end{proof}




\begin{lemma}\label{rings-pts-inj}
Let $A$ be a reduced $\Q_p$-affinoid algebra, and let $\B_\ast$ be a period ring.  Then the natural map
\[	A\widehat\otimes\B_\ast\rightarrow \prod_{x\in\Sp(A)}A/\mathfrak{m}_x\otimes_{\Q_p}\B_\ast	\]
is injective.
\end{lemma}
\begin{proof}
If $\B_\ast=\B_{\HT}$, it suffices to show that $A\widehat\otimes\C_p\rightarrow \prod_{x\in\Sp(A)}A/\mathfrak{m}_x\otimes_{\Q_p}\C_p$ is injective.  If $\B_\ast=\B_{\dR}$, it suffices to show that $A\widehat\otimes\widetilde\B_h\rightarrow \prod_{x\in\Sp(A)}A/\mathfrak{m}_x\otimes_{\Q_p}\widetilde\B_h$ is injective for all $h$.  And since $\B_{\st}$ is a polynomial algebra over $\B_{\max}$, it suffices to show that $A\widehat\otimes\widetilde\B_{\max}\rightarrow \prod_{x\in\Sp(A)}A/\mathfrak{m}_x\otimes_{\Q_p}\widetilde\B_{\max}$ is injective.  We are therefore reduced to showing that for any $\Q_p$-Banach space $B$, the natural map $A\widehat\otimes B\rightarrow \prod_{x\in\Sp(A)}A/\mathfrak{m}_x\otimes_{\Q_p}B$ is injective.

Since $B$ is a $\Q_p$-Banach space, it is potentially orthonormalizable in the sense of \cite{buzzard}.  That is, it admits a basis $\{e_i\}_{i\in I}$ such that 
\[	B\cong c_I(\Q_p):=\{f:I\rightarrow \Q_p| \lvert f(i)\rvert<\varepsilon\text{ for almost all }i\in I \text{ for all }\varepsilon>0\}	\]
Then $A\widehat\otimes B\cong c_I(A)$, and the desired injectivity follows from the injectivity of the natural map $A\rightarrow \prod_{x\in \Sp(A)}A/\mathfrak{m}_x$.
\end{proof}




    \bibliographystyle{alpha}
    \bibliography{thesis}

\end{document}